\documentclass{amsart}
\usepackage{graphicx} 
\usepackage{epsfig}
\usepackage[pagewise]{lineno}%\linenumbers
\usepackage{mathrsfs, comment}
\usepackage{hyperref}
\usepackage{bbm,bm}
\usepackage[normalem]{ulem}
\usepackage{cite}


\makeatletter
\def\@tocline#1#2#3#4#5#6#7{\relax
  \ifnum #1>\c@tocdepth % then omit
  \else
    \par \addpenalty\@secpenalty\addvspace{#2}%
    \begingroup \hyphenpenalty\@M
    \@ifempty{#4}{%
      \@tempdima\csname r@tocindent\number#1\endcsname\relax
    }{%
      \@tempdima#4\relax
    }%
    \parindent\z@ \leftskip#3\relax \advance\leftskip\@tempdima\relax
    \rightskip\@pnumwidth plus4em \parfillskip-\@pnumwidth
    #5\leavevmode\hskip-\@tempdima
      \ifcase #1
       \or\or \hskip 1em \or \hskip 2em \else \hskip 3em \fi%
      #6\nobreak\relax
    \hfill\hbox to\@pnumwidth{\@tocpagenum{#7}}\par% <---- \dotfill -> \hfill
    \nobreak
    \endgroup
  \fi}
\makeatother

\usepackage{amsmath}
\usepackage{mathrsfs}

\usepackage{mathtools}

\DeclarePairedDelimiter{\floor}{\lfloor}{\rfloor}

\let\text=\mbox

\catcode`\@=11

\newcommand{\z}{\zeta}
\newcommand{\ov}[1]{\overline{#1}}

\newcommand{\I}{\mathbbm{i}}

\newcommand{\iy}{\infty}

\newcount\br@j
\newcommand{\udesno}[1]{\unskip\nobreak\hfil\penalty50\hskip1em\hbox{}
             \nobreak\hfil{#1\unskip\ignorespaces}
                 \parfillskip=\z@ \finalhyphendemerits=\z@\par
                 \parfillskip=0pt plus 1fil}
\catcode`\@=11

\newcommand{\eR}{\mathbb{R}}
\newcommand{\eN}{\mathbb{N}}
\newcommand{\Ze}{\mathbb{Z}}

\newcommand{\Ce}{\mathbb{C}}

\newcommand{\re}{\mathop{\mathrm{Re}}}
\newcommand{\im}{\mathop{\mathrm{Im}}}

\newcommand{\res}{\operatorname{res}}

\newcommand{\ovb}[1]{\mkern 1.5mu\overline{\mkern-1.5mu#1\mkern-1.5mu}\mkern 1.5mu}

\newcommand{\E}{\mathrm{e}}
\newcommand{\di}{\,\mathrm{d}}

\newcommand{\sideremark}[1]{\ifvmode\leavevmode\fi\vadjust{\vbox to0pt{\vss % the remark
      \hbox to 0pt{\hskip\hsize\hskip1em           %                will appear only
 \vbox{\hsize2cm\tiny\raggedright\pretolerance10000%                on the side
 \noindent #1\hfill}\hss}\vbox to8pt{\vfil}\vss}}}%
\newcommand{\edz}[1]{\sideremark{#1}}

\newcommand{\eps}{\varepsilon}	 % Epsilon

\address[mardesic@u-bourgogne.fr]{Pavao Marde\v si\'c, Institut de Math\'ematiques de Bourgogne, UMR 5584 du CNRS,  UFR Sciences et Techniques, Universit\'e de Bourgogne Franche-Comt\'e, B.P. 47870, 21078 Dijon, France  and  University of Zagreb, Faculty of Science, Horvatovac 102a, 10000 Zagreb, Croatia}

\address[goran.radunovic@math.hr]{Goran Radunovi\'c, University of Zagreb, Faculty of Science, Horvatovac 102a, 10000 Zagreb, Croatia}

\address[maja.resman@math.hr]{Maja Resman, University of Zagreb, Faculty of Science, Horvatovac 102a, 10000 Zagreb, Croatia}

\newtheorem{theorem}{Theorem}[section]
\newtheorem*{thmA}{Theorem A}
\newtheorem*{thmB}{Theorem B}
\newtheorem*{thmC}{Theorem C}
\newtheorem{corollary}[theorem]{Corollary}
\newtheorem{lemma}[theorem]{Lemma}
\newtheorem{proposition}[theorem]{Proposition}

\theoremstyle{definition}
\newtheorem{definition}[theorem]{Definition}

\newtheorem{remark}[theorem]{Remark}

\numberwithin{equation}{section}

\title[Fractal zeta functions of parabolic orbits]{Fractal zeta functions\\ of orbits of parabolic diffeomorphisms}

\author[P.\ Marde\v si\'c, G.\ Radunovi\'c and M.\ Resman]{Pavao Marde\v si\'c, Goran Radunovi\'c and Maja Resman}

\thanks{The research of Goran Radunovi\'c and Maja Resman was supported by the Croatian Science Foundation under the grant UIP-2017-05-1020. The research of all authors was partially supported by Croatian Science Foundation (HRZZ) grant PZS-2019-02-3055 from
Research Cooperability funded by the European Social Fund and from the bilateral Hubert-Curien Cogito grant 2021-22. The research of Pavao Marde\v si\' c was partially supported by  
EIPHI Graduate School (contract ANR-17-EURE-0002). }

\begin{document}

\maketitle
\begin{abstract}
In this paper, we prove that fractal zeta functions of orbits of parabolic germs of diffeomorphisms can be meromorphically extended to the whole complex plane. We describe their set of poles (i.e. their \emph{complex dimensions}) and their principal parts which can be understood as their {\it fractal footprint}. We study the fractal footprint of one orbit of a parabolic germ $f$ and extract intrinsic information about the germ $f$ from it, in particular, its formal class.

Moreover, we relate complex dimensions to the generalized asymptotic expansion of the tube function of orbits with oscillatory 'coefficients' as well as to the asymptotic expansion of their dynamically regularized tube function.

Interestingly, parabolic orbits provide a first example of sets that have nontrivial Minkowski (or box) dimension and their tube function possesses higher order oscillatory terms, however, they do not posses non-real complex dimensions and are therefore not called fractal in the sense of Lapidus.
\end{abstract}

\tableofcontents\

%\begin{classification}
%Primary: 11M41, 28A12, 28A75, 28A80, 28B15, 42B20, 44A05.
%Secondary: 35P20, 40A10, 42B35, 44A10, 45Q05.
%\end{classification}

%\begin{keywords}
%Mellin transform, complex dimensions of a relative fractal drum, relative fractal drum, fractal set, box dimension, fractal zeta functions, distance zeta function, tube zeta function, fractal string, Minkowski content, Minkowski measurability criterion, Minkowski measurable set, residue, meromorphic extension, gauge-Minkowski measurability, singularities of fractal zeta functions.
%\end{keywords}

\section{Introduction}

This paper is motivated by the famous question, originally posed by Mark Kac and then extended to its non-smooth version by H.\ Weyl, M.\ Berry and others: 
\emph{Can we hear the shape of a fractal drum?} \cite{kac, LapPom1,LapPom2,LapPom3}

In our case the problem can be formulated in a similar fashion:
\emph{Can we see a diffeomorphism?} 

More precisely, let $f$
be an attracting germ of a diffeomorphism on the real line $\mathbb{R}$ at a fixed point $0$. Let $x_0$ be a point in the basin of attraction of $0$ and let   
$$
\mathcal O_f(x_0):=\{f^{\circ n}(x_0):\,n\in\mathbb N\},
$$  
be its orbit by $f$.

The orbit $\mathcal{O}_f(x_0)$ is the geometric object 
that we are looking at.  We are looking at it with two different pairs of glasses: the \emph{tube function} $V_{f,x_0}$ and the \emph{distance zeta function} $\zeta_{f,x_0}$.

\medskip
We study the asymptotic expansion of the tube function $V_{f,x_0}$ for a certain class of diffeomorphisms,  and show that the distance zeta function $\zeta_{f,x_0}$ extends to a meromorphic function and that its poles are related to the exponents in the asymptotic expansion of the tube function $V_{f,x_0}$. 
We will show that in some sense these two functions will determine the diffeomorphism, more specifically, its formal class.

\medskip

  Let us introduce the above functions.
   The \emph{tube function}
\begin{equation*}
	\label{tube_f_def}
	V_{f}:=V_{f,x_0}\colon\varepsilon\mapsto |\mathcal O_f(x_0)_\varepsilon\cap [0,x_0]|
\end{equation*}
is defined for $\varepsilon>0$ and given as the Lebesgue measure (here: length) of the Euclidean $\varepsilon$-neighborhood of the orbit $\mathcal O_f(x_0)$ intersected by $[0,x_0]$. For convenience, we cut off the left- and right-end intervals of length $\varepsilon$. This is usually referred to as the \emph{inner} tube function in contrast to the \emph{complete} tube function where no cut-offs are made, see \cite{lapidusfrank12}.
The value $\varepsilon$ can be thought of as the \emph{size} or \emph{resolution} of a point $f^{\circ{n}}(x_0)$ as represented on a computer screen. 
\begin{comment}
What we see is the \emph{density of the orbit} as represented on the screen for different resolutions corresponding to different values of $\varepsilon$.
\end{comment} 
 We will be interested in the asymptotic expansion of the tube function $V_f(\varepsilon)$ as $\varepsilon\to0^+$. Note that, if $V_f(\varepsilon)\sim M\varepsilon^{1-d},\ 0< d\leq 1,$ as $\varepsilon\to 0^+$, then $d$ is the \emph{Minkowski $($box$)$ dimension} of the orbit, and $M>0$ its \emph{Minkowski content}. For details, see \cite{falc}.

The \emph{distance zeta function} $\zeta_{A}$ of a bounded set $A\subseteq\eR^N$ is defined by 
\begin{equation}\label{eq:dz}
\zeta_{A}(s):=\int_{A_\delta} d(x,A)^{s-N}\,dx,
\end{equation}
where $A_\delta$ is the Euclidean $\delta$-neighborhood of the set $A$, for some fixed $\delta>0$, and $d$ denotes the Euclidean distance in $\mathbb R^N$.
By \cite{fzf,dtzf}, this function is analytic on the right half-plane $\{\re s>\ov{\dim}_B A\}$ where $\ov{\dim}_B A$ is the upper Minkowski (or box) dimension of the set $A$.
Furthermore, under relatively mild hypotheses, $\zeta_A$ can be meromorphically continued to a larger set. In that case, at $s=\dim_B A$, it has a pole with a residue equal to the (average) Minkowski content of $A$ (modulo a multiplicative constant). 
Note that in formula \eqref{eq:dz}, $\zeta_A$ depends on the choice of $\delta>0$. 
Usually, one is only interested in poles and residues of fractal zeta functions and, since changing $\delta>0$ in \eqref{eq:dz} results in adding an entire function, the dependence on $\delta$ is not essential and therefore often not indicated at all; see \cite[\S2.1.5, esp., Proposition 2.1.76]{fzf} or \cite[Proposition 2.22]{dtzf}.

Inspired by \eqref{eq:dz}, we define the \emph{distance zeta function of the orbit $\mathcal{O}_f(x_0)$} of a real-line germ $f$ by 
\begin{equation}\label{zeta_f}
	\zeta_f(s)=\zeta_{f,x_0}(s):=\int_0^{x_0} d(x,\mathcal{O}_f(x_0))^{s-1}\,dx,
\end{equation}
where we usually omit $x_0$ from the notation for simplicity.
It is a slight adaptation of definition \eqref{eq:dz} for $N=1$ and $A:=\mathcal{O}_f(x_0)$, where we also first fix $\delta>0$ large enough and then discard an unimportant part of $\zeta_A(s)$; for exact details see Definition \ref{distance_zeta_germ}. We would like to emphasize that this approach is very different from the classical theory of dynamical zeta functions studied in, e.g.,  \cite{Rue1,Rue2,Rue3,GLP,Poll}.

The \emph{complex dimensions} of a set $A$, which generalize the notion of box dimension, are defined as the set of poles of the distance zeta function $\zeta_{A}$ if its meromorphic extension exists. If $\zeta_{A}$ is meromorphic in all of $\Ce$, we denote the set of all complex dimensions of $A$ by $\dim_{\mathbb{C}}A$.
Initially, complex dimensions were introduced by Lapidus and van Frankenhuijsen in \cite{lapidusfrank12} in the case of subsets of the real line and then generalized to the general case of subsets of the arbitrary-dimensional Euclidean space in \cite{fzf,dtzf}. 
Furthermore, under appropriate hypotheses, they are closely connected to the geometry of the set $A$, i.e., to the tube function $V_A(\varepsilon)$, and can be understood as a sort of a \emph{fractal footprint} of the set $A$. Note that, for a general set $A\subseteq\mathbb R^N$, $V_A(\varepsilon)$ is the Lebesgue measure $|A_\varepsilon|$ of the $\varepsilon$-neighborhood $A_\varepsilon$ of the set $A$.

In addition, this paper gives a new application of the relatively new theory of complex dimensions and fractal zeta functions developed in \cite{fzf,dtzf,mefzf,refds,ftf_A},\ \cite{lapidusfrank12}, to studying orbits of particular dynamical systems. 
In particular, we compute here explicitly the distance zeta function of orbits for simplest, model dynamical systems \eqref{eq:model}, and find their meromorphic extension to all of $\mathbb C$. Furthermore, for general parabolic germs \eqref{eq:par}, we show theoretically the existence of such meromorphic extensions and describe their poles.
 
We note that, in \cite{fzf,ftf_A},\ \cite{lapidusfrank12}, it was shown on numerous examples that the complex dimensions have a very specific geometric meaning which can be thought of as describing the intrinsic oscillations in the geometry of fractal (self-similar) sets.
It was observed that these \emph{intrinsic oscillations} generate complex dimensions which are not real. On the other hand, these intrinsic oscillations explicitly break the classical asymptotic expansion of the tube function $V_A(\varepsilon)$ of the given set as $\varepsilon\to 0^+$. The first term where the expansion breaks is not monotonic, and it was conjectured that it is directly connected to the nonreal complex dimensions with maximal real part.

The dichotomy of existence of real versus nonreal complex dimensions has led the authors of \cite{fzf,lapidusfrank12} to suggest a new definition of \emph{fractality} of subsets of $\mathbb{R}^N$, according to which a set is fractal if it possesses a nonreal complex dimension. It was noted in e.g. \cite{fzf} and \cite{lapidusfrank12} that self-similar sets usually satisfy this definition of fractality. 
Some of the well-known examples which fail to be fractal according to the classical\footnote{Mandelbrot's original definition states the set to be fractal if its topological dimension is strictly less than its fractal (Hausdorff) dimension; see \cite{Man}.} ``definition'', but almost everyone expects them to be fractal, such as the famous \emph{devil's staircase}, are indeed detected to be fractal by the presence of nonreal complex dimensions. By this new definition, orbits of parabolic diffeomorphisms $f(x)=x+O(x^2)$, $x\to 0$, are not fractal, while, on the other hand, orbits of a hyperbolic diffeomorphism $f(x)=\lambda x+O(x^2)$, $0<\lambda<1$, $x\to 0$, are fractal. We also  note that this difference between hyperbolic and parabolic orbits can be justified since the orbits of a hyperbolic diffeomorphism may be considered self-similar (in a more general sense), as is explained in Section \ref{sec:hyp}. The parabolic orbits also provide a first example where the asymptotic expansion of the tube function breaks down due to oscillations, but this does not generate non-real complex dimensions.
This means that the notion of \emph{intrinsic oscillations} of a set should not be thought of as directly connected to the oscillations in its tube function but rather more subtle and yet to be defined rigorously.

In Section \ref{sec_conc} we give more comments on the fractal geometric prospects of the paper.
We also note that the orbits of dynamical systems that we study in this paper determine {\em fractal strings} in the sense of \cite{lapidusfrank12}. The corresponding zeta functions of these strings (called {\em geometric zeta functions} in \cite{fzf}) are simply related to distance zeta functions \eqref{zeta_f}. Moreover, the geometric zeta functions generalize the notion of the classical Riemann zeta function.
We would also like to draw the attention of the reader to the fact that there is a very interesting reformulation of the Riemann hypothesis in terms of fractal strings and the existence of their nonreal complex dimensions; see \cite[Chapter 9]{lapidusfrank12}.

\bigskip

More precisely, in this paper we study
 \emph{parabolic} analytic germs\footnote{\emph{Germified} at fixed point $0$. By analytic \emph{germs}, we mean equivalence classes by the following relation: two diffeomorphisms analytic at $0$ are considered equal if there exists a neighborhood $(-\delta,\delta),\ \delta>0,$ of $0$, on which they coincide.} of diffeomorphisms on the real line, i.e., germs $f\in\mathrm{Diff}(\mathbb R,0)$ of the form:
\begin{equation}\label{eq:par}
f(x)=x-ax^{k+1}+o(x^{k+1}),\ a>0,\ x\to 0.
\end{equation} 
Note that, due to $a>0$, $0$ is an attracting fixed point. Otherwise, if $a<0$, $0$ is a repelling fixed point, and we consider the inverse $f^{-1}$ instead.

By a formal change of variables in the class of formal power series tangent to the identity, $x+x^2\mathbb R[[x]]$, $f$ can be reduced to a normal form which is a time-one map of a simple vector field \cite{loray}:
\begin{equation}\label{eq:model}
f_0(x)=\mathrm{Exp}\left(-\frac{x^{k+1}}{1-\rho x^k}\frac{d}{dx}\right).\mathrm{id},\ k\in\mathbb N,\ \rho\in\mathbb R.
\end{equation}
Parabolic germs of the type \eqref{eq:model} are called \emph{model diffeomorphisms}. The pair $(k,\rho)\in\mathbb N\times\mathbb R$ determines the formal class of $f$ and is called the \emph{formal invariants} of $f$.

The first term of the asymptotic expansion of the tube function $V_f(\varepsilon)$, as $\varepsilon\to0^+$, was studied in \cite{NeVeDa} and then in \cite{MRZ} in a more general context. It was shown that the box dimension of the orbit determines the multiplicity $k$ of the fixed point.
\begin{comment}
we studied the first term of the 
we have shown that we can read the multiplicity of a fixed point of $f$ from the box dimension and the Minkowski content of its orbit. They are related to the first asymptotic term (its exponent and its coefficient) in the asymptotic expansion of $A(\mathcal O_f(x_0))_\varepsilon$, as $\varepsilon\to 0$. 
\end{comment}
Furthermore, by \cite{formal}, the formal class of a parabolic germ $f$ can be explicitly seen from three terms in the asymptotic expansion of the tube function $V_f$.

\begin{comment}
it generates a \emph{one-dimensional discrete dynamical system} on some interval $(0,d)$, $d>0$. In dynamics, such germs appear, for example, as first return maps around monodromic eliptic points or limit cycles in planar systems \cite{Roussarie}. Their complex versions $f\in\mathrm{Diff}(\mathbb C,0)$ appear as holonomy maps of complex resonant saddles, see e.g. \cite{loray}. In view of questions of cyclicity \cite{Roussarie} or of classifications of saddle foliations \cite{moussou}, the first step is reduction to formal normal forms. The formal class of a germ $f\in\mathrm{Diff}(\mathbb R,0)$ is described by two parameters - the multiplicity $k\in\mathbb N$ and the \emph{residual invariant} $\rho\in\mathbb R$.  Standardly, see e.g. \cite{loray}, formal changes of variables belong to $z\mathbb C[[z]]$. Here, we consider formal classes with respect only to formal changes of variables $\widehat\varphi(x)$ with \emph{real} coefficients, belonging to $x\mathbb R[[x]]$. Therefore, the attracting direction of $f$ does not change under changes of variables and remains $\mathbb R_+$. 
\end{comment}

\medskip

Inspired by this, one would also like to be able to determine the analytic class of $f$ from the tube function $V_f$.
Unfortunately, the analytic class of the diffeomorphism $f$
cannot be deduced from any finite jet of the Taylor expansion of $f$ (see e.g. \cite{ilyayak}) and, accordingly, of the asymptotic expansion of $V_f$.  On the other hand, it was shown in \cite[Theorem B]{MRRZ3} that the asymptotic expansion of the tube function $V_f$ in the power-logarithmic scale exists only up to a term $O(\varepsilon^{2-\frac{1}{k+1}})$.
The reason why the expansion breaks lies in the noncontinuous nature of the {\em integer critical time} $n_\varepsilon$ appearing in the calculation of the tube function $V_f(\varepsilon)$. 
Namely, the $\varepsilon$-neighborhood $\mathcal{O}(x_0)_\varepsilon$ of the orbit is a union of intervals of length $2\varepsilon$ centered at the iterates of $x_0$ under the function $f$.
Then, the integer critical time $n_\varepsilon$ represents the first index $n$ for which the intervals $\big((f^{\circ n}(x_0)-\varepsilon,f^{\circ n}(x_0)+\varepsilon)\big)_n$ start to overlap.  
It is then easy to see that
\begin{equation}\label{eq:epsi}
V_f(\varepsilon)=f^{\circ n_\varepsilon}(x_0)+n_\varepsilon\cdot 2\varepsilon.
\end{equation}
The integer critical time $n_\varepsilon$ is determined by inequalities arising from the above described overlapping condition. This is precisely the reason why the function $V_f$ fails to have an asymptotic expansion beyond a certain term; namely, $n_{\varepsilon}$ introduces higher-order oscillatory terms in \eqref{eq:epsi}. 
In order to circumvent this problem, the idea in \cite{MRRZ3} was to introduce \emph{continuous iterates} $f^{\circ\tau},\ \tau>0,$ of the function $f$ by embedding $f$ as a time-one map in a flow. By standard results about formal and analytic classification of parabolic analytic germs, each parabolic diffeomorphism with real coefficients can be embedded in a flow formally, see \eqref{eq:model}, but there also exists an analytic embedding on $(0,\delta)$. This enables one to define the \emph{continuous critical time} $\tau_\varepsilon$ by replacing the inequalities defining $n_{\varepsilon}$ by the corresponding equality, i.e., $f^{\circ \tau_\varepsilon}(x_0)-f^{\circ (\tau_\varepsilon+1)}(x_0)=2\varepsilon$. Then the \emph{continuous time tube function} was defined by:
\begin{equation}\label{eq:ac}
V_f^{\mathrm{c}}(\varepsilon)=f^{\circ \tau_\varepsilon}(x_0)+\tau_\varepsilon\cdot 2\varepsilon.
\end{equation}
Note that the integer critical time $n_\varepsilon$ is the integer part of the continuous critical time $\tau_{\varepsilon}$.
For details, see \cite{MRRZ3}.

\begin{comment}
Indeed $\varepsilon\mapsto n_\varepsilon$ is an integer jump-function as $\varepsilon\to 0$, characterized by the inequalities:
$$
|f^{\circ n_\varepsilon}(x_0)-f^{\circ (n_\varepsilon+1)}(x_0)|\leq 2\varepsilon,\ |f^{\circ (n_\varepsilon-1)}(x_0)-f^{\circ n_\varepsilon}(x_0)|>2\varepsilon.
$$ 
To solve this problem of nonexistence of the expansion, we have introduced in \cite{MRRZ3} the \emph{continuous length} of the $\varepsilon$-neighborhood of the orbit, denoted by $A^c(\mathcal O_f(x_0)_\varepsilon)$, which relies on \emph{embedding of a germ in a flow}. This allows the definition of the \emph{continuous critical time} $\tau_\varepsilon$ instead of discrete time $n_\varepsilon$:
\begin{equation}\label{eq:teeps}
|f^{\circ \tau_\varepsilon}(x_0)-f^{\circ (\tau_\varepsilon+1)}(x_0)|=2\varepsilon,
\end{equation}
and

The formal normal form is just a \emph{coarse} approximation of a germ, and the question that is posed is the analytic class of a germ. No finite jet of a germ can reveal its analytic class; instead, the analytic class is described by a sequence of diffeomorphisms defined on space of orbits called the \emph{Ecalle-Voronin moduli} \cite{voronin}.
\end{comment}

It was shown in \cite[Theorem B]{MRRZ3} that $V_f^{\mathrm{c}}(\varepsilon)$ admits a full asymptotic expansion in the power-logarithmic scale as $\varepsilon\to 0^+$ and, moreover, that this asymptotic expansion coincides with the asymptotic expansion of the standard tube function $V_f(\varepsilon)$ up to $O(\varepsilon^{2-\frac{1}{k+1}})$. 

The continuous time tube function $V_f^{\mathrm c}(\varepsilon)$ thus gives a \emph{dynamical smoothening} of $V_f(\varepsilon)$. In this paper, we also obtain the existence of the full asymptotic expansion of $V_f(\varepsilon)$ in the distributional sense (which is also related to complex dimensions of the orbit), and examine the relationship of these two notions. We hope that in the future we will be able to determine the analytic class of $f$ from either of these two expansions.

\begin{remark}[About notation]
We remark here that the notation and terminology concerning epsilon-neighborhoods of orbits and their tube functions that we are using in this paper is the one used in, e.g., \cite{fzf,dtzf} for general bounded subsets of the Euclidean space. 
On the other hand, the same notions are also used in \cite{MRRZ3} under different terminology and notation.
Namely, in the present paper the tube function is denoted by $V_f(\varepsilon)$, whereas it is denoted by $A\big(\mathcal{O}_f(x_0)_\varepsilon\big)$ in e.g. \cite{MRRZ3},\ \cite{MRZ},\ \cite{formal}, and called the \emph{length of the $\varepsilon$-neighborhood of the orbit}.
The same remark applies to $V_f^{\mathrm{c}}(\varepsilon)$ here and $A^c\big(\mathcal{O}_f(x_0)_\varepsilon\big)$ in \cite{MRRZ3}. Moreover, in \cite{MRRZ3}, $V_f^{\mathrm{c}}(\varepsilon)$ is called the {\em continuous length of the $\varepsilon$-neighborhood of the orbit}.
\end{remark}

\section{Main results}
In this paper we prove two main results: Theorem A and Theorem B.
In Theorem A, we obtain explicit asymptotic expansion of the tube function $V_f$ and the continuous time tube function $V_f^{\mathrm{c}}$. 
In fact, we obtain three expansions. Namely, we give a \emph{generalized} full asymptotic expansion of the tube function $V_f$, where, once the expansion in the power-logarithm scale ceases to exist, the expansion continues using a periodic discontinuous function of the critical time $\tau_\varepsilon$.
Next, we show that the tube function $V_f$ has a full asymptotic expansion in the power-logarithmic scale, but \emph{in the sense of Schwartz distributions}. 
Finally, we give a full pointwise asymptotic expansion of the \emph{continuous tube function} in the power-logarithm scale. 

Theorem A is used in the proof of Theorem B which shows that the zeta function $\zeta_f$ of a parabolic diffeomorphism $f$ can be extended to a meromorphic function on the whole complex plane $\mathbb{C}.$ Moreover, the poles and residues of the zeta function $\zeta_f$ are related to the asymptotic expansion of the continuous time tube function $V_f^{\rm{c}}(\varepsilon)$, as stated in Theorem A. 

In Subsection~\ref{subsec:model} we explicitely compute the meromorphic extensions of distance zeta functions and complex dimensions for the simplest, model parabolic diffemorphisms \eqref{eq:model} with residual invariant $\rho$ equal to $0$.

In addition, we also state and prove Theorem C which is used, alongside Theorem A, in the proof of Theorem B. It is also of independent interest in the general theory of fractal zeta functions.
Moreover, it gives new insights into the connection between distributional asymptotics of the tube function of a given set and its complex dimensions.
It is in fact a generalization of \cite[Theorem 2.3.18]{fzf} and partial converse of \cite[Theorem 5.4.30]{fzf}, thus giving, under some assumptions, a one-to-one correspondence of the distributional expansion of the tube function of a set and its complex dimensions.

Finally, we also draw the attention of the reader to Theorem \ref{shifted_a_prop} which gives better understanding of the complex dimensions of the (shifted) $a$-string, which is a well-known example of a non self-similar fractal string much discussed in \cite{lapidusfrank12,fzf}. Theorem~\ref{shifted_a_prop} generalizes and slightly extends \cite[Theorem 6.2]{lapidusfrank12}.

\bigskip

\begin{thmA} \label{thmA}
Let $f\in\mathrm{Diff}(\mathbb R,0)$ be a parabolic analytic diffeomorphism of formal type $(k,\rho)$, $k\in\mathbb N$, $\rho\in\mathbb R$, with attracting direction at $\mathbb R_+$: 
\begin{equation}\label{eq:formaA}
f(x)=x-ax^{k+1}+o(x^{k+1}),\ a>0,\ x\to 0.
\end{equation}
Let $\mathcal O_f(x_0),\ 0<x_0<1,$ be an orbit, and let $\tau_\varepsilon$ be its continuous critical $\varepsilon$-time defined by:
$$
g(f^{\circ \tau_\varepsilon})=2\varepsilon,\ \varepsilon>0,
$$
where $g=\mathrm{id}-f$. Let us denote by
$$
I_{a,k}(\varepsilon):=2^{\frac{1}{k+1}}a^{-\frac{1}{k+1}}\frac{k+1}{k}\cdot\varepsilon^{\frac{1}{k+1}}+\sum_{m=2}^{k}a_m\cdot \varepsilon^{\frac{m}{k+1}}+2\rho\frac{k-1}{k}\cdot \varepsilon\log\varepsilon+b_{k+1}(x_0)\varepsilon.
$$

\begin{enumerate}
\item The {\em tube function} $\varepsilon\mapsto V_f(\varepsilon)$ admits the following asymptotic expansion:
\begin{align*}
V_f(\varepsilon)&\sim I_{a,k}(\varepsilon)+\sum_{m=k+2}^{2k}\sum_{p=0}^{\lfloor\frac{m}{k}\rfloor+1} c_{m,p}\varepsilon^{\frac{m}{k+1}}\log^p \varepsilon+\sum_{p=1}^{\lfloor \frac{2k+1}{k}\rfloor+1}c_{2k+1,p}\varepsilon^{\frac{2k+1}{k+1}}\log^p\varepsilon+\\
&\quad+\tilde P_{2k+1}(G(\tau_\varepsilon))\cdot\varepsilon^{\frac{2k+1}{k+1}}+\sum_{m=2k+2}^{\infty}\sum_{p=0}^{\lfloor \frac{m}{k}\rfloor+1} \tilde Q_{m,p}(G(\tau_\varepsilon))\cdot\varepsilon^{\frac{m}{k+1}}\log^p\varepsilon,\ \varepsilon\to 0^+.
\end{align*} Here, $a_m\in\mathbb R$, $m=2,\ldots,k$, and $c_{m,p}\in\mathbb R, \ m=k+2,\ldots,2k+1,\ p=0,\ldots,\lfloor\frac{m}{k}\rfloor+1,$ do not depend on the initial condition $x_0$, while $b_{k+1}(x_0)\in\mathbb R$ depends on $x_0$. Moreover, $\tilde P_{2k+1}(s)$ resp. $\tilde Q_{m,p}(s)$, $m\geq 2k+2,\ p=0,\ldots,\lfloor\frac{m}{k}\rfloor+1$, are \emph{polynomials} of degree at most $2$ resp. $\lfloor \frac{m-1}{k}\rfloor$ for $p=0$, i.e., $\lfloor \frac{m}{k}\rfloor-1$ for $p\geq 1$, whose coefficients in general depend on coefficients of $f$ and on the initial condition $x_0$, and $G:[0,+\infty)\to\mathbb R$ is $1$-periodic, given on its period by $G(s)=1-s,\ s\in(0,1)$,  $G(0)=0$.
\medskip

\item The {\em continuous time tube function} $\varepsilon\mapsto V_f^{\mathrm{c}}(\varepsilon)$ admits the following asymptotic expansion:
\begin{align*}
V_f^{\mathrm{c}}(\varepsilon) \sim  I_{a,k}(\varepsilon)+\sum_{m=k+2}^{\infty}\sum_{p=0}^{\lfloor\frac{m}{k}\rfloor+1} c_{m,p}\varepsilon^{\frac{m}{k+1}}\log^p \varepsilon,\ \varepsilon\to 0^+.
\end{align*}
Here, $a_m$, $m=2,\ldots,k,$ $b_{k+1}(x_0),$ and $c_{m,p}, \ m=k+2,\ldots,2k+1,\ p=0,\ldots,\lfloor\frac{m}{k}\rfloor+1,$ $(m,p)\neq (2k+1,0)$, are as in $(1)$. Furthermore, $c_{2k+1,0}$ resp. $c_{m,p}$, $m\geq 2k+2,\ p=0,\ldots,\lfloor\frac{m}{k}\rfloor+1$, are \emph{free coefficients} of polynomials $\tilde P_{2k+1}$ resp. $\tilde Q_{m,p}$ from $(1)$. Only the coefficient $b_{k+1}(x_0)$ depends on the initial condition $x_0$.
\medskip

\item The {\em distributional} asymptotic expansion of $\varepsilon\mapsto V_f(\varepsilon)$ is given by:
\begin{align*}
&V_f(\varepsilon)\sim_{\mathcal D} I_{a,k}(\varepsilon)+\sum_{m=k+2}^{2k}\sum_{p=0}^{\lfloor\frac{m}{k}\rfloor+1} c_{m,p}\varepsilon^{\frac{m}{k+1}}\log^p \varepsilon+\sum_{p=1}^{\lfloor \frac{2k+1}{k}\rfloor+1}c_{2k+1,p}\varepsilon^{\frac{2k+1}{k+1}}\log^p\varepsilon+\\
&\quad+d_{2k+1,0}(x_0)\cdot\varepsilon^{\frac{2k+1}{k+1}}+\sum_{m=2k+2}^{\infty}\sum_{p=0}^{\lfloor \frac{m}{k}\rfloor+1} d_{m,p}(x_0)\cdot\varepsilon^{\frac{m}{k+1}}\log^p\varepsilon,\ \varepsilon\to 0^+.
%&\quad +\sum_{m=k+1}^{2k}b_m(x_0)\varepsilon^{\frac{m}{k+1}}+\sum_{m=k+2}^{2k+1}\sum_{p=1}^{\lfloor\frac{m}{k}\rfloor} c_{m,p}(x_0)\varepsilon^{\frac{m}{k+1}}\log^p\varepsilon+\\
%&\quad+\sum_{m=2k+1}^{\infty} p_m(x_0)\cdot\varepsilon^{\frac{m}{k+1}}+\sum_{m=2k+2}^{\infty}\sum_{p=1}^{\lfloor \frac{m}{k}\rfloor} q_{m,p}(x_0)\cdot\varepsilon^{\frac{m}{k+1}}\log^p\varepsilon,\ \varepsilon\to 0.
\end{align*}
Here, $a_m$, $m=2,\ldots,k,$ $b_{k+1}(x_0),$ and $c_{m,p}, \ m=k+2,\ldots,2k+1,\ p=0,\ldots,\lfloor\frac{m}{k}\rfloor+1,$ $(m,p)\neq (2k+1,0)$, are as in $(1)$ and $(2)$, and only $b_{k+1}(x_0)$ depends on $x_0$. Furthermore, $d_{2k+1,0}(x_0)$ resp. $d_{m,p}(x_0)$, $m\geq 2k+2,\ p=0,\ldots,\lfloor\frac{m}{k}\rfloor+1$, are given as
\begin{align*}
&d_{2k+1,0}(x_0):=\int_0^1 \tilde P_{2k+1}(s)\,ds,\\
&d_{m,p}(x_0):=\int_0^1 \tilde Q_{m,p}(s)\,ds,\ m\geq 2k+2,\ p=0,\ldots,\left\lfloor\frac{m}{k}\right\rfloor+1,
\end{align*}
i.e., as \emph{mean values} of $1$-periodic functions $\tilde P_{2k+1}\circ G$ and $\tilde Q_{m,p}\circ G$ from $(1)$, that is, as integrals over the period divided by the length of the period. In general, they depend on the initial condition $x_0$.
\end{enumerate}
\end{thmA}
\smallskip

The proof of Theorem~A is in Section~\ref{sec:proofA}.
\bigskip

Recall that the {\em complex dimensions} of a bounded set are defined as the set of poles of the corresponding distance zeta function, after it is meromorphically extended to all of $\Ce$.
For details, see Definition \ref{def_dim_c}.  
\medskip
 
\begin{thmB}[Distance zeta functions and complex dimensions of orbits of general parabolic germs]
	\label{para_dist_zeta}
	 Let $f\in\mathrm{Diff}(\mathbb R,0)$ be an attracting parabolic germ in the formal class $(k,\rho)$, $k\in\mathbb N$, $\rho\in\mathbb R$, $f(x)=x-ax^{k+1}+\ldots,\ a>0$. Let $\mathcal{O}_f(x_0)$ be its orbit with initial point $x_0>0$. 
	Then the distance zeta function $\zeta_f(s)$ can be meromorphically extended to all of $\Ce$.
	More precisely, in any open right half-plane $W_{M}:=\{\re s>1-\frac{M}{k+1}\}$, where $M\in\mathbb N$, $M>k+2$, it is given as%\footnote{Here, $[x]_j:=x(x-1)\cdots(x-j+1)$, $x\in\mathbb R$, $j\in\mathbb N$, denotes the \emph{decreasing factorial}, with convention $[x]_0:=1$.}
	\begin{align}\begin{split}\label{dis.zet.para}
		\zeta_f(s)=(1-s)\sum_{m=1}^{k}&\frac{a_m}{s-\left(1-\frac{m}{k+1}\right)}+(1-s)\Big(\frac{b_{k+1}(x_0)}{s}+\frac{a_{k+1}}{s^2}\Big)+\\
		+&(1-s)\sum_{m=k+2}^{M-1}\sum_{p=0}^{\lfloor\frac{m}{k}\rfloor+1}\frac{(-1)^p p!\cdot c_{m,p}(x_0)}{\Big(s-\big(1-\frac{m}{k+1}\big)\Big)^{p+1}} +g(s),\ s\in W_{M},
	\end{split}\end{align}
	where $g(s)$ is holomorphic in $W_{M}$. 
% 	\begin{equation}
% 		g(s)=\int_0^{1}t^{s-1}O(t^{\frac{M}{k+1}-d})\di t                                              
% 	\end{equation}
% 	in that half-plane.  
%and the corresponding residues are $a_i\cdot \frac{i}{k+1}$, for $i=1,\ldots,2k$, and $b_i\cdot \frac{i}{k+1}$, for $i=2k+1,\ldots M-1$. 

\noindent Here, the coefficients are the same as in the distributional expansion $(3)$ in Theorem~A. More precisely, $a_m$, $m=2,\ldots,k$, are the coefficients in front of $\varepsilon^{\frac{m}{k+1}}$, $a_1:=2^{\frac{1}{k+1}}a^{-\frac{1}{k+1}}\frac{k+1}{k}$, and $b_{k+1}(x_0)$ resp. $a_{k+1}:=2\rho\frac{k-1}{k}$ are the coefficients in front of $\varepsilon$ resp. $\varepsilon\log\varepsilon$ in this expansion. Furthermore, $c_{m,p}(x_0)$, $m\geq k+2$, $p=0,\ldots, \lfloor \frac{m}{k}\rfloor+1$, are the coefficients in front of $\varepsilon^{\frac{m}{k+1}}\log^p\varepsilon$ in expansion $(3)$ of Theorem~A.
Up to $\varepsilon^{\frac{2k+1}{k+1}}$, they do not depend on the initial condition $x_0$.

Moreover, $\zeta_f(s)$ is languid\footnote{This means roughly that $|\zeta_f(\beta+i\tau)|\leq C|\tau|^\alpha$, for some $\alpha\geq 0$, as $|\tau|\to\infty$ inside a finite vertical strip, but the full definition of languidity is a little bit more technical. For details see \cite[Definitions 5.1.3 and 5.3.9]{fzf} and Definition \ref{sup-languidity}.} in the polynomial sense for any {\em screen} given as a vertical line $\{\re s=\beta\}$, where $\beta>1-\frac{M}{k+1}$, and the corresponding {\em window} $\{\re s>\beta\}$.\footnote{The notions of the screen and the window are introduced in Definition \ref{def_dim_w} below.}
\end{thmB}
Note that, by Theorem~B, all of the poles of $\zeta_f(s)$ in the right half-plane $W_M$, for $M\in\mathbb N$, are given as: 
$$\omega_m=1-\frac{m}{k+1},\ \ m=1,\ldots,M-1.$$ 
\smallskip

From Theorems~A and B, we deduce the following Corollary. 
It enables one to see the formal class of a parabolic germ directly from its two complex dimensions and their principal parts. 
In this sense, it gives a partial answer to the question which has motivated this paper. 
Having obtained the full asymptotic expansion of the tube functions and having related its terms to the complex dimensions of  the orbit, we hope, in the future, to be able to recover also the analytic class of a parabolic germ.

\begin{corollary}[Formal class of a parabolic germ from complex dimensions]\label{cor:formal}
Let $f\in\mathrm{Diff}(\mathbb R,0)$, $f(x)=x-ax^{k+1}+o(x^{k+1})$, $a>0$, be an attracting parabolic germ from the formal class $(k,\rho)$, $k\in\mathbb Z,\ \rho\in\mathbb R$. Let $\zeta_f$ be the distance zeta function of its orbit $\mathcal{O}_f(x_0)$, $x_0>0$. Let $a_1$ and $a_{k+1}$ be as in \eqref{dis.zet.para}. Then the formal class $(k,\rho)$ of $f$ can be seen explicitely from two complex dimensions of $\mathcal{O}_f(x_0)$ and their principal parts:
\begin{enumerate}
\item the simple pole of $\zeta_f$ with largest real part, $\omega_1=1-\frac{1}{k+1}$, and its residue: $$\mathrm{Res}(\zeta_f(s),\omega_1)=\frac{a_1}{k+1}=\frac{2^{\frac{1}{k+1}}a^{-\frac{1}{k+1}}}{k},$$
\item the double pole of $\zeta_f$ with largest real part, $\omega_{k+1}=0$, and the residue: 
$$\mathrm{Res}(s\cdot \zeta_f(s),\omega_{k+1})=a_{k+1}=2\rho\frac{k-1}{k}.$$
\end{enumerate}
\end{corollary}
\smallskip

Theorem~A is used in the proof of the main Theorem~B. 
It gives a precise description of the oscillatory terms in the asymptotic expansion of the tube function $V_f$ of orbits of parabolic germs, thus refining the statement of \cite[Theorem B]{MRRZ3}.
It also describes the relation of those oscillatory terms to non-oscillatory terms of the asymptotic expansion of the continuous time tube function $V_f^{\rm{c}}$ of orbits, defined by \eqref{eq:ac}.
\begin{comment}
Moreover, in statement (3), used explicitly in Theorem ~B, it explicitly reveals the complex dimensions and their corresponding residues for orbits of general parabolic germs, expressed in terms of coefficients of the distributional asymptotic expansion of the tube function $V_f(\varepsilon)$, as $\varepsilon \to 0$. 

As a simple consequence, relying on results of \cite{formal}, where the formal class of a parabolic germ was deduced from particular coefficients and exponents of the asymptotic expansion of the length of the $\varepsilon$-neighborhood of its orbits, we deduce Corollary~\ref{cor:formal} below. It reads the formal class of a parabolic germ from complex dimensions of its any orbit.
\end{comment}

Both distributional asymptotic expansions of the tube function, introduced in \cite{lapidusfrank12,fzf}, and the asymptotic expansions of the continuous time tube function of orbits, introduced in \cite{MRRZ3}, are continuations of the pointwise power-logarithmic asymptotic expansions of the tube function $V_f$ of orbits of parabolic germs, after it ceases to exist in power-logarithmic scale due to presence of oscillations. This happens after first finitely many terms, see \cite{formal} and \cite{nonlin}. Note that the first term \emph{carrying} an oscillatory term in statement $(1)$ of Theorem~A in the expansion of $V_f$ is of order $O(\varepsilon^{2-\frac{1}{k+1}})$, in accordance with the weaker statement of \cite[Theorem B]{MRRZ3}. 
In Section~\ref{sec:examples} we show that, for simplest model parabolic germs from formal class $(k,\rho=0)$, $k\in\mathbb N$, the two expansions (distributional and continuous time) seem to be related by an explicit formula on their coefficients. The formula is proven in the case $k=1$ in Proposition~\ref{prop:regular} and conjectured for $k\geq 2$.

The continuous time tube function may thus be considered as some sort of \emph{dynamical regularization} of $V_f$. In Theorem~A , we compare the two {regularizations} of the asymptotic expansion of $V_f$, as $\varepsilon\to 0$: 
\begin{itemize}
\item[-] the former given by the asymptotic expansion of $V_f^{\mathrm{c}}$ (from \cite{MRRZ3}),
\item[-] the latter given by the distributional asymptotic expansion of $V_f$ (from \cite{fzf}). 
\end{itemize}
%In Proposition~\ref{prop:regular} in Section~\ref{sec:examples}, we show that, for \emph{model} parabolic germs with $\rho=0$, the coefficients of the above two expansions are related by a simple formula. 
Moreover, in Subsection~\ref{sec:hyp}, we give an example of a model \emph{hyperbolic} orbit for which these two \emph{regularizations} are, unlike in the parabolic case, \emph{essentially} different, in the sense that ``resonant'' oscillatory terms survive in the distributional expansion, while, at the same time, they do not appear in the expansion of $V_f^{\mathrm{c}}(\varepsilon)$.

\medskip

 Theorem~B establishes the existence of the meromorphic extension of the distance zeta function $\zeta_f(s)$ of orbits of \emph{general parabolic germs}, and gives their complex dimensions and polynomial languidity bounds, \emph{without explicitly computing the distance zeta-functions}.
 This is concluded using only the distributional expansion of tube functions $V_f(\varepsilon)$, as $\varepsilon\to 0^+$, given in Theorem~A for general parabolic orbits.
 Theoretically, this approach represents an important advantage, since a direct computation of distance zeta functions of orbits of parabolic germs and the proof of their meromorphic extension is elaborate already in the model cases, as can be seen in Subsection~\ref{subsec:model} and Section~\ref{sec:modelcomp}. 

 It is even more difficult to obtain rational or polynomial languidity bounds necessary to be able to recover the (distributional) asymptotics of the tube function $V_f(\varepsilon)$, as $\varepsilon\to 0^+$, from the complex dimensions of $\mathcal{O}_f(x_0)$, as explained in \cite[\S5.3.2--5.3.3, esp., Thms.\ 5.3.16\ and 5.3.21]{fzf} or \cite{ftf_A}.
 Here, the existence of polynomial languidity bounds follows theoretically from Theorem~B. Therefore, this theoretical importance aside, another way of using Theorem~B is the following. Assume, for instance, that one is able to obtain  a meromorphic extension of the zeta function of one particular parabolic orbit in some way, and from this extension one finds the complex dimensions of the orbit. One then already knows that they will, indeed, be the co-exponents (i.e. dimension of the ambient space minus the exponents) that appear in the (pointwise or distributional) asymptotics of the tube function $V_f(\varepsilon)$, the existence of which is guaranteed by Theorem~A. This conclusion follows without the difficult task of first proving the languidity bounds on the fractal zeta function because they are already guaranteed by Theorem~B.
Indeed, the coefficients appearing in the distributional asymptotics in Theorem~A may be rather difficult to compute directly, but can be obtained more easily from the residues at the corresponding complex dimensions.
 This approach was used in the model case, see Proposition \ref{prop:geo_k} below and its proof in Section \ref{sec:modelcomp}. 
 For more details, see Remark~\ref{rem:zadnji}.
\medskip

By Theorems A and B, the distributional asymptotic expansions of the tube functions and the collection of complex dimensions and corresponding principal parts of orbits of parabolic germs carry the same information, as stated in the following corollary.
\begin{corollary}\label{cor:oneone} For orbits of general parabolic germs from $\mathrm{Diff}(\mathbb R,0)$, there is a one-to-one correspondence between: on one hand, the distributional asymptotics of tube functions of their orbits and, on the other hand, the collection of poles $($the complex dimensions$)$ and their corresponding principal parts of the meromorphic extensions of distance zeta functions of their orbits. 
\end{corollary}
\begin{proof}
Both the distributional asymptotic expansions of the tube functions of orbits of parabolic germs and the meromorphic extensions to all of $\mathbb C$ of their distance zeta functions exist by Theorems~A $(3)$ and Theorem~B. The statement of the Corollary follows directly by \eqref{dis.zet.para} in Theorem~B and Theorem~A $(3)$.
\end{proof}
 \bigskip

 Next we state Theorem~C which is used directly in the proof of Theorem~B but is also of independent interest in the theory of complex dimensions. Therefore we state it here in the full generality, for arbitrary bounded subsets of $\mathbb{R}^N$.

\begin{thmC}\label{prop:tym}
Let $A\subseteq\eR^N$ be a bounded set such that, for some $m\in\mathbb N_0$, the $m$-th primitive\footnote{The $m$-th primitive function $V_A^{[m]}$ of $\varepsilon\mapsto V_A(\varepsilon)$ is precisely defined in \eqref{k-th-prim}. In the sequel, we will call $V_A^{[m]}$ the \emph{$m$-th primitive tube function of the set $A$}, as in \cite{fzf}.} of the function $\varepsilon\mapsto V_A(\varepsilon)$, denoted by $V_A^{[m]}(\varepsilon)$, has the following asymptotics, for some $I\in\mathbb N$:
	\begin{equation}\label{log_asym11}
		V_A^{[m]}(\varepsilon)=\sum_{i=1}^{I} \varepsilon^{\alpha_i+m}\cdot M_i\cdot P_i(-\log\varepsilon)+O(\varepsilon^{\alpha_I+\gamma+m}),
	\end{equation}
	as $\varepsilon\to0^+$, for some $0<\alpha_1< \alpha_2<\ldots<\alpha_I$, $\gamma>0$, and some constants $M_i\in\mathbb R$. Here, $P_i$ denote arbitrary monic\footnote{the coefficient of the leading monomial equal to $1$} polynomials with real coefficients and we let $n_i\in\eN_0$ denote their respective degrees.

Then, the distance zeta function ${\zeta}_A(s)$ of the set $A$ is meromorphic in the right open half-plane $\{\re s>N-\alpha_I-\gamma\}$ and given by:
	\begin{equation}\label{distance_zeta_polovi1}
		{\zeta}_A(s)=(N-s)_{m+1}\sum_{i=1}^{I} \Big(\delta^{s-N+\alpha_i}\sum_{j=0}^{n_i}\frac{M_i\cdot P_i^{(j)}(-\log\delta)}{(s-N+\alpha_i)^{j+1}}\Big)+R(s),
	\end{equation}
	where $R$ is holomorphic in $\{\re s>N- \alpha_I-\gamma\}$ and given by
	\begin{equation*}
		\label{R(s)}
		R(s)=M\sum_{n=0}^m(N-s)_{n}\delta^{s-N-n}V_{A}^{[n]}(\delta)+\int_{0}^{\delta}t^{s-1}O(t^{-\alpha_I-\gamma})\di t,
	\end{equation*}
and $M:=\sum_{i=1}^I M_i$.

Finally, the distance zeta function ${\zeta}_A(s)$ is languid in the polynomial sense for any screen given as a vertical line $\{\re s=\beta\}$ where $\beta >N- \alpha_I-\gamma$ and the corresponding window $\{\re s>\beta\}$. 
\end{thmC}

\noindent Here,  for $x\in\mathbb C$ and $m\in\mathbb N$,
\begin{equation*}
	(x)_m:=x(x+1)(x+2)\cdots (x+m-1),
\end{equation*}
denotes the \emph{rising factorial}, with the convention $(x)_0:=1$.
Note that $(x)_m=\frac{\Gamma(x+m)}{\Gamma(x)}$ and is also known as the \emph{Pochhammer symbol}.
\smallskip

Note that Theorem~C is stated for standard distance zeta functions of sets defined in \eqref{eq:dz}, not for their cut-off counterparts \eqref{zeta_f} that we use in this paper for convenience for orbits of parabolic germs.
\smallskip 

The choice of $\delta>0$ in Theorem~C is non-essential, in the sense that changing $\delta>0$ amounts to adding an entire function.
This is not immediately clear from \eqref{distance_zeta_polovi1}, but is a theoretical fact stemming from the same ambiguity in the definition of the distance zeta function \eqref{eq:dz}, which is clear from the proof of Theorem~C in Section~\ref{subsec:proofC}, in particular from the proof of Lemma~\ref{lemma_tilde}. For instance, we may let $\delta=1$ to get a nicer formula \eqref{distance_zeta_polovi1}. However, the change in $\delta>0$ can, in general, affect the languidity of $\zeta_A$ so we state the theorem here in full generality. In some applications, one may have to choose $\delta>0$ differently in order to obtain better languidity conditions. 
\smallskip

If  \eqref{log_asym11} holds with $m=0$ and $n_1=0$, then the constant $M_1$ is the Minkowski content of $A$ and $N-\alpha_1$ is its box dimension, $\dim_BA$.
In general, the $(n_1+1)$-th order pole $N-\alpha_1$ of ${\zeta}_A$ in \eqref{distance_zeta_polovi1}, which is the rightmost pole of $\zeta_A$, equals to the upper Minkowski (or box) dimension of $A$, i.e., $\overline{\dim}_B A=N-\alpha_1$. Indeed, the upper box dimension of a set $A$ is equal to the abscissa of convergence of $\zeta_A$; see \cite[Theorem 2.1.11]{fzf}.
See also Remark \ref{Tauber} in which we briefly discuss the generalized gauge Minkowski content in connection to Theorem~C.

% Moreover, the \emph{principal part of the pole $N-\alpha_i$}, $i=1,\ldots,I,$ is obtained by putting $\delta=1$ in \eqref{distance_zeta_polovi1}, which does not change the poles and their principal parts, as:
% $$
% M_i(\alpha_i)_{m+1}\sum_{j=0}^{n_i}\frac{P_{n_i}^{(j)}(0)}{(s-N+\alpha_i)^{j+1}},\ i\in\mathbb N.
% $$
\smallskip
Note that the assumption \eqref{log_asym11} in Theorem~C, by Proposition~\ref{lem:integrali} in Section~\ref{sec:proofA}, implies the following weaker distributional asymptotics of $\varepsilon\mapsto V_A(\varepsilon)$:
\begin{equation*}\label{eq:joh}
V_A(\varepsilon)=_{\mathcal D} \sum_{i=1}^{I}\varepsilon^{\alpha_i}Q_i(-\log\varepsilon)+O(\varepsilon^{\alpha_I+\gamma}),\ \varepsilon\to 0^+,
\end{equation*}
where $Q_i$ denote some polynomials of degree $n_i$, $i=1,\ldots,I$, which can be explicitly expressed by $\alpha_i,\, M_i,\, P_i$ from \eqref{log_asym11}, $i=1,\ldots,I$. This follows from the fact that $\frac{d^m}{d\varepsilon^m}\big(\varepsilon^{\alpha_i+m}\cdot M_i\cdot P_i(-\log\varepsilon)\big)=\varepsilon^{\alpha_i}Q_i(-\log\varepsilon)$. Note that pointwise asymptotics of the above type is not implied by \eqref{log_asym11}.
%\newpage

% ------------------------------------
% IZBACILA BIH!!!
% Aside from distributional asymptotics\edz{M: Ja bih izbacila ovaj dio. To je u dokazu Teorema B, ne bih tu ponavljala. KOrolar sam bacila naprijed, Korolar 2.2. G: Ok, možemo izbaciti što se mene tiče} stated in Theorem A $(1)$, \emph{stronger} assumptions of the type \eqref{log_asym11} with asymptotically decreasing remainders are satisfied for orbits of general parabolic germs, see \eqref{eq:ha} in the proof of Lemma~\ref{lema:druga} in Section~\ref{sec:proofA}. Using that, Theorem~B follows from Theorem~C somewhat indirectly. For details, see the proof of Theorem~B in Subsection \ref{sec:proof_B}.
% --------------------------------------

% 
% 
% 
% The orbits of $1$-dimensional dynamical systems to which we apply the theory of complex dimensions to read the intrinsic properties of generating functions are not \emph{fractal} in the standard sense (they are not \emph{self-similar}), but have an accumulation point at $0$. In fact, they are asymptotically $1/k$-strings, $k\in\mathbb N$, which were introduced in \cite{??}. \edz{Mozda ovo zadnje nema smisla jer i Lapidus ima te $a$-stringove koji su nasa orbita-dakle, fraktalni u nekom smislu} By Lapidus' [] definition of fractal sets, these orbits are not fractal, since they do not have non-real complex dimensions. This can be seen by the fact that the distributional expansion of the length of the $\varepsilon$-neighborhood of orbits does not contain oscillatory terms.

\subsection{Model parabolic germs}\label{subsec:model}\

In this subsection we present a model case in which the distance zeta function can be computed explicitly.
We consider the \emph{model parabolic diffeomorphsims} 
 $$f_0(x)=\mathrm{Exp}\big(-x^{k+1}\frac{d}{dx}\big).\mathrm{id},$$
whose formal invariants are: \emph{multiplicity} $k\in\mathbb N$ and the \emph{residual invariant} $\rho$ equal to $\rho=0$.
In this model case we are able to directly compute the distance zeta function $\zeta_{f_0}(s)$, as well as determine the complex dimensions, of its orbit. 
These cases are simple since we can use the theory of \emph{shifted $a$-strings} for meromorphic extensions of distance zeta functions of their orbits to all of $\mathbb C$. Additionally, since $\rho=0$, by proofs in Section~\ref{sec:proofA}, there are no logarithmic terms in the asymptotic expansion of the continuous time tube function $\varepsilon\mapsto V_f^{\mathrm{c}}(\varepsilon)$ nor in the distributional asymptotic expansion of $\varepsilon\mapsto V_f(\varepsilon)$, as $\varepsilon\to 0^+$, in Theorem~A. Only power terms of the type $\varepsilon^{\frac{i}{k+1}}$, $i\in\mathbb N$, are present.
% The proposition deals, in fact, with the geometric zeta function $\zeta_{\mathcal{L}_f}$, which is, by the functional equation \eqref{func_dist_geo}, directly related to the distance zeta function of $\zeta_f$ of the orbit $\mathcal{O}_f(x_0)$.
\smallskip

Let $\zeta(s,q)$, $q>0$, denote the \emph{Hurwitz zeta function}:
\begin{equation}\label{eq:hur}
\zeta(s,q):=\sum_{j=0}^{\infty}\frac{1}{(j+q)^s},\ \re s>1.
\end{equation}        
It admits an analytic extension to $\mathbb C\setminus\{1\}$. At $s=1$, it has a simple pole, with $\mathrm{Res}\big(\zeta(s,q),1\big)=1$. Moreover, $\zeta(0,q)=\frac12-q$. See e.g. \cite{apostol}.

\begin{proposition}[Meromorphic extension of distance zeta functions for model orbits, $k\geq 1$]\label{prop:geo_k}
Let $\zeta_{f_0}(s)$, \footnote{Note that $\dim_B\mathcal O_{f_0}(x_0)=\frac{k}{k+1}$, see e.g. \cite{NeVeDa}. Therefore, $\zeta_{f_0}(s)$ given by formula \eqref{zeta_f} is well-defined and analytic for $\mathrm{Re}(s)>\frac{k}{k+1}$.}$\mathrm{Re}(s)>\frac{k}{k+1}$, denote the distance zeta function of an orbit $\mathcal O_{f_0}(x_0)$ of a model parabolic germ $f_0$ belonging to the formal class $(k,\rho=0)$, $k\in\mathbb N$, as defined by \eqref{zeta_f}.
\begin{enumerate}
\item
The distance zeta function $\zeta_{f_0}(s)$ can be meromorphically extended to all of $\Ce$. The poles of $\zeta_{f_0}(s)$ are located at $\frac{k}{k+1}$ and at $($a subset of$\,)$ the set of points $\frac{-mk}{k+1}$, $m\in\eN_0$, and they are all simple. In particular, the Minkowski $($box$)$ dimension of $\mathcal O_{f_0}(x_0)$ is $D=\frac{k}{k+1}$, and this is the only pole of $\zeta_{f_0}(s)$ with a positive real part.
\\
\noindent More precisely, for $M\in\eN_0$, $\zeta_{f_0}(s)$ can be expressed as a sum of Hurwitz zeta functions in the following sense:
\begin{equation*}\label{zetaksan}
	\zeta_{f_0}(s)=\frac{2^{1-s}}{s}\sum_{n=0}^{M}\tilde{Z}_k(s,n)\zeta\left(\frac{k+1}{k}s+n,\frac{1}{kx_0^k}\right)+R(s),
\end{equation*}
where $R(s)$ is defined and holomorphic for $\re s>-\frac{Mk}{k+1}$. Moreover, for every $\gamma>0$, $R(s)=O(|s|^{M})$, as $|s|\to+\infty$ in $\{s\in\mathbb C:\mathrm{Re}(s)\geq -\frac{Mk}{k+1}+\gamma\}$.
% The functions $\tilde{Z}_k(s,n)$ are entire and given as
% \begin{equation}
% 	\tilde{Z}_k(s,n):=k^{-\frac{k+1}{k}s}\sum_{m=1}^{n}\left(-k\right)^m{s\choose m}\sum_{|{\mathbf{q}}_m|=n-m}\,\prod_{j=1}^m{-1/k\choose q_j+2},
% \end{equation}
% where the inner sum goes over all multiindices ${\mathbf{q}}_m:=(q_1,\ldots,q_m)\in\eN_0^m$ such that their length $|{\mathbf{q}}_m|:=q_1+\cdots+q_m$ equals to $n-m$ with the convention that $\tilde{Z}(s,0):=k^{-\frac{k+1}{k}s}$
The functions $\tilde{Z}_k(s,n)$, $n=0,\ldots, M,$ are entire and given by:
\begin{equation}\label{eq:kaa}
	\tilde{Z}_k(s,n):=k^{-\frac{k+1}{k}s}\sum_{m=0}^{n}(-k)^m{s\choose m}\sum_{i=0}^m(-1)^i{m\choose i}\sum_{j=0}^i(-k)^{-j}{i\choose j}{\frac{i-m}{k}\choose n+m-j}.
\end{equation}
In particular, $\tilde{Z}_k(s,0)=k^{-\frac{k+1}{k}s}$.
\smallskip

\item For any {\em screen} chosen as a vertical line $\{\re s=\sigma\}$ with fixed abscissa $\sigma\in\big(-\frac{Mk}{k+1},-\frac{(M-1)k}{k+1}\big]$, $\zeta_{f_0}(s)$ is super languid\footnote{For the precise definition of the \emph{screen} and the \emph{super languidity}, see Subsection~\ref{sec:results} below.} with exponent $\kappa_{f_0}=M$.
\end{enumerate}
\end{proposition}

\noindent The proof of Proposition~\ref{prop:geo_k} is in Subsection~\ref{subsec:astring}.
\medskip

As a special case, for $k=1$, we obtain the following formal expression for the distance zeta function $\zeta_{f_0}$:
\begin{equation}\label{zeta}
\zeta_{f_0}(s)\sim \frac{2^{1-s}}{s}\cdot \sum_{r=0}^{\infty}{-s \choose r}\zeta(2s+r,x_0^{-1}),
\end{equation}
from which we can see the poles of $\zeta_{f_0}(s)$, i.e., the complex dimensions of the orbit $\mathcal O_{f_0}(x_0)$.
Indeed, the only pole of the Hurwitz zeta function is at $s=1$, with residue equal to $1$. Therefore, $\zeta_{f_0}$ has possible simple poles at $w_r=\frac{1}{2}(1-r)$, $r\in\mathbb N_0$. It can be shown that the integer \emph{poles} $w_r$ (except $w_1=0$) have zero residues, since, in these cases, the binomial coefficient ${-w_r \choose r}$ is zero, and we have a zero-pole cancellation in \eqref{zeta}. 

\noindent Therefore, $\zeta_{f_0}$ has simple poles at $\omega_m:=\frac{1}{2}-m$, $m\in\mathbb N_0$, and at $\omega_{1/2}:=0$, with non-zero residues:\footnote{Note also that for $r>0$ the pole at zero is canceled by the factor ${-s\choose r}$ so that the only component of the residue comes from $2\cdot\zeta(0,x_0^{-1})=1-2x_0^{-1}$.}
$$
\mathrm{Res}(\zeta_{f_0},\omega_m)={m-\frac{1}{2}\choose 2m}\frac{2^{1/2+m}}{1-2m},\ m\in\mathbb N_0,\ \text{ and }\mathrm{Res}(\zeta_{f_0},0)=1-2x_0^{-1}.
$$ 
This also shows that the set of complex dimensions is exactly 
$$\dim_{\mathbb{C}}\left(\mathcal O_{f_0}(x_0)\right)=\{\omega_{1/2}\}\cup \{\omega_m:m\in\mathbb N\}.$$

In the case $k=1$, we also show directly, that is, without relying on Theorem~B, that $\zeta_{f_0}$ satisfies appropriate polynomial languidity estimates. 
%as $|\tau|\to\infty$ for any screen given as a vertical line $\{\re s=c\}$, while the estimates are 
These estimates are even negative power estimates (so-called \emph{rational}) if we restrict the domain of $\zeta_{f_0}$ to a closed vertical strip contained in $\{\re s>0\}$; see Proposition~\ref{geo_k1} and Corollary~\ref{dist_zeta_1}. 

\noindent For detailed computations in the case $k=1$, see Subsection~\ref{subsec:dva}.

%By \cite[Theorem 5.3.21]{fzf} (the \emph{inverse Mellin transform}), \eqref{zeta} gives the following \emph{asymptotic expansion of $\ell(\mathcal O^{f_0}(x_0))$} in the \emph{distributional sense}:
%\begin{equation}\label{eq:dist}
%\mathcal D\big(\ell(O^{f_0}(x_0)_\varepsilon)\big)\sim 2\sqrt 2\varepsilon^{1/2}-\frac{2}{x_0}\cdot \varepsilon+\sum_{r=1}^{\infty}\frac{2^{3/2+r}}{1-4r^2}{r-\frac{1}{2}\choose 2r} \varepsilon^{r+\frac{1}{2}},\ \varepsilon\to 0.
%\end{equation}

%Every exponent in the expansion corresponds to $1-s_0$ and its coefficient to $\mathrm{Res(\zeta,s_0)}$, where $s_0=\frac{1}{2}(1-r),\ r\in\mathbb N_0,$ are the poles of $\zeta(s)$. The residues of the integer $s_0$ are by \eqref{zeta} equal to $0$.
 
%Moreover, the expansion \eqref{eq:dist} is by \cite[Theorem 5.3.17]{fzf} even \emph{pointwise} expansion of $\varepsilon\mapsto \ell(O^{f_0}(x_0)_\varepsilon)$ up to the term $\varepsilon^1$ (\textbf{not included}), due to the rational estimates of its distance zeta function on vertical lines \textbf{strictly to the left} of the line $\{\re s=0\}$. 
\medskip

\begin{comment}
We prove in Subsection~\ref{subsec:racun} that for \emph{model germs}, which are time-one maps of simple vector fiels, due to their symmetries, the relation is explicit. Furthermore, in Theorem~\ref{thmA} in Section~\ref{sec:results}, we show that, for general parabolic germs, they are two distinct ways of \emph{smoothening} the oscillations in the standard length of the $\varepsilon$ neighborhood. In Theorem~\ref{thmA} we describe the exact oscillatory nature of coefficients of the asymptotic expansion of the standard length of $\varepsilon$-neighbohoods of parabolic germs. Then we provide explicit coefficients of both expansions expressed from these oscillatory coefficients. In this way, we relate the distributional asymptotic expansion that provides complex dimensions of the orbit with geometrically defined object that is the continuous length of the $\varepsilon$-neighborhood of the orbit. 
\end{comment}

\section{Auxiliary definitions}\label{sec:results}
Let $f\in\mathrm{Diff}(\mathbb R,0)$ be an attracting germ tangent to the identity, $$f(x)=x-ax^{k+1}+o(x^{k+1}),\ a>0,$$
and let $\mathcal{O}_f(x_0)$, $x_0>0$, be an orbit in the basin of attraction. %Let $$\varepsilon\mapsto V_f^{\rm{c}}(\varepsilon),\ \varepsilon\in(0,\varepsilon_0),$$ be the continuous time tube function, as introduced in \eqref{eq:ac}.

\medskip

We denote the \emph{fractal string} generated by the orbit $\mathcal{O}_f(x_0)$ by:
\begin{equation}\label{eq:fras}\mathcal L_{f}:=\mathcal L_{f,x_0}:=\{\ell_j:\,j\in\mathbb N\}.\end{equation}
Here, $\ell_j:=f^{\circ (j-1)}(x_0)-f^{\circ j}(x_0),\ j\in\mathbb N,$ are consecutive distances between points of the orbit. We will use the abbreviated notation $\mathcal L_f$ whenever there is no ambiguity.
By $\zeta_{\mathcal{L}_f}$ we denote the \emph{geometric zeta function} of the fractal string $\mathcal L_f$:
\begin{equation}\label{eq:zet}
\zeta_{\mathcal{L}_f}(s):=\sum_{j\in\mathbb N}\ell_j^s,\ \re s>\frac{k}{k+1}.
\end{equation}
Indeed, $f^{\circ j}(x_0)\sim j^{-\frac{1}{k}}$, $\ell_j\sim j^{-\frac{1}{k}-1}$, as $j\to\infty$, so the series \eqref{eq:zet} converges for $\re s>\frac{k}{k+1}$. For the precise definition of a fractal string and its geometric zeta function, see e.g. \cite{lapidusfrank,lapidusfrank12}.
\smallskip

We reiterate here more precisely the definition \eqref{zeta_f} of the distance zeta function of a an orbit $\mathcal{O}_f(x_0)$.
\begin{definition}[Distance zeta function of an orbit $\mathcal{O}_f(x_0)$]\label{distance_zeta_germ}
The \emph{distance zeta function} of the orbit $\mathcal{O}_f(x_0)$ is defined by:
\begin{equation}\label{eq:dist}
\zeta_f(s):=\int_0^{x_0} d(x,\mathcal{O}_f(x_0))^{s-1}\,dx,\ \re s>\dim_B\mathcal{O}_f(x_0).
\end{equation}
\end{definition}
It was proven in \cite{dtzf,fzf} that the integral \eqref{eq:dist} converges absolutely for $\re s>\dim_B\mathcal{O}_f(x_0)$, which is by e.g. \cite{NeVeDa} known to be equal to $\frac{k}{k+1}$ in the present parabolic case. Hence, it defines an analytic function on the open half-plane $\{\re s>\frac{k}{k+1}\}$. 

Definition \eqref{eq:dist} arises from \eqref{eq:dz} for $N=1$, where we let $A:=\mathcal{O}_f(x_0)\subseteq\mathbb{R}$. Moreover, we  integrate over the segment $[0,x_0]$ instead of over $A_\delta$, as was the case in \eqref{eq:dz}.
The reason behind it is that $\delta>0$ in \eqref{eq:dz} can be chosen arbitrarily. Therefore, we first choose $\delta$ large enough so that $[0,x_0]\subseteq A_{\delta}=(-\delta,x_0+\delta)$ (e.g. $\delta:=1$, since $x_0<1$). Then we discard the left- and right- end intervals of $A_{\delta}$, i.e, we discard $(-\delta,0)\cup(x_0,x_0+\delta)$, since these parts, on one hand, do not give any interesting fractal information about the set $A$ and, on the other hand, we obtain much nicer formulas and a more straightforward connection with the geometric zeta function $\zeta_{\mathcal{L}_f}$.\footnote{Note that the definition of $\zeta_f$ can be also explained in the more general context of {\em relative fractal drums} or in short {\em RFDs} and it is in fact the distance zeta function of the relative fractal drum $(\mathcal{O}_f(x_0),[0,x_0])$ in the terminology of \cite[Chapter 4]{fzf}; see also \cite[\S5.5.3 and Proposition 5.5.4]{fzf}.} It is straightforward to check the following functional equation connecting definition \eqref{eq:dist} to the standard definition of the distance zeta function $\zeta_{\mathcal O_f (x_0)}$ of $\mathcal O_f(x_0)$:
\begin{equation}\label{eqdelta}                                        
	\zeta_{\mathcal{O}_f(x_0)}(s)-\zeta_f(s)=\frac{2\delta^s}{s},\ \re s>\dim_B\mathcal{O}_f(x_0),
\end{equation}
whenever $\delta\geq (x_0-f(x_0))/2$.

\medskip 
By \cite[Example 2.1.58 and Corollary 2.1.61]{fzf}, $\zeta_{\mathcal{L}_f}$ and $\zeta_f$ are related by a simple functional equation which the reader can also check easily by direct integration:
\begin{equation}\label{func_dist_geo}
\zeta_f(s)=\frac{2^{1-s}}{s}\zeta_{\mathcal{L}_f}(s),\  \re s>\dim_B\mathcal{O}_f(x_0).
\end{equation}

For more on fractal zeta functions of compact sets and relative fractal drums, see e.g.\ \cite{lapidusfrank12,fzf}. 
 \smallskip

\begin{definition}[Complex dimensions]\label{def_dim_c}
Let $A\subset\mathbb R^N$ be bounded and assume that its distance zeta function $\zeta_{A}(s)$ has a meromorphic extension to all of $\Ce$.
Then the set of poles of $\zeta_{A}$ is called the {\em set of complex dimensions} of $A$ and denoted by $\dim_{\mathbb{C}}A$.
\end{definition}
\noindent If one cannot extend $\zeta_{A}(s)$ to all of $\Ce$ or wishes to work only with a subset of all of the complex dimensions of $A$, one uses the following terminology.

\begin{definition}[The window, the screen and the visible complex dimensions]\label{def_dim_w}
Let $A\subset\mathbb R^N$ be bounded and assume that $\zeta_{A}(s)$ has a meromorphic extension to the open right half-plane $W\subseteq \mathbb{C}$, which contains the half-plane of absolute convergence of $\zeta_A$.
Then the set of {\em visible complex dimensions} is defined as the set of poles of $\zeta_A$ contained in $W$ and is denoted by $\mathcal{P}(\zeta_A,W)$.

The open right half-plane $W$ is called the {\em window} and its left edge consisting of a vertical line is called the {\em screen}.\footnote{The window and the screen are usually defined in a more general way, but the above definition will suffice for our needs here; for more details see \cite[\S2.1.5]{fzf} or\cite{lapidusfrank12}.}
\end{definition}
By \cite{fzf}, under mild hypotheses on the set $A\subseteq \mathbb R^N$, the unique complex dimension with maximal real part is equal to $\dim_B A$, and its residue is equal to the Minkowski content of $A$, $\mathcal M(A)$, modulo a multiplicative constant.

\begin{remark}
In our context, Definitions \ref{def_dim_c} and \ref{def_dim_w} have their exact analogues if one replaces $\zeta_{\mathcal{O}_f(x_0)}$ by $\zeta_f$, since, by the functional equation \eqref{eqdelta}, the only changes that occur are in the pole at zero.
Furthermore, an analog comment applies to replacing $\zeta_f$ by $\zeta_{\mathcal{L}_f}$, in which case the order of the pole at zero changes and an explicit change occurs at the principal parts of other poles, but not in their orders, as can be seen from the functional equation \eqref{func_dist_geo}. Therefore, throughout the paper, we will use the notion of complex dimensions, the screen and the window in terms of $\zeta_f$.
\end{remark}
\smallskip

Next, we introduce the technical growth properties of the distance zeta function which are needed to reconstruct the pointwise or distributional asymptotics (introduced below at the end of this section) of the tube function $V_f(\varepsilon)$ directly from the complex dimensions of the orbit $\mathcal{O}_f(x_0)$. 
These are called {\em languidity} conditions in \cite{lapidusfrank12,fzf} (see especially \cite[Definitions 5.1.3 and 5.3.9]{fzf}).
In order to keep the presentation clear, here we impose stronger, but less technical conditions, called \emph{super languidity}, which will suffice for our needs.

\begin{definition}[Super languidity]\label{sup-languidity}
Let $\zeta_f(s)$ be the distance zeta function given by \eqref{eq:dist}. Let $W\subseteq\mathbb{C}$ be a window, with the corresponding screen $S=\partial W=\{\re s=\alpha\}$, for some $\alpha\in\mathbb{R}$ such that $\alpha<\dim_B\mathcal{O}_f(x_0)$. 

We say that $\zeta_f$ is {\em super languid} with the {\em super languidity exponent $\kappa_f\in\mathbb{R}$ for the screen} $S$ (or: {\em in the closed strip} $\{\alpha\leq \re s\leq\beta\}$) if there exists $\beta\in\mathbb{R}$, $\beta>\dim_B\mathcal{O}_f(x_0)$, such that
\begin{equation}
	|\zeta_f(s)|=O(|\im s|^{\kappa_f}),\quad |s|\to+\infty,
\end{equation}
uniformly in the closed strip $\{\alpha\leq \re s\leq\beta\}$.\footnote{It is straightforward to check that super languidity implies languidity in the sense of \cite[Definitions 5.1.3 and 5.3.9]{fzf}. Note also that in \cite{fzf}, for the distance zeta function, the notion is actually called \emph{$d$-languidity}, in order to distinguish when one refers to the languidity of the tube zeta function $\tilde{\zeta}_A$ as opposed to the languidity of the distance zeta function $\zeta_A$.}
\end{definition}
In other words, super languidity guaranties, at most, uniform polynomial growth in vertical strips (the case when $\kappa_f\geq 0$) or uniform rational decay (when $\kappa_f<0$), as the imaginary part grows by absolute value to $+\infty$.
The theory of \cite[Chapter 5]{fzf} then gives a way to reconstruct the pointwise asymptotics, in case of rational decay, or the distributional asymptotics, in case of polynomial growth, of the tube function $V_f(\varepsilon)$, as $\varepsilon\to0^+$.

In the above definition, the super languidity exponent $\kappa_f$ depends on the screen, i.e., on the value of $\alpha$ and on the value of $\beta$ but the dependence on $\beta$ is not so important because it can be controlled more easily.\footnote{Compare with \cite[Definitions 5.1.3 and 5.3.9]{fzf} where $\beta$ is refered to as the constant $c$.}
On the other hand, the dependence on $\alpha$ is crucial and usually not easy to obtain.
 
More specifically, $\kappa_f$ grows larger the more we push the screen to the left, i.e., as the value of $\dim_B\mathcal{O}_f(x_0)-\alpha$ becomes larger.
Furthermore, if one denotes the critical value of the screen, when the languidity exponent becomes zero, by $\{\mathrm{Re}(s)=\alpha_0\}$.
Then, all of the complex dimensions of $\mathcal{O}_f(x_0)$ that have real part strictly greater than $\alpha_0$ generate pointwise terms in the asympototics of $V_f(\varepsilon)$, as $\varepsilon\to 0^+$. The smaller the real part of a particular complex dimension, the higher is the asymptotic order of the term that the particular complex dimension generates in the asymptotics.
Furthermore, if we can find a screen $\{\mathrm{Re}(s)=\alpha_1\}$, $\alpha_1<\alpha_0$ (i.e., we push the left edge from $\alpha_0$ to $\alpha_1$), for which $\zeta_f$ is languid with $\kappa_f\geq 0$, then all of the complex dimensions of $\mathcal{O}_f(x_0)$ that have real part strictly grater than $\alpha_1$ and smaller than or equal to $\alpha_0$ will, in general, generate terms in the asymptotics of $V_f(\varepsilon)$, as $\varepsilon\to 0^+$, only \emph{in the distributional sense}. For the definition of the distributional asymptotics, see the end of the section. As we will show in the case of model parabolic germs, by combination of Proposition~\ref{prop:geo_k} and Theorem~A, some of distributional terms may be actually pointwise, see Remark~\ref{obs:rema} in Section~\ref{sec:cont}. This fact cannot be deduced directly from the theory of \cite{fzf}, but is shown here by exact computation of the pointwise and distributional asymptotics of $V_f(\varepsilon)$ in Theorem~A.

\begin{remark}\label{sup-lang-ex}
	The above definition extends directly to general distance zeta functions of bounded sets or even to general meromorphic functions defined on closed vertical strips.
	Moreover, we will often use the following facts that can be easily checked. Namely, if two meromorphic functions $\zeta_{1}$ and $\zeta_{2}$ are super languid on the same closed vertical strip $H\subset\Ce$ with exponents $\kappa_1$ and $\kappa_2$,
	then the function $\zeta_{1}\cdot\zeta_{2}$ is also super languid on the same strip $H$ with exponent $\kappa:=\kappa_1+\kappa_2$.
	Furthermore, if $a,b\in\Ce$, then the linear combination $a\zeta_1+b\zeta_2$ is super languid on the same strip $H$ with exponent $\kappa:=\max\{\kappa_1,\kappa_2\}$.
\end{remark}

Finally, let us recall the standard notion of the \emph{distributional asymptotics}. Let \begin{align*}\mathcal S(0,\delta)=\{\varphi\in C^\infty(0,\delta):\ &t^m \varphi^{(q)}(t)\to 0,\ t\to 0^+,\\
 &(t-\delta)^m\varphi^{(q)}(t)\to 0,\ t\to \delta^-,\ m\in\mathbb Z,\ q\in\mathbb N\}\end{align*} be the space of \emph{Schwartz test-functions} on the interval $(0,\delta)$.\footnote{Every derivative has faster than power decay at interval endpoints.} Its dual space of linear continuous functionals on $\mathcal S(0,\delta)$ is called the space of \emph{Schwartz distributions} and denoted by $\mathcal S'(0,\delta)$. Continuous functions $f\in C([0,\delta])$ may be considered as belonging to the class $\mathcal S'(0,\delta)$, by identifying $f$ with the functional $T_f$ given by $T_f(\varphi)=\left\langle T_f,\varphi\right\rangle:=\int_0^\delta f(t) \varphi(t)\,dt$, see \cite{Schw}.

Let $\varphi\in\mathcal S(0,\delta)$. We extend it by zero on the whole interval $(0,+\infty)$ and put $\varphi_a(t):=\frac{1}{a}\varphi(\frac{t}{a}),\ t\in(0,\delta)$. Then $\varphi_a\in\mathcal S(0,\delta)$, $0<a<1$. We say that $f\in C([0,\delta])$ is \emph{of asymptotic order $f(x)=O(x^\beta)$, $\beta\in\mathbb R$, as $x\to 0^+$, in the sense of Schwartz distributions}, if $$\langle T_f,\varphi_a\rangle =O(a^\beta),\ a\to 0^+,$$ for every Schwartz function $\varphi\in \mathcal S(0,\delta)$ extended by zero on $(0,+\infty)$, see \cite{lapidusfrank,fzf}.
By $\sim_{\mathcal D}$, we denote the asymptotic expansion \emph{in the sense of Schwartz distributions}.
It is straightforward to show that classical asymptotics of a continuous function induce distributional asymptotics of the corresponding regular distribution, but the converse is not true.

\begin{comment}
\bigskip 

Next, we summarize briefly the structure of the paper.

Proposition~\ref{prop:geo_k} in Subsection~\ref{subsec:model} below gives meromorphic extensions to whole $\mathbb C$ of distance zeta functions of orbits of \emph{model parabolic germs} from the formal class $(k,\rho=0)$, $k\in\mathbb N$, and explicitely gives their complex dimensions.
In Section~\ref{sec:cont} we apply well-known results from \cite{fzf} how one recovers, from complex dimensions of a set, the distributional expansion of the Lebesgue measure of the $\varepsilon$-neighborhood of the set, i.e., of its tube function, under super languidity conditions of polynomial type.
We apply these results to model examples (orbits of model parabolic germs in the formal class $(k,\rho=0)$) in Proposition~\ref{k=1_razvoj_p}, to obtain explicit distributional asymptotic expansions the tube functions of their orbits

Conversely, in Section~\ref{sec:general}, we prove Theorem C, a general result that recovers the distance zeta function of a set, if one knows the distributional asymptotics of its tube function.

\medskip
\end{comment}

\section{Proof of Proposition~\ref{prop:geo_k}}\label{sec:modelcomp}
%In this section we compute the distance zeta function for the model case of parabolic germs $f$ explicitly and obtain their complex dimensions, as well as languidity estimates of the distance zeta function.
%By the theory of \cite{fzf} (or \cite{lapidusfrank12}) this enables us to recover the asymptotics of the tube function $V_f(\varepsilon)$ as $\varepsilon\to 0^+$, pointwise in the first term and distributional after the first term.
First, in Subsection~\ref{subsec:dva}, we prove only the special case $k=1$ of Proposition~\ref{prop:geo_k}, while in Subsection~\ref{subsec:astring}, we prove Proposition~\ref{prop:geo_k} in full generality, i.e., for $k\geq 1$.
The main idea is using the \emph{shifted $a$-string method}. The case $k=1$ is indeed just a special case, however, we first describe and prove it in Subsection~\ref{subsec:dva} by a different and more intuitive method of meromorphic extension, which seems, unfortunately, not applicable in the general case.
Finally, in Subsection~\ref{subsec:maca}, we apply yet another method of meromorphic extension for $k=1$, the \emph{Mellin-Barnes method}, which has an advantage of giving better languidity estimates but also seems not to be easily generalized to the case $k\geq 1$.
\smallskip
      
\begin{remark}[Parabolic orbits asymptotically form $\frac{1}{k}$-strings]
For a parabolic germ $f$ of multiplicity $k\in\mathbb N$, the following asymptotics of iterates holds; see, e.g., \cite{NeVeDa,formal}:
$$
f^{\circ n}(x_0)\sim n^{-\frac{1}{k}},\ n\to\infty.
$$
The \emph{approximate}\footnote{The distances $n^{-\frac{1}{k}}-(n+1)^{-\frac{1}{k}}\sim n^{-1-\frac{1}{k}},\ n\to\infty,$ between the {\em approximate} consecutive points of the orbit form a standard $a$-string, while the true distances between consecutive points, i.e.,  $f^{\circ n}(x_0)-f^{\circ n+1}(x_0)$, form a \emph{shifted $\frac 1 k$-string}, which is introduced in Subsection~\ref{subsec:astring} and slightly generalizes the notion of the $\frac 1 k$-string from  \cite{lapidusfrank12}.} distances between consecutive points, as $n\to\infty$, are $n^{-\frac{1}{k}-1}$. They form an \emph{approximate $a$-string} of the $a$-string $\mathcal L:=\{n^{-1-\frac{1}{k}}:\ n\in\mathbb N\}$,  where $a=\frac1k$, as in \cite{lapidusfrank12}. The results about zeta-functions and distributional asymptotics for $\frac{1}{k}$-strings are already known from \cite{fzf,lapidusfrank12}. In \cite[Theorem 6.21]{lapidusfrank12}, the geometric zeta function of the $\frac1k$-string, $$\zeta_{\mathcal L}(s):=\sum_{j\in\mathbb N}n^{-s\frac{k}{k+1}},$$ was shown to converge absolutely in the half-plane $\{\re s>\frac{k}{k+1}\}$, and it was meromorphically extended to $\mathbb C$, thus revealing complex dimensions of the $\frac{1}{k}$-string  and their residues. 

Let $\Omega:=\{n^{-\frac{1}{k}}:\,n\in\mathbb N\}$. The geometric zeta function of $\frac{1}{k}$-string $\zeta_{\mathcal L}$ of $\mathcal L$ is then related to the distance zeta function $\zeta_{\Omega}$ of the set $\Omega$ by a simple formula \cite{fzf}:
$$
\zeta_\Omega(s)=\frac{2^{1-s}}{s}\zeta_{\mathcal L}(s)+\frac{2\delta^s}{s},\ \re s>\frac{k}{k+1},
$$
where one chooses $\delta>1$.

Moreover, it was shown, by the theory of geometric zeta functions and their languidity estimates, that the inner epsilon-neighborhood, $|\Omega_{\varepsilon}\cap[0,1]|$, admits the following distributional asymptotic expansion, see \cite[Example 8.1.2]{lapidusfrank12}: 
$$
|\Omega_{\varepsilon}\cap[0,1]|\sim_{\mathcal D} a_0(k)\varepsilon^{\frac{1}{k+1}}+\sum_{j=1}^{\infty} a_j(k)\cdot \mathrm{Res}\left(\zeta_{\mathcal L},-j\frac{k}{k+1}\right)\varepsilon ^{1+j\frac{k}{k+1}},\ \varepsilon\to  0^+,
$$
where $a_j(k)\in\mathbb R$, $j\in\mathbb N_0$.
\end{remark}
\medskip

In the rest of this section we compute precisely the distance zeta functions for orbits of model germs from the formal class $k\in\mathbb{N}$ and $\rho=0$, without approximating heuristically by the $\frac{1}{k}$-strings as above, but working directly with the exact fractal string generated by the orbit $\mathcal{O}_{f}(x_0)$. This brings us, in the proof of Proposition~\ref{prop:geo_k} in Subsection~\ref{subsec:astring}, to the method of \emph{shifted $a$-strings} introduced in this paper which slightly generalize the $a$-string method from \cite[\S6.5.1]{lapidusfrank12}. 

Furthermore, we meromorphically extend the distance zeta function of the orbit to all of $\mathbb{C}$, and use it later in Section~\ref{sec:cont} to obtain explicit formulas for the distributional expansion of its tube function. 
  
Note that the explicit computations shown in Sections~\ref{sec:modelcomp} and \ref{sec:cont} are elaborate even in the model cases. In the case of general tangent to the identity diffeomorphisms $f(x)=x+o(x)$, non-model and belonging to a general formal class $(k,\rho)$, $k\in\mathbb N,\ \rho\in\mathbb R$, we show in Theorem~B \emph{theoretically} the existence of the meromorphic extension of distance zeta functions of their orbits to all of $\mathbb C$. We also describe their complex dimensions in terms of coefficients appearing in the generalized asymptotics of their tube function $V_f(\varepsilon)$, established in Theorem~A.

\subsection{The case $k=1$}\label{subsec:dva}
Let $f_0\in\mathrm{Diff}(\mathbb R,0)$ be a model germ for $k=1$ and $\rho=0$:
\begin{equation}
	\label{eq:f(z)}
	f_0(x)=\frac{x}{1+x}=x(1-x+x^2+o(x^2)),\ x\to 0.
\end{equation}
Let $x_0\in(0,1)$ and let $\mathcal O_{f_0}(x_0)$ be the (attracting) orbit of $f_0$ with initial point $x_0$.
Let, as in \eqref{eq:fras}, $\mathcal L_{f_0}=\{\ell_j:\, j\in\mathbb N\}$ be the fractal string generated by $\mathcal{O}_f(x_0)$. Here,
\begin{equation*}
	\ell_j:=f_0^{\circ j}(x_0)-f_0^{\circ (j+1)}(x_0)=g_0(f_0^{\circ j}(x_0)),
\end{equation*}
where 
\begin{equation*}
	g_0(x):=x-f_0(x)=\frac{x^2}{1+x}.
\end{equation*}
One can easily obtain by induction that
\begin{equation*}
	f_0^{\circ j}(x)=\frac{x}{1+jx},\ j\in\mathbb N_0.
\end{equation*}
Consequently,
\begin{equation*}
	\ell_j=g\left(\frac{x_0}{1+jx_0}\right)=\frac{x_0^2}{(1+jx_0)(1+(j+1)x_0)}=\frac{1}{(j+x_0^{-1})(j+1+x_0^{-1})},\ j\in\mathbb N.
\end{equation*}
The geometric zeta function $\zeta_{\mathcal{L}_{f_0}}$ is now given as the Dirichlet series:
\begin{equation}
	\label{zeta_raw}
	\zeta_{\mathcal{L}_{f_0}}(s):=\sum_{j=0}^{\infty}\ell_j^s=\sum_{j=0}^{\infty}\frac{1}{(j+x_0^{-1})^s(j+1+x_0^{-1})^s}.
\end{equation}
The series is obviously absolutely convergent for all $s\in\Ce$ such that $\re s>1/2$, and divergent if $\re s\leq 1/2$, since the terms in the sum are asymptotic to $j^{-2s}$, as $j\to\infty$. 

By the theory of fractal zeta functions, the critical value $1/2$ is equal to the box dimension $\dim_B\mathcal{O}_{f_0}(x_0)$. Furthermore, $\zeta_{\mathcal{L}_{f_0}}$ is a holomorphic function in the open right-half plane $\{\re s>1/2\}$, see \cite[Theorem 1.10]{lapidusfrank12} or \cite[\S2.1.4]{fzf}.
\medskip

We first provide a heuristics for extending $\zeta_{\mathcal{L}_{f_0}}(s)$ meromorphically. Then we justify this heuristics in Proposition~\ref{geo_k1}, where we obtain the meromorphic extension of the geometric zeta function $\zeta_{\mathcal{L}_{f_0}}(s)$ to all of $\mathbb C$. As a direct consequence of the functional equation \eqref{func_dist_geo}, this implies the meromorphic extension of the distance zeta function $\zeta_{f_0}$ to whole $\mathbb C$ in Corollary~\ref{dist_zeta_1}.

\bigskip

\noindent \emph{A heuristics behind the meromorphic extension of $\zeta_{\mathcal{L}_{f_0}}$ to all of $\mathbb C$.}

\noindent Heuristically, we obtain the meromorphic extension of \eqref{zeta_raw} from $\{\re s>1/2\}$ to $\mathbb C$ by using the binomial series expansion. Introducing a shorthand
\begin{equation*}
	X:=x_0^{-1},
\end{equation*}
we have
\begin{equation*}
	\ell_j=\frac{1}{(j+X)(j+1+X)}=\frac{1}{(j+X)^2}\cdot\left(1+\frac{1}{j+X}\right)^{-1}.
\end{equation*}
Therefore,
\begin{equation}\label{eq:eljotes}
	\ell_j^s=\frac{1}{(j+X)^{2s}}\cdot\left(1+\frac{1}{j+X}\right)^{-s}=\sum_{m=0}^{\iy} {-s \choose m}\frac{1}{(j+X)^{2s+m}}.
\end{equation}
Heuristically, we recover \eqref{zeta} as:
\begin{equation}\label{zeta_l_hurwitz}
	\zeta_{\mathcal{L}_{f_0}}(s)=\sum_{j=0}^{\iy}\ell_j^s\sim\sum_{m=0}^{\iy} {-s \choose m}\zeta(2s+m,X),
\end{equation}
where $\zeta(s,q)$ is the Hurwitz zeta function introduced in \eqref{eq:hur}. Note that $\sim$ is not mathematically rigorous, due to formal change of the order of summation.
\medskip

We show in the following proposition that the above heuristics is indeed mathematically justified.

\begin{proposition}[Meromorphic extension of the geometric zeta $\zeta_{\mathcal{L}_{f_0}}(s)$ for $k=1$]\label{geo_k1}\
Let $f_0$ be the model germ as in \eqref{eq:f(z)}, and let $\mathcal O_{f_0}(x_0)$ be its orbit.
\begin{enumerate}
	\item For any fixed $M\in\eN$, the geometric zeta function $\zeta_{\mathcal{L}_{f_0}}(s)$ has a meromorphic extension to the open half plane $\{\re s >-M/2\}$, given by:
	\begin{equation}\label{fun_eq}
		\zeta_{\mathcal{L}_{f_0}}(s)=\sum_{j=0}^{\iy}\ell_j^s=\sum_{m=0}^{M} {-s \choose m}\zeta(2s+m,x_0^{-1})+g_{M+1}(s),
	\end{equation}
	where the function $g_{M+1}$ is holomorphic in the open half plane $\{\re s >-M/2\}$.
	Furthermore, for every $\gamma>0$ and every $s\in\mathbb{C}$ such that $\re s\geq-M/2+\gamma$,
	\begin{equation*}
		g_{M+1}(s)=O(|s|^{M+1}),\ |s|\to+\infty.
	\end{equation*}
All the poles of $\zeta_{\mathcal{L}_{f_0}}(s)$ in $\{\re s >-M/2\}$, $M\in\mathbb N$, are simple and located at:
	\begin{equation*}
		\omega_{l}=\frac{1}{2}-l,\quad l=0,1,2,\ldots,\floor{M/2}.
	\end{equation*}
\item The geometric zeta function $\zeta_{\mathcal{L}_{f_0}}(s)$ is super languid for any screen $S_{\alpha}=\{\re s =\alpha\}$ with $\alpha<\frac12$.
%in any closed vertical strip $\{\alpha\leq\re s\leq\beta\}$ where $\alpha,\beta\in\mathbb{R}$, $\alpha<\beta$. 
The upper bound on the super languidity exponent $\kappa_{\mathcal{L}_{f_0}}$ depends on $\alpha$.
More precisely, the holomorphic remainder $g_{M+1}$ is the dominating term in \eqref{fun_eq}, and hence the super languidity exponent is bounded from above by $M+1$, when $\alpha\in\left(-\frac{M}{2},-\frac{M-1}{2}\right]$.
% \edz{DODALA SAM OVO pod (2) tu. HTJELA BIH ANALOGIJU PROPOZICIJA ZA $k=1$ i $k>1$, da se tu prvo iskaze sve o zeta funkcijama, a onda u sljedecem poglavlju poveze s distrib. razvojima eps. okolina, bilo je zbrkano sve!} has at most polynomial growth along any vertical line and the growth rate becomes faster as the real part $\sigma=\re s$ grows in the negative direction.
% In particular, for $\re s>1/2$, the geometric zeta function $\zeta_{\mathcal{L}_{f_0}}(s)$ is bounded, and, for $\re s>0$, the growth is rational (with negative power).
\end{enumerate}
\end{proposition}
\begin{remark}
Proposition \ref{geo_k1} does not provide equality in \eqref{zeta_l_hurwitz}, since the remainder function $g_{M+1}(s)$ is $O(|s|^{M+1})$, $M\in\mathbb N$, as $|s|\to +\infty$.
	Therefore, \eqref{zeta_l_hurwitz} should be understood only as a useful \emph{formal} series from which one can extract the poles of $\zeta_{\mathcal{L}_{f_0}}(s)$ and their residues. 
\end{remark}
Using the functional equation \eqref{func_dist_geo} between geometric and distance zeta functions, we get the following direct corollary of Proposition \ref{geo_k1}:
\begin{corollary}[Meromorphic extension of the distance zeta function $\zeta_{f_0}(s)$ for $k=1$]\label{dist_zeta_1} Let $f_0$ be as in \eqref{eq:f(z)} and let $\mathcal O_{f_0}(x_0)$ be its orbit.
\begin{enumerate}

\item For any fixed $M\in\eN_0$, the distance zeta function $\zeta_{f_0}$ has a meromorphic extension to the open half-plane $\{\re s >-M/2\}$, given by:
	\begin{equation}\label{fun_eq_dist}
		\zeta_{f_0}(s)=\frac{2^{1-s}}{s}\sum_{m=0}^{M} {-s \choose m}\zeta(2s+m,x_0^{-1})+\tilde{g}_{M}(s),
	\end{equation}
	where the function $\tilde{g}_{M}$ is holomorphic in the open half plane $\{\re s >-M/2\}$.
	Furthermore, for every $\gamma>0$ and for every $s\in\mathbb{C}$ such that $\re s\geq-M/2+\gamma$,
	\begin{equation*}
		\tilde{g}_{M}(s)=O(|s|^{M}).
	\end{equation*}
For $M>1$, all of the poles of $\zeta_{f_0}$ in $\{\re s >-M/2\}$ are simple and located at
	\begin{equation}\label{pole0}
		\mathcal{P}_M:=\Big\{\omega_{l}=\frac{1}{2}-l,\quad l=0,1,2,\ldots,\floor{\frac{M}{2}}\Big\}\cup\{\omega_{1/2}:=0\},
	\end{equation}
	whereas in the special case $M=0$ we have $\mathcal P_0:=\{1/2\}$.

\item The distance zeta function $\zeta_{f_0}(s)$ is super languid  for any screen $S_{\alpha}=\{\re s=\alpha\}$, with $\alpha <\frac12$.
%in any closed vertical strip $\{\alpha\leq\re s\leq\beta\}$ where $\alpha,\beta\in\mathbb{R}$, $\alpha<\beta$. 
The upper bound on the super languidity exponent $\kappa_{{f_0}}$ depends on $\alpha$ and equals $M$ whenever $\alpha\in\left(-\frac{M}{2},-\frac{M-1}{2}\right]$, $M\in\mathbb N$.
\end{enumerate}
\end{corollary}

Note that statement (2) follows by Proposition~\ref{geo_k1} $(2)$, applying Remark \ref{sup-lang-ex} to the functional equation \eqref{func_dist_geo}. The factor $2^{1-s}$ does not affect languidity since it can be uniformly bounded by a constant inside any closed vertical strip.
Note also that the additional pole $\omega_{1/2}=0$ of the distance zeta function $\zeta_{f_0}$ arises only from the first term of \eqref{fun_eq_dist}. Indeed, it is cancelled by the binomials ${-s \choose m}$ for all $1\leq m\leq M$, as well as by the binomial ${-s\choose M+1}$, which is a factor of the remainder $\tilde{g}_M$, see \eqref{func_eq_bin}.
\bigskip

\noindent \emph{Proof of Proposition~\ref{geo_k1}}.\ 

$(1)$ Let us fix $M\in\eN$ and consider the $M$-th Taylor polynomial at $0$ of the parenthesis in $\ell_j^s$ from \eqref{eq:eljotes}:
	\begin{align}\begin{split}\label{taylor_ell}
		\ell_j^s=&\frac{1}{(j+X)^{2s}}\cdot\left(1+\frac{1}{j+X}\right)^{-s}=\\
		=&\frac{1}{(j+X)^{2s}}\left(\sum_{m=0}^{M} {-s \choose m}\frac{1}{(j+X)^{m}}+R_{M}\big((j+X)^{-1}\big)\right),
	\end{split}\end{align}
	where the remainder is given by its integral form:
	\begin{equation*}
		R_{M}\big((j+X)^{-1}\big)=(M+1){-s \choose M+1}\int_0^{(j+X)^{-1}}(1+t)^{-s-(M+1)}((j+X)^{-1}-t)^M\di t.
	\end{equation*}
	If we restrict $\re s>-M/2>-(M+1)$, we have that $|(1+t)^{-s-(M+1)}|\leq 1$ for all $t\geq 0$.
	Therefore, for $s\in\mathbb C$ such that $\re s>-M/2$,
	\begin{equation}\label{ostatak}
		|R_{M}((j+X)^{-1})|\leq \left|{-s \choose M+1}\right|\cdot (j+X)^{-(M+1)}.
	\end{equation}

Summing \eqref{taylor_ell} over all $j\geq 0$, we get:
	\begin{equation}\label{func_eq_bin}
	\begin{aligned}
		\sum_{j=0}^{\iy}\ell_{j}^s&=\sum_{m=0}^M{-s \choose m}\sum_{j=0}^{\iy}\frac{1}{(j+X)^{2s+m}}+\sum_{j=0}^{\iy}\frac{R_{M}((j+X)^{-1})}{(j+X)^{2s}}=\\
		&=\sum_{m=0}^M{-s \choose m}\zeta(2s+m,X)+\sum_{j=0}^{\iy}\frac{R_{M}((j+X)^{-1})}{(j+X)^{2s}}.
	\end{aligned}
	\end{equation}
	In order to justify the interchange of order of summation in the first sum, note that the first double sum on the right hand side in the above equation converges absolutely and uniformly on every compact subset of $\{\re s>1/2\}$. Therefore, it defines a holomorphic function in that half plane, which coincides with the first sum in the second line.
	The second sum above also converges absolutely and uniformly on every compact subset of $\{\re s>-M/2\}$ since, by \eqref{ostatak}, we have:
	\begin{equation}\label{f-ocjena}
		\sum_{j=0}^{\iy}\left|\frac{R_{M}((j+X)^{-1})}{(j+X)^{2s}}\right|\leq\left|{-s \choose M+1}\right|\cdot\sum_{j=0}^{\iy}\frac{1}{(j+X)^{2\re s+M+1}}.
	\end{equation}
By the Weierstrass M-test it also defines a holomorphic function in the open half plane $\{\re s>-M/2\}$. We have shown that the functional equation \eqref{func_eq_bin} holds in the open half plane $\{\re s>1/2\}$ and therefore, by the uniqueness of analytic continuation, it holds also in $\{\re s>-M/2\}$. We let:
	\begin{equation*}
		g_{M+1}(s):=\sum_{j=0}^{\iy}\frac{R_{M}((j+X)^{-1})}{(j+X)^{2s}}.
	\end{equation*}

	Take $\gamma>0$. From \eqref{f-ocjena}, on the half-plane $\{\re s\geq-M/2+\gamma\}$ it holds that: 
	\begin{equation*}
		|g_{M+1}(s)|\leq \left|{-s \choose M+1}\right|\cdot \sum_{j=0}^{\iy}\frac{1}{(j+X)^{1+2\gamma}}=O(|s|^{M+1}),\ |s|\to\infty.
	\end{equation*}
	The last statement about the poles follows now from \eqref{fun_eq}, from the fact that $g_{M+1}$ is holomorphic in the half plane $\{\re s >-M/2\}$ and the zero-pole cancellations in the case of integer poles discussed earlier. This concludes the proof of (1).
\smallskip
	
\emph{$(2)$} 
In light of equation \eqref{func_eq_bin}, we will use a result about the growth of the Hurwitz zeta function along vertical lines.
Namely, if we define
\begin{equation*}
	\mu(\sigma;q)=\overline{\lim}_{\tau\to\pm\iy}\frac{\log|\zeta(\sigma+\I\tau,q)|}{\log|\tau|},
\end{equation*}
where $q>0$ is arbitrary, then, by \cite[Lemma 2]{katsurada} we have the following bounds:
\begin{equation}\label{hurwitz_bound}
		\mu(\sigma;q)\leq\begin{cases}
	                      	1/2-\sigma& \textnormal{ if }\sigma\leq 0,\\
	                      	\frac{1}{2}(1-\sigma)&\textnormal{ if } 0\leq\sigma\leq 1,\\
	                      	0&\textnormal{ if }\sigma\geq 1.
	                      \end{cases}
\end{equation}

The above bounds on $\mu(\sigma;q)$ and equation \eqref{func_eq_bin} imply that the geometric zeta function has at most polynomial growth along any vertical line and the growth rate becomes faster as the real part $\sigma=\re s$ becomes larger in the negative direction.
Also, note that, for $\re s>1/2$, the geometric zeta function $\zeta_{\mathcal{L}_{f_0}}(s)$ is bounded, a fact that follows directly from \eqref{zeta_raw}.

Finally, the above bounds also imply trivially that the Hurwitz zeta function is super languid in any closed vertical strip. Its super languidity coefficient is bounded by $\mu(\alpha,q)$, where $\alpha\in\mathbb{R}$ is the abscissa of the left edge of the strip. 
By Remark \ref{sup-lang-ex}, along with close inspection of the terms in equation \eqref{func_eq_bin}, we have that $\zeta_{\mathcal{L}_{f_0}}$ is also super languid in any closed vertical strip.
Namely, for fixed $M\in\mathbb{N}_0$ and fixed left edge abscissa $\alpha\in\left(-\frac{M}{2},-\frac{M-1}{2}\right]$, all of the terms ${-s \choose m}\zeta(2s+m,x_0^{-1})$ for $m=0,\ldots,M-1,$ are super languid with exponent $\frac{1}{2}-\alpha$, while the term ${-s \choose M}\zeta(2s+M,x_0^{-1})$ is super languid with exponent $\frac{1+M}{2}-\alpha$.\footnote{See the proof of Theorem \ref{shifted_a_prop} for a more detailed exposition of determining the super languidity exponent bounds in a similar type of the sum.}
Hence, the uniform bound on the super languidity exponent for all these terms is $M+\frac{1}{2}$, which is dominated by the bound $M+1$ on the super languidity exponent of the remainder $g_{M+1}$.
\qed
\medskip

\subsection{The general case $k\geq 1$. The \emph{shifted $a$-string} method.}\label{subsec:astring}\

Let $f_0\in\mathrm{Diff}(\mathbb R,0)$ be a germ of tangent to the identity model diffeomorphism of the formal class $(k,0)$, $k\in\mathbb N$:
$$
f_0(x):=\mathrm{Exp}\Big(-x^{k+1}\frac{d}{dx}\Big).\mathrm{id}.
$$
We get:
\begin{equation}
	\label{eq:f(z)k}
	f_0(x)=\frac{x}{(1+kx^k)^{1/k}}.
\end{equation}
By induction,
\begin{equation*}\label{eq:fk}
	f_0^{\circ j}(x)=\frac{x}{(1+jkx^k)^{1/k}},\ j\in\eN_0.
\end{equation*}

Let $\mathcal O_{f_0}(x_0)$ be its (attracting) orbit, $0<x_0<1$, and let $\mathcal L_{f_0}=\{\ell_j:j\in\mathbb N_0\}$ be the fractal string generated by its orbit. Then
\begin{equation}\label{eq:fsk}
\ell_j:=f_0^{\circ (j+1)}(x_0)-f_0^{\circ j}(x_0)=k^{-1/k}\Big(\big(j+\frac{1}{kx_0^k}\big)^{-1/k}-\big(j+1+\frac{1}{kx_0^k}\big)^{-1/k}\Big),\ j\in\mathbb N_0.
\end{equation}
Hence, its geometric zeta function is given by: 
\begin{equation}\label{eq:zz}
\zeta_{\mathcal{L}_{f_0}}(s)=\sum_{j=0}^{\infty}\ell_j^s=k^{-\frac{s}{k}}\sum_{j=0}^{\infty}\Big(\big(j+\frac{1}{kx_0^k}\big)^{-1/k}-\big(j+1+\frac{1}{kx_0^k}\big)^{-1/k}\Big)^s,\ \re s>\frac{k}{k+1}.
\end{equation}
The half-plane of convergence is  $\{\re s>\frac{k}{k+1}\}$, since $\ell_j\sim j^{-\frac{1}{k}-1}$, $j\to\infty$.
\bigskip

\begin{proof}[Proof of Proposition~\ref{prop:geo_k}]\

$(1)$ We interpret the fractal string $\mathcal L_{f_0}$ given by \eqref{eq:fras} with $\ell_j$, $j\in\mathbb N_0$, as in \eqref{eq:fsk} as a special case of a \emph{shifted $\frac{1}{k}$-string}; see \cite[\S6.5.1]{lapidusfrank12}.
Indeed, the $a$-string for $a>0$ is defined as a fractal string with lengths:
\begin{equation*}
	\ell_j:=j^{-a}-(j+1)^{-a},\ j\in\mathbb N.
\end{equation*}
For some $b>0$, we define the \emph{shifted $a$-string} $\mathcal L_{a,b}$ as a fractal string with lengths:
\begin{equation}
	\label{shift-a}
	\ell_j(b):=(j+b)^{-a}-(j+1+b)^{-a},\ j\in\mathbb N_0.
\end{equation}
Note that, for $b=1$, we get the standard $a$-string. 

Put $b_k:=\frac{1}{kx_0^k}$. Then, $\ell_{j}$ from \eqref{eq:fsk} can be written as $\ell_j=k^{-1/k}\ell_j(b_k)$, $j\in\mathbb N_0$. Therefore, \eqref{eq:zz} becomes:
\begin{equation*}
	\zeta_{\mathcal{L}_{f_0}}(s)=k^{-\frac{s}{k}}\sum_{j=0}^{\infty}\ell_j(b_k)^s,\ \re s>\frac{k}{k+1}.
\end{equation*}
Now the statement of Proposition~\ref{prop:geo_k} follows immediately from the more general Theorem~\ref{shifted_a_prop} about meromorphic extension of geometric zeta functions of shifted $a$-strings, which is stated and proven in Subsection~\ref{subsec:sas} of the Appendix, combined with the functional equation \eqref{func_dist_geo} relating the distance zeta function to the geometric zeta function.
\smallskip

$(2)$ By Theorem~\ref{shifted_a_prop} in the Appendix, on closed vertical strips with left edge abscissa $\alpha\in\left(-\frac{Mk}{k+1},-\frac{(M-1)k}{k+1}\right]$, $\zeta_{\mathcal{L}_{f_0}}(s)$ is super languid with exponent $\kappa_{\mathcal{L}_{f_0}}=M+1$, since the holomorphic remainder $k^{-\frac{k+1}{k}s}R(s)$ in \eqref{zetaab_W} dominates the rest of the terms. Hence, by the functional equation \eqref{func_dist_geo} and Remark \ref{sup-lang-ex}, the distance zeta function $\zeta_{f_0}$ is languid with exponent $\kappa_f=M$.
\end{proof}

\subsection{The method by \emph{Mellin-Barnes integral formula}}\label{subsec:maca}\

In this subsection we construct, by another method, the meromorphic extension of the geometric zeta function $\zeta_{\mathcal{L}_{f_0}}(s)$ for the model diffeomorphism $f_0$ of multiplicity $k=1$ given by \eqref{eq:f(z)}.
This method gives better languidity estimates than Proposition~\ref{geo_k1}, which guarantee that the first term of the asymptotics of the tube function $\varepsilon\mapsto V_{f_0}(\varepsilon)$, as $\varepsilon\to 0^+$, in \eqref{k=1_razvoj_f} in Section~\ref{sec:cont} is not only distributional, but pointwise.
It seems that this method can be generalized to the case $k>1$ (by induction), in order to get better languidity estimates than by the shifted $a$-string method used in the proof Proposition~\ref{prop:geo_k}. However, we omit it, since it involves tedious computer software computations without any conceptual gain, as stated in Remark~\ref{obs:rema}.

The method is a slight modification of the method used in \cite{matzu} and arises from the classical \emph{Mellin-Barnes formula}:
\begin{equation}\label{eq:MB}
	(1+\lambda)^{-s}=\frac{1}{2\pi\I\Gamma(s)}\int_{(-c)}\Gamma(s+z)\Gamma(-z)\lambda^z\di z,
\end{equation}
valid for $|\arg \lambda|<\pi$, $\lambda\neq 0$, $0<c<\re s$. The path of integration, denoted by $(-c)$, is the vertical line $\{\re z=-c\}$.

Let $f_0$ be a model parabolic germ of multiplicity $k\in\mathbb N$ given by \eqref{eq:f(z)k}, and let $\mathcal O_{f_0}(x_0)$ be its orbit. Recall from \eqref{eq:zz} that $\zeta_{\mathcal L_{f_0}}(s)=\sum_{j=0}^{\infty}\ell_{j}^s$, where
\begin{equation*}\label{ell_kj}
	\ell_{j}=f^{\circ j}_0(x_0)-f^{\circ (j+1)}_0(x_0)=\frac{k^{-\frac{1}{k}}}{B_{k,j}^{\frac{k+1}{k}}\sum_{l=1}^k(1+B_{k,j}^{-1})^{\frac{k+1-l}{k}}},\ j\in\mathbb N_0,
\end{equation*}
and
\begin{equation*}
	B_{k,j}:=j+\frac{1}{kx_0^k},\ j\in\mathbb N_0.
\end{equation*}
\medskip

In the case $k=1$, by the Mellin-Barnes formula \eqref{eq:MB}, we get:
\begin{equation*}
	\ell_j^s=B_{1,j}^{-2s}(1+B_{1,j}^{-1})^{-s}=\frac{1}{2\pi\I\Gamma(s)}\int_{(-c)}\Gamma(s+z)\Gamma(-z)B_{1,j}^{-2s-z}\di z,\ \mathrm{Re}\, s>c.
\end{equation*}
We take $c>0$ to be a small positive number, $c\approx 0^+$. Summing over $j\in\mathbb N_0$, we obtain:
\begin{equation}\label{mb_k1}
	\zeta_{\mathcal{L}_{f_0}}(s)=\frac{1}{2\pi\I\Gamma(s)}\int_{(-c)}\Gamma(s+z)\Gamma(-z)\zeta(2s+z,x_0^{-1})\di z,\ \mathrm{Re}\, s>\frac{c+1}{2},
\end{equation}
where the interchange of integration and summation is easily justified by the dominated convergence theorem applied to the sequence of partial sums in the open half plane $\{\re s >(1+c)/2\}$.

To get a meromorphic continuation of $\zeta_{\mathcal{L}_{f_0}}(s)$ to $\mathbb C$, one shifts the path of integration in the integral on the right hand side of \eqref{mb_k1} to the right, i.e., to the vertical line $\{\re s=M-\eta\}$, with $M\in\eN$ and $\eta>0$ a small positive number.
The shifted integral is well-defined since, for $\re s >(1+c)/2$ and $\re z>-c$, $\zeta(2s+z,x_0^{-1})$ is bounded and the product of the gamma functions in the integrand of \eqref{mb_k1} decays exponentionally when $|\im z|\to\infty$ uniformly in any vertical strip of finite width contained in the open half-plane $\{\re z>-c\}$.
Furthermore, the only residues of the integrand that are contained in the vertical strip $\{-c<\re z<M-\eta\}$ arise from the factor $\Gamma(-z)$ and are exactly equal to $0,1,2\ldots,M-1$. By the residue theorem, we get, in the larger open half-plane $\left\{\re s> -\frac{M-1}{2}+\frac{\eta}{2}\right\}$, 
\begin{equation*}\label{zeta_1_matzu}
\begin{aligned}
	\zeta_{\mathcal{L}_{f_0}}(s)&=\sum_{n=0}^{M-1}{-s\choose n}\zeta(2s+n,x_0^{-1})+I_{M-\eta}(s),
\end{aligned}
\end{equation*}
where the right-hand side is meromorphic in $\left\{\re s> -\frac{M-1}{2}+\frac{\eta}{2}\right\}$. Indeed, the integral
\begin{equation}\label{eq:I}
	I_{M-\eta}(s):=\frac{1}{2\pi\I\Gamma(s)}\int_{(M-\eta)}\Gamma(s+z)\Gamma(-z)\zeta(2s+z,x_0^{-1})\di z
\end{equation}
converges absolutely exactly for such $s\in\mathbb C$.
Shifting further in this way with increasing $M\to+\infty$, we get a meromorphic extension to all of $\Ce$.
\medskip

\emph{The languidity estimates by Mellin-Barnes.}\ The advantage of this approach in deducing languidity estimates is that we have more information on the holomorphic remainder, which is now given as an integral \eqref{eq:I}. Moreover, we are not confined to successive meromorphic extensions by half-planes with integer steps, as it was the case in the proof of Proposition~\ref{geo_k1}, but rather with \emph{continuous} steps controlled by $\eta>0$.

We estimate the remainder \eqref{eq:I} by using the technique described in \cite[\S2]{matzu}.
Namely, let $s=\sigma+\I\tau$ and $-\frac{M-1}{2}+\frac{\eta}{2}<\sigma<\Theta$ for some fixed $\Theta$. Then, by Stirling's formula, for large $|s|$, firstly we bound the integral remainder $I_{M-\eta}(s)$ in this strip by a positive constant multiplied by
\begin{equation}\label{est_matz}
	(1+|\tau|)^{-\sigma+\frac{1}{2}}\E^{\frac{\pi}{2}|\tau|}\int_{-\infty}^{+\infty}(1+|\tau+x|)^{\sigma-\frac{1}{2}+M-\eta}(1+|x|)^{-M+\eta-\frac{1}{2}}\E^{-\frac{\pi}{2}(|\tau+x|+|x|)}\di x.
\end{equation}
Indeed, $\zeta(2s+z,x_0^{-1})=O(1)$, as $|s|\to\infty,$ since we are restricting the range of $\sigma$, which assures that $\re(2s+z)>1$.
Next, we apply the technical result \cite[Lemma 2]{matzu} to the above integral with $u:=\tau$, $y:=x$, $p:=\sigma-\frac{1}{2}+M-\eta$, $q:=0$, $r:=-M+\eta-\frac{1}{2}$ in the notation of \cite[Lemma 2]{matzu}, which gives us the uniform estimate of integral from \eqref{est_matz} in the strip $-\frac{M-1}{2}+\frac{\eta}{2}<\sigma<\Theta$ as:
\begin{equation*}
	O\left((1+(1+|\tau|)^p)(1+(1+|\tau|)^{r+1})\E^{-\frac{\pi}{2}|\tau|}\right),\ |\tau|\to\infty.
\end{equation*}
Combining the above estimate with the factor in front of the integral in \eqref{est_matz} and putting $L:=M-\eta$, we obtain the uniform estimate of \eqref{eq:I} in the strip as:
\begin{equation*}
	I_{L}(s)=O\left(\max\left\{(1+|\tau|)^Q:Q\in\left\{\frac{1}{2},L,-\sigma+\frac{1}{2},-\sigma+1-L\right\}\right\}\right),
\end{equation*}
as $|\tau|\to\infty$. 
Note that the choice of $M\in\eN$ means that we have $L\in(M-1,M)$, depending on the choice of $\eta>0$.

Taking $M=1$, $L\in(0,1/2)$, and choosing for the screen the vertical line $\re s=\tilde{\sigma}$, where $\tilde{\sigma}\in(0,1/2)$, we obtain that $I_{L}(s)=O((1+|\tau|)^{\tilde{Q}})$, where $\tilde{Q}<1$.
This implies an overall languidity estimate of the corresponding distance\footnote{Recall that, in order to obtain the distance zeta function, we just multiply the geometric zeta function by the factor $\frac{2^{1-s}}{s}$, where the term $2^{1-s}$ is bounded by a constant in any closed vertical strip of finite width.} zeta function $\zeta_{\mathcal O_{f_0}(x_0)}(s)$ of the type $O\big((1+|\tau|)^{\tilde{Q}-1}\big)$. That is, we have rational decay when $|s|\to\infty$ in the half-plane ${\re s>\tilde{\sigma}}$. Therefore, since the only pole in this half-plane is $\omega=1/2$, the asymptotic formula \eqref{k=1_razvoj_f} for the asymptotic expansion of the tube function $V_{f_0}(\varepsilon)$, as $\varepsilon\to 0^+$, is valid not only distributionally, but also \emph{pointwise in the first term}. For details, see the next Section~\ref{sec:cont}.
Also, note that, as soon as we take $M=2$, we fail to get a pointwise term, since $L>1$. However, in \cite[Theorem B]{MRRZ3}, it is shown directly that the asymptotic formula \eqref{k=1_razvoj_f} actually holds pointwise in the first two terms. 
\medskip

\section{From complex dimensions to distributional expansions of the tube functions}\label{sec:cont}

It is known from \cite[Theorems 5.3.16, 5.3.17, 5.3.21]{fzf} or \cite[Theorems 8.1 and 8.7]{lapidusfrank12}; see also \cite[\S5.5.2]{fzf}, that, under appropriate \emph{languidity estimates} on the growth of the meromorphic extension of the geometric or distance zeta function of a fractal string along closed vertical strips, there exists an explicit formula that recovers, from the complex dimensions of the set and their residues, the asymptotic expansion of its tube function, i.e., the Lebesgue measure of the $\varepsilon$-neighborhood of the set, as $\varepsilon\to 0^+$. Moreover, in case of the distance zeta function, while the languidity estimates are \emph{rational} (with negative exponents), the expansion is pointwise, and it becomes a distributional expansion when the languidity estimates become \emph{polynomial}.

See \cite[\S5.3]{fzf} for recovering the asymptotics of the tube function from the distance zeta function and \cite[\S8]{lapidusfrank12} or \cite[\S5.5.2]{fzf} for recovering the asymptotics from the geometric zeta function of the fractal string.
Although we could use \cite[Theorems 8.1 and 8.7]{lapidusfrank12} and work with the geometric zeta function directly, here we choose to use \cite[Theorems 5.3.16, 5.3.21]{fzf}, since then the explicit formulas for the asymptotics are nicer and less technical.
Namely, if one uses the geometric zeta function and the explicit formulas obtained in \cite[Theorems 8.1 and 8.7]{lapidusfrank12}, one has to take special care about the pole at zero, i.e., the complex dimension $0$, whereas if one uses the the more general distance zeta function from \cite{fzf}, the pole at zero does not need special attention.
\smallskip

In particular, in this Section we will apply \cite[Theorem 5.3.21]{fzf} to meromorphic extensions of distance zeta functions of orbits of model germs $f_0\in\mathrm{Diff}(\mathbb R,0)$ from the formal classes $(k,\rho=0)$, $k\in\mathbb N$, to recover the distributional asymptotics of the tube function $V_{f_0}(\varepsilon)$, as $\varepsilon\to 0^+$. Recall that the complex dimensions and languidity estimates are given in Proposition~\ref{geo_k1} and Corollary~\ref{dist_zeta_1} (in the case $k=1$) and Proposition~\ref{prop:geo_k} (in the case $k\geq 1$), and that languidity estimates are improved for the case $k=1$ by using the Mellin-Barnes formula in Subsection~\ref{subsec:maca}.

Note that it was proven in \cite{formal} and \cite[Theorem B]{MRRZ3} that, for every parabolic $f\in\mathrm{Diff}(\mathbb R,0)$, the tube function $V_f(\varepsilon)$ admits the \emph{pointwise} asymptotic expansion  in the power-logarithm scale \emph{up to the term $O(\varepsilon^{2-\frac{1}{k+1}})$} as $\varepsilon\to 0^+$, and a complete description of this expansion was provided.

Note also that in the terminology (and notation) of \cite[\S4]{fzf}, the choice of the definition \eqref{zeta_f} of the distance zeta function $\zeta_f$ that we made in this paper means that we are actually working with the the {\em relative distance zeta function} $\zeta_{A,\Omega}$ of the \emph{relative fractal drum} $(A,\Omega)$ (or RFD, in short), where the pair of sets $(A,\Omega)$ is given as  $A:=\mathcal{O}_{f_0}(x_0)$ and $\Omega:=[0,x_0]$.
In other words, by definition, we have $\zeta_f(s)=\zeta_{\mathcal{O}_{f_0}(x_0),[0,x_0]}$, $\re s>\dim_B\mathcal{O}_{f_0}(x_0)$.\footnote{One also has to choose $\delta>0$ large enough in \cite[Definition 4.1.1]{fzf} of $\zeta_{A,\Omega}$, i.e., $\delta>(x_0-f_0(x_0))/2$.}
Furthermore, our definition of the tube function $V_f$ is such that  \begin{equation*}
	V_f(\varepsilon)=|(\mathcal{O}_{f_0}(x_0))_{\varepsilon}\cap [0,x_0]|=|(\mathcal{O}_{f_0}(x_0))_{\varepsilon}|-2\eps,
\end{equation*} which is also known in \cite{fzf,lapidusfrank} as the {\em inner $\varepsilon$-neighborhood} of the fractal string generated by the orbit $\mathcal{O}_{f_0}(x_0)$. We are just removing the part of $\varepsilon$-neighborhood to the left of 0 and to the right of $x_0$, which always equals to two intervals of length $\varepsilon$, so there is no actual loss of generality here.
The two definitions, of $\zeta_f$ and $V_f$, are therefore perfectly tailored so that we can use directly \cite[Theorem 5.3.21]{fzf} in order to recover the asymptotics of $V_f(\varepsilon)$ as $\varepsilon\to 0^+$ from the poles of $\zeta_f$.
The reason why we have chosen to introduce such definitions of $\zeta_f$ and $V_f$ is aesthetical, in order to keep the presentation of the paper as clear as possible.  
\medskip

We state here the result in the case $k=1$. For the case $k>1$, the same can be done using Proposition~\ref{prop:geo_k} instead of  Corollary~\ref{dist_zeta_1}, but the expressions for residues, that is, for coefficients of the distributional expansion, are more complicated, so we omit it.
\begin{proposition}\label{k=1_razvoj_p}
	Let $f_0$ be a model parabolic germ with formal invariants  $k=1$ and $\rho=0$. For any $M\in\eN$, the function $\varepsilon\mapsto V_{f_0}(\varepsilon)$, $\varepsilon\in(0,\varepsilon_0)$, has the following complete distributional asymptotic expansion, as $\varepsilon\to 0^+$:
	\begin{equation}\label{k=1_razvoj_f}
	\begin{aligned}
		V_{f_0}(\varepsilon)&=_{\mathcal D}\sum_{l=0}^{\floor{M/2}}\frac{2^{\frac{3}{2}+l}}{1-4l^2}{l-\frac{1}{2}\choose 2l}\eps^{\frac{1}{2}+l}-\frac{2}{x_0}\eps+O(\eps^{1+\frac{M}{2}}).
	\end{aligned}
	\end{equation}
Moreover, the expansion is pointwise in the first term.
\end{proposition}

\begin{proof}
From Proposition~\ref{geo_k1} and  Corollary~\ref{dist_zeta_1}, we get the complex dimensions of $\mathcal O_{f_0}(x_0)$ and their residues, as well as polynomial languidity estimates of the meromorphic extension of $\zeta_{f_0}$ along screens which are vertical lines. This is sufficient for the application of \cite[Theorem 5.3.21]{fzf}, to deduce a distributional fractal tube formula for $V_{f_0}(\varepsilon)=|(\mathcal{O}_{f_0}(x_0)){\varepsilon}\cap [0,x_0]|$. By \cite[Theorem 5.3.21]{fzf}, we have:
	\begin{equation*}
	\begin{aligned}
		V_{f_0}(\varepsilon)&=_{\mathcal D}\sum_{\omega_l\in \mathcal{P}_M}\res\left(\frac{\eps^{1-s}}{1-s}\zeta_{f_0}(s),\omega_l\right)+O(\eps^{1+\frac{M}{2}})=\\
		&=_{\mathcal D}\sum_{\omega_l\in\ \mathcal{P}_M}\frac{\eps^{1-\omega_l}}{1-\omega_l}\res\left(\zeta_{f_0}(s),\omega_l\right)+O(\eps^{1+\frac{M}{2}})=\\
		&=_{\mathcal D}\sum_{l=0}^{\floor{M/2}}\frac{2^{\frac{3}{2}+l}}{1-4l^2}{l-\frac{1}{2}\choose 2l}\eps^{\frac{1}{2}+l}-\frac{2}{x_0}\eps+O(\eps^{1+\frac{M}{2}}).
	\end{aligned}
	\end{equation*}
	Here, $\mathcal{P}_M$ is the set of complex dimensions with real part larger than or equal to $\frac{1}{2}-\lfloor \frac{M}{2}\rfloor$, as in \eqref{pole0}. 

Moreover, the better (rational) languidity estimates provided in Subsection~\ref{subsec:maca} prove that the expansion is, in addition, \emph{pointwise} in the first term.
\end{proof}

\begin{remark}\label{obs:rema}
The above asymptotic formula \eqref{k=1_razvoj_f} is only distributional. In order to show that it is also valid pointwise up to finitely many first terms, we need better (rational) languidity estimates. Note that, for $k=1$, our estimates give us that the asymptotics \eqref{k=1_razvoj_f} is pointwise in the first term. Rational languidity estimates are obtained in Subsection \ref{subsec:maca} by using an adaptation of the Mellin-Barnes method from \cite{matzu}, but we have not generalized it to $k>1$. For $k>1$, using the $a$-string method, we cannot obtain rational languidity estimates, hence the formula is distributional. 
However, from \cite[Theorem B]{MRRZ3} we know that the asymptotics of the tube function $V_{f_0}(\varepsilon)$ is pointwise all the way up to the term $O(\varepsilon^{2-\frac{1}{k+1}})$, which becomes oscillatory. In the special case \eqref{k=1_razvoj_f} the expansion is therefore \emph{pointwise in the first two terms} (up to $\varepsilon^{\frac{3}{2}}$). The equality of the pointwise and the distributional asymptotic expansions up to the term where pointwise asymptotics still exists is guaranteed by the uniqueness of the distributional power-logarithmic asymptotic expansions; see, e.g., \cite{Schw,Bre}.
\end{remark}

%\subsubsection{The case when $k=2$ and the general case}

%We will now demonstrate the Mellin-Barnes method in the case when $k=2$ and comment on the general case.
%By \eqref{ell_kj} we have
%\begin{equation}
%\begin{aligned}
	%\ell_{2,j}^s&=2^{-\frac{s}{2}}B_{2,j}^{-\frac{3}{2}s}\left((1+B_{2,j}^{-1})+(1+B_{2,j}^{-1})^{\frac{1}{2}}\right)^{-s}\\
	%&=2^{-\frac{s}{2}}B_{2,j}^{-\frac{3}{2}s}(1+B_{2,j}^{-1})^{-s}\left(1+(1+B_{2,j}^{-1})^{-\frac{1}{2}}\right)^{-s}.
%\end{aligned}
%\end{equation}
%Next we apply the Mellin-Barnes formula to the last parenthesis above to obtain
%\begin{equation}
%\begin{aligned}
	%\ell_{2,j}^s&=2^{-\frac{s}{2}}B_{2,j}^{-\frac{3}{2}s}\int_{(-c_1)}\frac{\Gamma(s+z_1)\Gamma(-z_1)}{2\pi\I\Gamma(s)}(1+B_{2,j}^{-1})^{-\left(s+\frac{z_1}{2}\right)}\di z_1\\
	%&=2^{-\frac{s}{2}}\int_{(-c_1)}\frac{\Gamma(s+z_1)\Gamma(-z_1)}{2\pi\I\Gamma(s)}\int_{(-c_2)}\frac{\Gamma(s+\frac{z_1}{2}+z_2)\Gamma(-z_2)}{2\pi\I\Gamma(s+\frac{z_1}{2})}B_{2,j}^{-\frac{3}{2}s+z_2}\di z_2\di z_1
	%\end{aligned}
%\end{equation}
%itd.... Puno pisanja, računa bez neke prevelike koristi, ako želimo sve detaljno napisati i korektno opravdati  - mislim da trebamo razmisliti ima li to smisla dalje raspisivati ili samo ostaviti kao napomenu da se može napraviti. Ja bih možda izbacio, a izbacio bih i heuristiku za $k\geq 2$ jer je na kraju ta formula s $k$ suma poprilično beskorisna.

\section{Proof of Theorem~A}\label{sec:proofA}

\noindent This section is dedicated to the proof of Theorem~A. We first introduce some definitions that will be used in the proof.
\medskip

Let $G:[0,+\infty)\to\mathbb R$ be the periodic function of period $1$ defined on $[0,1)$ by:
\begin{equation}\label{eq:ge}
G(s):=\begin{cases} 0,& s=0,\\
1-s,& s\in(0,1).\end{cases}
\end{equation}
It is discontinuous, periodic and bounded on $[0,+\infty)$.
\begin{definition}\label{def:IBP}
A function $\tilde H:[0,+\infty)\to\mathbb R$ is called \emph{integrally bounded periodic} of period $1$ if it is periodic of period 1, bounded and if its primitive 
$$
P_{\tilde H}(t):=\int_0^t \tilde H(s)\,ds,\ t\in[0,\infty),
$$
is again periodic of period $1$ and bounded.
\end{definition}

\begin{proposition}[Integral normalization]\label{prop:intnorm} Let $H:[0,+\infty)\to\mathbb R$ be a periodic function of period $1$ and bounded on its period. Let $$\tilde H(t):=H(t)-\int_0^1 H(s)\,ds,\ t\in[0,+\infty),$$
be its \emph{integral normalization}. Then $\tilde H$ is integrally bounded periodic.
\end{proposition}
\begin{proof}
Elementary, by direct integration.
\end{proof}
\emph{Sketch of the proof of Theorem A.} In Lemma~\ref{lem:prva} below, we first compare the generalized asymptotic expansion of $V_f(\varepsilon)$ and the asymptotic expansion of $V_f^{\mathrm{c}}(\varepsilon)$ in power-logarithmic scale of the continuous critical time $\tau_\varepsilon$, as $\tau_\varepsilon\to+\infty$ (i.e., as $\varepsilon\to 0^+$). We say \emph{generalized}, in the sense that we allow the coefficients in the expansion of $V_f^{\mathrm{c}}(\varepsilon)$ to be oscillatory functions in $\tau_\varepsilon$. Note that $\varepsilon\mapsto\tau_\varepsilon$ itself does not contain oscillatory terms and is asymptotically \emph{fully} expandable in power-logarithmic scale in $\varepsilon$, as $\varepsilon\to 0^+$. 

Then, in Lemma~\ref{lema:druga}, by iterated integration, we obtain the distributional asymptotic expansion of $V_f(\varepsilon)$ in power-logarithmic scale of $\tau_\varepsilon$. The central point of the proof is the fact that, by consecutive integration of $V_f(\varepsilon)$, in each step we 'shift' the oscillatory terms \emph{sufficiently} further, in the sense of asymptotic order. By \emph{sufficiently}, we mean by a power of $\varepsilon$ strictly (and fixedly) larger than one.
In the limit (integrating infinitely many times), using Proposition~\ref{lem:integrali}, we obtain the distributional asymptotic expansion of $V_f(\varepsilon)$ in powers of $\tau_\varepsilon$, as $\tau_\varepsilon\to+\infty$ ($\varepsilon\to 0^+$). 

Finally, to prove Theorem~A, we expand $\varepsilon\mapsto\tau_\varepsilon$ in power-logarithmic monomials in $\varepsilon,$ as $\varepsilon\to 0^+$, and re-group the terms with matching powers of $\varepsilon$. %Now, the same conclusions follow for expansions in $\varepsilon$ as for $\tau_\varepsilon$, due to linearity of the re-grouping with respect to the coefficients/oscillatory terms multiplying the powers of $\tau_\varepsilon$. 

\begin{lemma}[Expansions of $V_f(\varepsilon)$ and of $V_f^{\mathrm{c}}(\varepsilon)$ as functions of the continuous critical time $\tau_\varepsilon$]\label{lem:prva} Let $f\in\mathrm{Diff}(\mathbb R,0)$ be a parabolic diffeomorphism of formal type $(k,\rho)$, $k\in\mathbb N$, $\rho\in\mathbb R$, with attracting direction at $\mathbb R_+$, as in \eqref{eq:formaA}. The following $($generalized$)$ asymptotic expansions hold:
\begin{align}
V_f(\varepsilon)&\sim\phantom{+}\ a^{-\frac{1}{k}}k^{-\frac{1}{k}}\frac{k+1}{k}\tau_{\varepsilon}^{-\frac{1}{k}}+\sum_{m=2}^{k}c_m\tau_\varepsilon^{-\frac{m}{k}}+\label{eq:ok}\\
&+\rho a^{-\frac{1}{k}}k^{-\frac{1}{k}-2}\frac{k+1}{k}\tau_\varepsilon^{-1-\frac{1}{k}}\log\tau_\varepsilon+\nonumber\\
&+\sum_{m=k+1}^{2k} c_m(x_0)\tau_\varepsilon^{-\frac{m}{k}}+\sum_{m=k+2}^{2k+1}\sum_ {p=1}^{\lfloor \frac{m}{k}\rfloor+1} d_{m,p}(x_0)\tau_\varepsilon^{-\frac{m}{k}}\log^p\tau_\varepsilon+\nonumber\\
&+\sum_{m=2k+1}^{\infty} \tilde H_m(G(\tau_\varepsilon))\tau_\varepsilon^{-\frac{m}{k}}+\sum_{m=2k+2}^{\infty}\sum_ {p=1}^{\lfloor \frac{m}{k}\rfloor+1} \tilde K_{m,p}(G(\tau_\varepsilon))\tau_\varepsilon^{-\frac{m}{k}}\log^p\tau_\varepsilon,\nonumber\\
V_f^{\mathrm{c}}(\varepsilon)&\sim\phantom{+}\ a^{-\frac{1}{k}}k^{-\frac{1}{k}}\frac{k+1}{k}\tau_{\varepsilon}^{-\frac{1}{k}}+\sum_{m=2}^{k}c_m\tau_\varepsilon^{-\frac{m}{k}}+\nonumber\\
&+\rho a^{-\frac{1}{k}}k^{-\frac{1}{k}-2}\frac{k+1}{k}\tau_\varepsilon^{-1-\frac{1}{k}}\log\tau_\varepsilon+ \nonumber\\
&+\sum_{m=k+1}^{2k} c_m(x_0)\tau_\varepsilon^{-\frac{m}{k}}+\sum_{m=k+2}^{2k+1}\sum_ {p=1}^{\lfloor \frac{m}{k}\rfloor+1} d_{m,p}(x_0)\tau_\varepsilon^{-\frac{m}{k}}\log^p\tau_\varepsilon+\nonumber\\
&+\sum_{m=2k+1}^{\infty}\!\!\!\! \tilde H_m(0)\tau_\varepsilon^{-\frac{m}{k}}+\!\!\!\sum_{m=2k+2}^{\infty}\!\sum_ {p=1}^{\lfloor \frac{m}{k}\rfloor+1} \tilde K_{m,p}(0)\tau_\varepsilon^{-\frac{m}{k}}\log^p\tau_\varepsilon, \ \tau_\varepsilon\to+\infty\ (\varepsilon\to 0^+).\nonumber
\end{align}
Here,  $c_m\in\mathbb R$, $m=2,\ldots k,$ depend only on the coefficients of $f$ and do not depend on the initial point $x_0$, and $c_{m}(x_0)\in\mathbb R$, $m=k+1,\ldots,2k$, $d_{m,p}(x_0)\in\mathbb R$, $m=k+2,\ldots,2k+1,$ depend on $x_0$. Furthermore, $G(s)$ is the $1$-periodic function defined in \eqref{eq:ge}, and $\tilde H_{m}(s)$, $m\geq k+1,$ resp. $\tilde K_{m,p}(s)$, $m\geq k+2,\,p=1,\ldots,\lfloor \frac{m}{k}\rfloor+1$, are polynomials of degree at most $\lfloor \frac{m-1}{k}\rfloor$ resp. $\lfloor \frac{m}{k}\rfloor-1$, both with coefficients in general \emph{depending on $x_0$}, and $\tilde H_{m}(0)$ and $\tilde K_{m,p}(0)$ are their free coefficients. 
\end{lemma}
\noindent The proof is in the Appendix.
\bigskip

\noindent The following Proposition~\ref{lem:integrali} will be used in the proof of Lemma~\ref{lema:druga} below.
\medskip 

Let $F\in C([0,\delta])$, $\delta>0$, such that $F(0)=0$. We denote by $F^{[k]}$, $k\geq 1$, the $k$-th primitive of $F$, $F^{[k]}\in C^k([0,\delta])$:\footnote{Note that of all the possible $k$-th primitives of $F$, $F^{[k]}$ is the unique one  which satisfies $F^{[k]}(0)=(F^{[k]})'(0)=\cdots=(F^{[k]})^{(k)}(0)=0$.}
\begin{equation}\label{k-th-prim}
	F^{[k]}(t):=\int_0^t\int_0^{s_1}\ldots \int_0^{s_{k-1}} F(u)\,du\, ds_{k-1}\ldots ds_1.
\end{equation}

\begin{proposition}[Distributional expansions and consecutive integration]\label{lem:integrali} Let $F\in C([0,\delta])$, $\delta>0$, $F(0)=0$. Let, for some $k\geq 1$, $m\geq 1$,
\begin{align*}           
F^{[k]}(t)= \sum_{p_1=0}^{n_1} c_{1,p_1}t^{\alpha_1}\log^{p_1}t+\ldots+\sum_{p_m=0}^{n_m}c_{m,p_m} t^{\alpha_m}\log^{p_m}t+O(t^{\alpha_{m+1}}),\\
 t\to 0^+, \ c_{i,p_i}\in\mathbb R,\ i=1\ldots m,
\end{align*}
$k<\alpha_1<\ldots<\alpha_{m+1}$, $n_i\in\mathbb N_0$,\ $i=1,\ldots,m$, be a \emph{pointwise} asymptotic expansion of $F^{[k]}$, as $t\to 0^+$, up to the order $\alpha_m$. Then
%\footnote{Here,
%$
%[\alpha]_{k}:=\alpha(\alpha-1)\cdots(\alpha-k+1),\ \alpha\in\mathbb R,\ k\in\mathbb N.
%$}
\begin{equation*}
F(t)=_{\mathcal D}\sum_{p_1=0}^{n_1} c_{1,p_1} (t^{\alpha_1}\log^{p_1}t)^{(k)}+\ldots+\sum_{p_1=0}^{n_1} c_{m,p_m} (t^{\alpha_m}\log^{p_m}t)^{(k)}+O(t^{\alpha_{m+1}-k}),
\end{equation*}
is a \emph{distributional asymptotic expansion} of $F$ up to the order $\alpha_m-k,\ t\to 0^+$. 
\end{proposition}
\noindent Note that, for $k\in\mathbb N$, $\alpha\in\mathbb R$, $p\in\mathbb N_0$, $t>0$, $$(t^\alpha \log^p t)^{(k)}=t^{\alpha-k}\sum_{\ell=0}^{\min\{k,p\}} a_\ell \log^{p-\ell} t,\ a_\ell\in\mathbb R.$$ 
The proof is in the Appendix.

\begin{lemma}[Distributional expansion of $V_f(\varepsilon)$ expressed by the continuous critical time $\tau_\varepsilon$]\label{lema:druga} Let $f\in\mathrm{Diff}(\mathbb R,0)$ be a parabolic diffeomorphism of formal type $(k,\rho)$, $k\in\mathbb N$, $\rho\in\mathbb R$, with attracting direction at $\mathbb R_+$, as in \eqref{eq:formaA}. Then the following distributional asymptotic expansion holds:
\begin{align}\label{dok}
V_f(\varepsilon)\sim_{\mathcal D}&\phantom{+}\ a^{-\frac{1}{k}}k^{-\frac{1}{k}}\frac{k+1}{k}\tau_{\varepsilon}^{-\frac{1}{k}}+\sum_{m=2}^{k}c_m\tau_\varepsilon^{-\frac{m}{k}}+\\
&+\rho a^{-\frac{1}{k}}k^{-\frac{1}{k}-2}\frac{k+1}{k}\tau_\varepsilon^{-1-\frac{1}{k}}\log\tau_\varepsilon+ \nonumber\\
&+\sum_{m=k+1}^{2k} c_m(x_0)\tau_\varepsilon^{-\frac{m}{k}}+\sum_{m=k+2}^{2k+1}\sum_ {p=1}^{\lfloor \frac{m}{k}\rfloor+1} d_{m,p}(x_0)\tau_\varepsilon^{-\frac{m}{k}}\log^p\tau_\varepsilon+\nonumber\\
&+\sum_{m=2k+1}^{\infty} \bigg(\int_0^1\tilde H_m(s)\,ds\bigg)\,\tau_\varepsilon^{-\frac{m}{k}}+\nonumber\\
&+\sum_{m=2k+2}^{\infty}\sum_ {p=1}^{\lfloor \frac{m}{k}\rfloor+1} \bigg(\int_0^1\tilde K_{m,p}(s)\,ds\bigg)\cdot \tau_\varepsilon^{-\frac{m}{k}}\log^p\tau_\varepsilon,\quad \tau_\varepsilon\to+\infty.\nonumber
\end{align}
Here, $c_m$, $m=2,\ldots,k,$ $c_{m}(x_0),\ m=k+1,\ldots,2k$,\ $d_{m,p}(x_0)$, $m=k+2,\ldots,2k+1,$ and $\tilde H_m$, $m\geq 2k+1$,\ $\tilde K_{m,p},\ m\geq 2k+2,$ are coefficients and polynomials from \eqref{eq:ok} in Lemma~\ref{lem:prva}.
\end{lemma}

\begin{proof} We consecutively integrate the asymptotic expansion \eqref{eq:ok} of $V_f(\varepsilon)$ in powers of $\tau_\varepsilon$, which is valid by Lemma~\ref{lem:prva}. 
The lemma is then proven by induction. %By \eqref{taueps}, $\tau_\varepsilon^{-\frac{m}{k}}\sim \varepsilon^{\frac m {k+1}}$, $m\in\mathbb N$, and expands in power-logarithmic scale in $\varepsilon$, as $\varepsilon\to 0$, see \eqref{taueps} for precise description. 

For simplicity, expansion \eqref{eq:ok} may be written as:
\begin{equation}\label{eq:use}
V_f(\varepsilon)\sim \sum_{m=1}^{\infty} \Big(\tilde H_m(G)\tau_\varepsilon^{-\frac{m}{k}}+\sum_ {p=1}^{\lfloor \frac{m}{k}\rfloor+1} \tilde K_{m,p}(G)\tau_\varepsilon^{-\frac{m}{k}}\log^p\tau_\varepsilon\Big),\ \varepsilon\to 0^+.
\end{equation}
Indeed, in this notation, the polynomials $\tilde H_m$, for $1\leq m\leq 2k$, and $\tilde K_{m,p}$, for $1\leq m\leq 2k+1$,  are of degree $0$, i.e., constants, and $\tilde K_{m,p}\equiv 0$ for $1\leq m\leq k$. 
We prove here by consecutive integration that:
\begin{align}\begin{split}\label{eq:dobij}
V_f(\varepsilon)\sim_{\mathcal D}
\sum_{m=1}^{\infty} \bigg[\Big(\int_0^1\tilde H_m(s)\,ds\Big)\,\tau_\varepsilon^{-\frac{m}{k}}+\sum_ {p=1}^{\lfloor \frac{m}{k}\rfloor+1} \Big(\int_0^1 \tilde K_{m,p}(s)\,ds\Big)\,\tau_\varepsilon^{-\frac{m}{k}}\log^p&\tau_\varepsilon\bigg],\\ 
&\varepsilon\to 0^+.
\end{split}\end{align}
It suffices to show that, for every $N\in \mathbb N$, there exists a $k_N\in\mathbb N$, such that, after $k_N$ consecutive integrations of $\varepsilon\mapsto V_f(\varepsilon)$, for every $\delta>0$, we get:
\begin{align}\label{eq:ha}
V^{[k_N]}_{f}(\varepsilon)=&
\sum_{m=1}^{N-1} \bigg[\Big(\int_0^1 \tilde H_m(s)\, ds\Big)\cdot  \big(\tau_\varepsilon^{-\frac{m}{k}}\big)^{[k_N]}+\\
&+\sum_ {p=1}^{\lfloor \frac{m}{k}\rfloor+1} \!\!\!\!\Big(\int_0^1 \tilde K_{m,p}(s)\,ds\Big)\big(\tau_\varepsilon^{-\frac{m}{k}}\log^p\tau_\varepsilon\big)^{[k_N]}\bigg]+O(\varepsilon^{\frac{N}{k+1}+k_N-\delta}),\quad \varepsilon\to 0^+.\nonumber
\end{align}
By notation $(\tau_\varepsilon^{-\frac{m}{k}})^{[k]}$, $k\in\mathbb N$, we mean the $k$-th primitive function of function $\varepsilon\mapsto\tau_\varepsilon^{-\frac{m}{k}}$, as defined in \eqref{k-th-prim}. Note that $\tau_\varepsilon^{-\frac{N-1}{k}}\sim \varepsilon^{\frac{N-1}{k+1}}$, $\varepsilon\to 0^+$. In our case, the number of needed integrations will grow with growing $N\in\mathbb N$. Once we have proven \eqref{eq:ha}, due to the fact that functions $\varepsilon \mapsto \tau_\varepsilon^{-\frac{m}{k}},\ \varepsilon\mapsto \tau_\varepsilon^{-\frac{m}{k}}\log^p\tau_\varepsilon,\ m\in\mathbb N,\ p\in\mathbb N_0,$ can, by \eqref{taueps}, be expanded in power-logarithmic scale in $\varepsilon$, without oscillatory terms, and that their expansions can be integrated termwise, directly applying Proposition~\ref{lem:integrali} to \eqref{eq:ha} we get \eqref{eq:dobij}, that is, the statement of Lemma~\ref{lema:druga}.
\smallskip

It is therefore left to prove \eqref{eq:ha}. By the first integration of \eqref{eq:use}, which can be done termwise (the asymptotics is proven first cutting-off at a term of arbitrarily high order and then integrating), we get \emph{termwise}:
\begin{align}\label{eq:newo}
&V^{[1]}_{f}(\varepsilon)\sim \sum_{m=1}^{\infty} \Big(\int_0^\varepsilon\tilde H_m(G(\tau_s))\tau_s^{-\frac{m}{k}}\,ds+\sum_ {p=1}^{\lfloor \frac{m}{k}\rfloor+1} \int_0^\varepsilon\tilde K_{m,p}(G(\tau_s))\tau_s^{-\frac{m}{k}}\log^p\tau_s\,ds\Big),
\end{align}
as $\varepsilon\to 0^+.$ Here, $\tilde H_m\circ G(s)$ and $\tilde K_{m,p}\circ G(s)$ are bounded and $1$-periodic, but they are not in general integrally bounded $1$-periodic (its primitive is not necessarily bounded $1$-periodic). We apply to it integral normalization described in Proposition~\ref{prop:intnorm}, and define:
\begin{align}\label{eq:ib}
&\tilde P_m(t):=\tilde H_m(t)-\int_0^1 \tilde H_m(G(u))\,du,\\
&\tilde R_{m,p}(t):=\tilde K_{m,p}(t)-\int_0^1 \tilde K_{m,p}(G(u))\,du,\quad t\in[0,1],\ m\in\mathbb N,\ p=1,\ldots,\lfloor\frac{m}{k}\rfloor+1,\nonumber
\end{align}
which are integrally bounded $1$-periodic. Obviously, $\tilde P_m\equiv 0$ for $1\leq m\leq 2k,$ and $\tilde R_{m,p}\equiv 0$ for $1\leq m\leq 2k+1,$ $p\in\mathbb N_0$, since $\tilde H_m$ and $\tilde K_{m,p}$ are then constants. Putting \eqref{eq:ib} in \eqref{eq:newo} and recalling that $G(s)=1-s$, $s\in(0,1)$, we get:
\begin{align}\label{eq:new1}
V^{[1]}_{f}(\varepsilon)\sim \sum_{m=1}^{\infty} & \Big(\int_0^1 \tilde H_m(u)\,du \Big)\cdot \big(\tau_\varepsilon^{-\frac{m}{k}}\big)^{[1]}+\sum_{m=2k+1}^{\infty}  \int_0^\varepsilon \tilde P_m(G(\tau_s)) \tau_s^{-\frac{m}{k}}\, ds+\\
&+ \sum_{m=k+2}^{\infty}\sum_ {p=1}^{\lfloor \frac{m}{k}\rfloor+1} \Big(\int_0^1 \tilde K_{m,p}(u)\,du\Big)\cdot \big(\tau_\varepsilon^{-\frac{m}{k}}\log^p\tau_\varepsilon\big)^{[1]}+\nonumber\\
&+\sum_{m=2k+2}^{\infty}\sum_ {p=1}^{\lfloor \frac{m}{k}\rfloor+1}\int_0^{\varepsilon} \tilde R_{m,p}(G(\tau_s))\tau_s^{-\frac{m}{k}}\log^p\tau_s\,ds,\ \varepsilon\to 0^+.\nonumber
\end{align}
Here, $\tilde P_m\circ G$ and $\tilde R_{m,p}\circ G$ are integrally bounded $1$-periodic. To prove \eqref{eq:ha} from \eqref{eq:new1}, it suffices to show that, by an appropriate number of consecutive integrations, we can make the terms
\begin{equation}\label{eq:s}
\sum_{m=2k+1}^{\infty}  \int_0^\varepsilon \tilde P_m(G(\tau_s)) \tau_s^{-\frac{m}{k}}\, ds\text{ and } \sum_{m=2k+2}^{\infty}\sum_ {p=1}^{\lfloor \frac{m}{k}\rfloor+1}\int_0^{\varepsilon} \tilde R_{m,p}(G(\tau_s))\tau_s^{-\frac{m}{k}}\log^p\tau_s\,ds
\end{equation}
asymptotically arbitrarily small. That is, by Proposition~\ref{lem:integrali}, we just need to prove that, for every $N\in\mathbb N$, there exists a sufficiently large $m_N\in\mathbb N$, such that:
\begin{align}\begin{split}\label{eq:jedan}
&\Big(\int_0^\varepsilon \tilde P_{2k+1}(G(\tau_s)) \tau_s^{-1-\frac{2}{k}}\,ds\Big)^{[m_N]}=O(\varepsilon^{N+m_N}),\\
&\Big(\int_0^\varepsilon \tilde R_{2k+2,1}(G(\tau_s)) \tau_s^{-1-\frac{2}{k}}\log\tau_s\,ds\Big)^{[m_N]}=O(\varepsilon^{N+m_N}),\ \varepsilon\to 0^+.
\end{split}\end{align}
The following integrals ($m> 2k+1$) in the asymptotic sum \eqref{eq:s} are of even bigger order and the discussion follows in the same way as for the first integral, with the same $m_N$. Also, since $\tau_s\sim s^{-\frac{k}{k+1}}$, $s\to 0^+$, the $r$-th integrals in the sums \eqref{eq:s} are of order $O(\varepsilon^{\frac{m}{k+1}+r-\gamma})$, for every $\gamma>0$, due to the boundedness of $\tilde P_m\circ G$ and $\tilde R_{m,p}\circ G$, so we can always bound the remainder of the sum (except for the first finitely many terms) with sufficiently small order. But \eqref{eq:jedan} follows directly using auxiliary Corollary~\ref{cor:h} of Lemma~\ref{lem:growth} in the Appendix consecutively, finitely many times. 
\end{proof}

\begin{proof}[Proof of Theorem~A] By Lemma~\ref{lem:prva}, we obtain asymptotic expansions of $\varepsilon\mapsto V_f(\varepsilon)$ and of $\varepsilon\mapsto V_f^{\mathrm{c}}(\varepsilon)$, and, by Lemma~\ref{lema:druga}, the distributional asymptotic expansion of $\varepsilon\mapsto V_f(\varepsilon)$, but in powers of the critical continuous time $\tau_\varepsilon$, as $\tau_\varepsilon\to +\infty$ ($\varepsilon\to 0^+$). Using a well known result about the asymptotics of the Fatou cooridnate,
$$
\Psi(x)\sim \big(\frac{1}{kx^k}+\rho\log x\big)\circ \widehat\varphi(x),\ x\to 0,$$
where $\widehat\varphi(x)-a^{\frac{1}{k}}x\in x^2\mathbb R[[x]]$ (see the discussion around \eqref{eq:psiinv} in the Appendix for more details), we obtain that:
\begin{align}\label{taueps}
\tau_\varepsilon+\Psi(x_0)=\Psi(g^{-1}(2\varepsilon))=&a^{-\frac{1}{k+1}}k^{-1}2^{-\frac{k}{k+1}}\varepsilon^{-\frac{k}{k+1}}+\sum_{r=1}^{k}c_r \varepsilon^{-\frac{k-r}{k+1}}+\frac{\rho}{k+1}\log\varepsilon+\\
&+\widehat S(\varepsilon),\ \ \ c_r\in\mathbb R,\ \widehat S(\varepsilon)\in \varepsilon^{\frac{1}{k+1}}\mathbb R[[\varepsilon^{\frac{1}{k+1}}]],\ \varepsilon\to 0^+.\nonumber
\end{align}
The coefficients $c_r,\ r=1,\ldots,k,$ and the coefficients in $\widehat S(\varepsilon)$ in \eqref{taueps} depend only on the coefficients of (finite jets of) $f$, and not on the initial condition $x_0$. Obviously, $\tau_\varepsilon\to +\infty$, as $\varepsilon\to 0^+$, and $\tau_\varepsilon$ possesses a full asymptotic expansion in powers of $\varepsilon$, starting with negative powers, with only one logarithmic term $\frac{\rho}{k+1}\log\varepsilon$.

We now use asymptotic expansion \eqref{taueps} of the critical time $\tau_\varepsilon$ in powers of $\varepsilon$, as $\varepsilon\to 0^+$, put it in the expansions \eqref{eq:ok} and \eqref{dok} and re-group the terms with matching powers of $\varepsilon$. It is easy to see that the coefficients thus obtained in  the expansion of $V_f(\varepsilon)$ are again polynomials in $G(\tau_\varepsilon)$, where the ones multiplying $\varepsilon^{\frac{m}{k+1}}$ are of degree at most $\lfloor \frac{m-1}{k}\rfloor$,\ and the ones multiplying $\varepsilon^{\frac{m}{k+1}}\log^p\varepsilon$, $p\in\mathbb N$, of degree at most $\lfloor \frac{m}{k}\rfloor-1$,\ $m\in\mathbb N$. Due to the linearity of the re-grouping with respect to the coefficients/oscillatory terms multiplying the powers of $\tau_\varepsilon$, and due to the linearity of integration and of taking the free coefficient of a polynomial, we get the same description of coefficients of the asymptotic expansion of $V_f^{\mathrm{c}}(\varepsilon)$ in $\varepsilon$, and of the distributional asymptotic expansion of $V_f(\varepsilon)$ in $\varepsilon$, as for their corresponding expansions in $\tau_\varepsilon$ in Lemmas~\ref{lem:prva},\,\ref{lema:druga}.
 
Moreover, by \cite[Theorem B]{MRRZ3}, $V_f(\varepsilon)-V_f^{\mathrm{c}}(\varepsilon)=O(\varepsilon^{\frac{2k+1}{k+1}})$. Since the expansion of $V_f^{\mathrm{c}}(\varepsilon)$ does not contain oscillatory terms, we conclude that all polynomials multiplying $\varepsilon^{\frac{m}{k+1}}$ for $1\leq m\leq 2k$ are of order $0$, that is, just constants. Similarly, all polynomials multiplying $\varepsilon^{\frac{m}{k+1}}\log^p\varepsilon$, $p\geq 1,$ are, for $1\leq m\leq 2k+1$, constants. 

Finally, the non-dependence on the initial term $x_0$ of the first finitely many coefficients in expansions $(1)$--$(3)$ in Theorem~A follows from formula for $V_f^{\mathrm{c}}(\varepsilon)=2\varepsilon\cdot \tau_\varepsilon+f^{\tau_{\varepsilon}}(x_0)=2\varepsilon\big(\Psi(g^{-1}(2\varepsilon))-\Psi(x_0)\big)+g^{-1}(2\varepsilon)$, from which it is visible that the only coefficient depending on $x_0$ in the asymptotic expansion of $V_f^{\mathrm{c}}(\varepsilon)$, as $\varepsilon\to 0$, is the one multiplying $\varepsilon$. Indeed, note that $\Psi,\ g^{-1}$ do not depend on $x_0$.\end{proof}

\begin{remark}\label{rem:sve}\ 

(a) {\it Parabolic orbits.} Note that the main reason that the proof of Theorem~A $(3)$ goes through in the parabolic case is Lemma~\ref{lem:growth} and Corollary~\ref{cor:h}. Integrating terms of the form $\varepsilon^\alpha \log^p\varepsilon\cdot G(\tau_\varepsilon)$, $\alpha\in\mathbb R$, $p\in\mathbb N_0$, with $G(s)$ integrally bounded $1$-periodic, the order after integration grows by \emph{strictly} (and fixedly) more than $1$. That is,
$$
\int_0^\varepsilon s^\alpha \log^p s\cdot G(\tau_s)\, ds=O(\varepsilon^{\alpha+(>1)}\log^p\varepsilon).
$$
The reason behind that is partial integration as described in the proof of Lemma~\ref{lem:growth} and the fact that, in the parabolic case, $$\frac{d}{ds}\tau_s\sim s^{-(>1)},\ s\to 0,$$ since $\tau_{s}\sim s^{-\frac{k}{k+1}}$, $s\to 0$, $k\geq 1$. This allows that, by each consecutive integration $\int_0^\varepsilon\,ds$ of the function $\varepsilon\mapsto V_f(\varepsilon)$, as described in the proof of Theorem~A $(3)$, we obtain more and more terms in the distributional expansion of $\varepsilon\mapsto V_f(\varepsilon)$, as $\varepsilon\to 0^+$. 
\smallskip

(b) \textit{Note on hyperbolic orbits.} Note that, in the hyperbolic case, described in detail for the model case in Subsection~\ref{sec:hyp} below, the proof of Theorem~A $(3)$ does not go through since
$$\tau_s\sim\log s,\ s\to 0.$$ The consequence is that $$\frac{d}{ds}\tau_s\sim s^{-1},\ s\to 0,$$ is of \emph{insufficient} growth to 'move forward' in obtaining more monotonic terms  in the distributional expansion by iterative integration (by consecutive integrations, we never \emph{surpass} the order of the growth $\varepsilon^1$ of the second term). Therefore, the statement of Theorem~A $(3)$ for the distributional expansion of $V_f(\varepsilon)$ does not hold in this case. Accordingly, the distributional expansion, as we see in Subsection~\ref{sec:hyp} below, contains oscillatory terms which multiply $\varepsilon^1$. The oscillatory terms are not \emph{regularized} as in the parabolic case. 

On the other hand, Proposition~\ref{prop:ntau} holds also in the hyperbolic case, so the hyperbolic counterparts of statements $(1)$ and $(2)$ of Theorem~A hold similarly for the hyperbolic case, and $V_f^{\mathrm{c}}(\varepsilon)$ \emph{regularizes} oscillatory terms and necessarily \emph{eliminates} oscillations in the asymptotic expansion.
\end{remark}

\section{Examples}\label{sec:examples}

\subsection{Model parabolic orbits}\label{subsec:racun}\

Recall from Proposition~\ref{k=1_razvoj_p} in Section~\ref{sec:cont} that, for the \emph{model diffeomorphism} $$f_0(x)=\frac{x}{1+x},$$ the distributional asymptotic expansion of the tube function $\varepsilon\mapsto V_{f_0}(\varepsilon)$, obtained by means of residues of the meromorphic extension of its distance zeta function $\zeta_{f_0}$, is equal to:
\begin{equation*}\label{eq:dist1}
V_{f_0}(\varepsilon)\sim_{\mathcal D} 2\sqrt 2\varepsilon^{1/2}-\frac{2}{x_0}\cdot \varepsilon+\sum_{l=1}^{\infty}\frac{2^{3/2+l}}{1-4l^2}{l-\frac{1}{2}\choose 2l} \varepsilon^{l+\frac{1}{2}},\ \varepsilon\to 0^+.
\end{equation*}
% Note that, for simplicity, we write here just $A\big(\mathcal O^{f_0}(x_0)_\varepsilon\big)$, in fact meaning\linebreak  $A\big(\mathcal O^{f_0}(x_0)_\varepsilon\big)\cap[0,x_0]$.
\smallskip

Here, in Proposition~\ref{prop:regular}, we compute the asymptotic expansion of the continuous time tube function $\varepsilon\mapsto V_{f_0}^{\mathrm{c}}(\varepsilon)$ and also obtain again the distributional expansion of $\varepsilon\mapsto V_{f_0}(\varepsilon)$. We compute directly, obtaining explicitely, for the model example, the generalized asymptotic expansion of the tube function with possibly oscillatory coefficients and then using Theorem~A.
Moreover, we compare the coefficients of these two types of regularization (or smoothening) of the (classical) tube function $V_{f_0}(\varepsilon)$. 
\medskip

Let us introduce the notation:
$$
[y^{k}\log^m y]h(y)
$$
for the coefficient in front of $y^k\log^m y$, $k\in\mathbb R$, $m\in\mathbb Z_0$, in the asymptotic expansion of function  $y\mapsto h(y)$, $y\to 0,$ or in the asymptotic series $\widehat h(y)$.

\begin{proposition}\label{prop:regular} Let $f_0(x)=\frac{x}{1+x}$, and $0<x_0<1$. Then:
\begin{equation}\label{eq:reg}
\frac{[\varepsilon^{m+\frac{1}{2}}]V_{f_0}^{\mathrm{c}}(\varepsilon)}{[\varepsilon^{m+\frac{1}{2}}]\mathcal D(V_{f_0}(\varepsilon))}=2m+1,\ m\in\mathbb N_0.
\end{equation}
Here, $\mathcal D(V_{f_0}(\varepsilon))$ denotes the distributional asymptotic expansion of $\varepsilon\mapsto V_{f_0}(\varepsilon)$. The coefficients in both asymptotic series with positive integer powers of $\varepsilon$ $($except $\varepsilon^1)$ are equal to zero, and the coefficient in front of $\varepsilon^1$ in both series is equal to $-\frac{2}{x_0}$.
\end{proposition}

\begin{proof} We compute explicitly the generalized asymptotic expansion of the tube function $V_{f_0}(\varepsilon)$, as $\varepsilon\to 0^+$, for $f_0(x)=\frac{x}{1+x}$, and then apply Theorem~A. Indeed, $G:=G(\tau_\varepsilon)$ being a periodic oscillatory function explicitly given in \eqref{eq:ge},
\begin{align}\label{eq:first}
V_{f_0}(\varepsilon)=n_\varepsilon\cdot 2\varepsilon+x_{n_\varepsilon}=&2\varepsilon(\tau_\varepsilon+G)+\Psi^{-1}(\tau_\varepsilon+G-\Psi(x_0)).
\end{align}
In the model case, due to simplicity, we can \emph{explicitely} compute:
\begin{align}\begin{split}\label{eq:second}
&\Psi(x)=\frac{1}{x},\ \Psi^{-1}(y)=\frac{1}{y},\\
&g_0(x)=x-f_0(x)=\frac{x^2}{1+x},\ g_0^{-1}(y)=\sqrt{y+y^2}+y\sim\sqrt y,\ y\to 0^+,\\
&\tau_\varepsilon=\Psi(g_0^{-1}(2\varepsilon))-\Psi(x_0)=\frac{(1+\frac{\varepsilon}{2})^{1/2}}{\sqrt {2\varepsilon}}-\frac{1}{2}-\frac{1}{x_0}.
\end{split}\end{align}
Putting \eqref{eq:second} in \eqref{eq:first}, we get, after some computation:
\begin{align}\label{eq:sve}
V_{f_0}(\varepsilon)=\sqrt{2\varepsilon} \big(1+\frac{\varepsilon}{2}\big)^{1/2}+\varepsilon\big(2G-1- \frac{2}{x_0}\big)+\frac{\sqrt {2\varepsilon}(1+\frac{\varepsilon}{2})^{1/2}-\varepsilon(2G-1)}{1-(2G^2-2G)\varepsilon}.
\end{align}
We now compute from \eqref{eq:sve}:
\begin{align*}
[\varepsilon^{m+\frac{1}{2}}]&V_{f_0}(\varepsilon)=\\
=&[\varepsilon^{m+\frac{1}{2}}]\Big(\sqrt {2\varepsilon} \big(1+\frac{\varepsilon}{2}\big)^{1/2}+\sqrt{2\varepsilon}\big(1+\frac{\varepsilon}{2}\big)^{1/2}\cdot \Big(1-(2G^2-2G)\varepsilon\Big)^{-1}=\\
=&\sqrt{2}\cdot{1/2 \choose m}\big(\frac 1 2 \big)^m+\sqrt{2}\cdot \sum_{i=0}^{m}\cdot{1/2 \choose i}\big(\frac 1 2 \big)^i (2G^2-2G)^{m-i},\ m\in\mathbb N_0,
\end{align*}
and
\begin{align*}
[\varepsilon^m]V_{f_0}(\varepsilon)=&[\varepsilon^m]\Big(-\varepsilon(2G-1)\cdot\big(1-(2G^2-2G)\varepsilon\big)^{-1}\Big)=\\
&=-(2G-1)(2G^2-2G)^{m-1},\ \ m\in\mathbb N,\ m\geq 2,\\
[\varepsilon]V_{f_0}(\varepsilon)=&-\frac{2}{x_0}.
\end{align*}
In notation of Theorem~A, this gives:
\begin{align*}
&\tilde Q_{2m+1,0}(s)=\sqrt{2}\cdot{1/2 \choose m}\big(\frac 1 2 \big)^m+\sqrt{2}\cdot \sum_{i=0}^{m}\cdot{1/2 \choose i}\big(\frac 1 2 \big)^i (2s^2-2s)^{m-i},\ m\in\mathbb N_0,\\
&\tilde Q_{2m,0}(s)=-(2s-1)(2s^2-2s)^{m-1},\ m\in\mathbb N,\ m\geq 2,\ \ \tilde Q_{2,0}(s)=-\frac{2}{x_0}.
\end{align*}
Therefore,
\begin{align*}
[\varepsilon^{m+\frac 1 2}]\mathcal D(V_{f_0}(\varepsilon))&=\int_0^1 \tilde Q_{2m+1,0}(s)\,ds=\\
&=\sqrt 2\cdot{1/2 \choose m}\big(\frac 1 2 \big)^m\cdot\Big(1+_3F_{2}\Big((1,1,-m),(\frac{3}{2},\frac 3 2-1),1\Big)\Big)=\\
&=2\sqrt 2\cdot{1/2 \choose m}\big(\frac 1 2 \big)^m\frac{1}{2m+1},\ m\in\mathbb N_0,\\
[\varepsilon^{m+\frac 1 2}]V_{f_0}^{\mathrm{c}}(\varepsilon)=&\tilde Q_{2m+1,0}(0)=2\sqrt 2\cdot{1/2 \choose m}\big(\frac 1 2 \big)^m,\ m\in\mathbb N_0,\\
[\varepsilon^{m}]\mathcal D(V_{f_0}(\varepsilon))=&\int_0^1 \tilde Q_{2m,0}(s)\,ds=0,\ \ 
[\varepsilon^{m}]V_{f_0}^{\mathrm{c}}(\varepsilon)=\tilde Q_{2m,0}(0)=0,\\
&\qquad \qquad\qquad\qquad\qquad \qquad \qquad\qquad\qquad\qquad m\in\mathbb N,\ m\geq 2.
\end{align*}
Here, $_pF_{q}(.)$ denotes a \emph{generalized hypergeometric function}.
\end{proof}

\begin{remark} Note that the regularity \eqref{eq:reg} of ratios of coefficients, expressed by coefficients of polynomials $\tilde Q_{i,0}$, $i\in\mathbb N,$ from Theorem~A, is equivalent to:
\begin{align*}
\frac{\tilde Q_{2m+1,0}(0)}{\int_0^1 \tilde Q_{2m+1,0}(s)\,ds}=&\frac{a_0}{a_0+\frac{a_1}{2}+\frac{a_3}{3}+\ldots+\frac{a_{2m-1}}{2m+1}}=2m+1\\
\Leftrightarrow \ a_0+1&=-\frac{1}{2m}\Big(\frac{a_1}{2}+\frac{a_3}{3}+\ldots+\frac{a_{2m-1}}{2m+1}\Big),\ m\in\mathbb N_0,
\end{align*}
where $\tilde Q_{2m+1,0}(s)=a_0+a_1 s+\ldots a_{2m-1}s^{2m}$ is the polynomial of degree $2m$ multiplying $\varepsilon^{m+\frac{1}{2}}$ in the asymptotic expansion of $V_{f_0}(\varepsilon)$, see Theorem~A.

This explicit relation between coefficients in polynomials $\tilde Q_{i,0}$, $i\in\mathbb N$, from Theorem~A $(1)$ seems to be characteristic only of the model case, due to some symmetry in geometry of orbits.
\end{remark}
\begin{remark}[Comment on general $k\geq 2$ and non-model diffeomorphisms] We conjecture that exactly the same ratio $2m+1$ between coefficients of the asymptotic expansion of the continuous time tube function $V_{f_0}^{\mathrm{c}}(\varepsilon)$ and of the distributional expansion of the standard tube function $V_f(\varepsilon)$, as $\varepsilon\to 0^+$, holds for \emph{any model parabolic case} $k\geq 1$, $\rho=0$. It has been tested using the computer algebra software (\emph{Mathematica, Maple}) and the formula \eqref{eq:kaa} for the coefficients of the distributional expansion that:
\begin{equation*}
\frac{[\varepsilon^{\frac{1}{k+1}+m\frac{2k}{k+1}}]V_{f_0}^{\mathrm{c}}(\varepsilon)}{[\varepsilon^{\frac{1}{k+1}+m\frac{2k}{k+1}}]\mathcal D(V_{f_0}(\varepsilon))}=2m+1,\ m\in\mathbb N_0.
\end{equation*}
Other coefficients in both asymptotic series $($except the coefficient of $\varepsilon^1)$ are equal to zero.

Since, unlike $k=1$, in the general case $k>1$ we cannot get the closed explicit formula for the inverse $g_0^{-1}(y)$ as in \eqref{eq:second}, this was not proven theoretically, but only using the computer algebra software.

However, for the orbits of \emph{non-model} parabolic diffeomorphisms $f$, we expect that the best general relation we can get between the asymptotic expansion of the continuous time tube function $V_{f}^{\mathrm{c}}(\varepsilon)$ and the distributional expansion of the standard tube function $V_f(\varepsilon)$, as $\varepsilon\to 0^+$, is Theorem~A. In general case, in Theorem~A, the coefficients of polynomials $\tilde Q_{i,0}$, $i\in\mathbb N$, depend on coefficients of $f$, but the nature of this dependence is not so \emph{regular} as in the very simple, model case.
\end{remark}
\smallskip

\subsection{Model hyperbolic orbits}\label{sec:hyp}

We consider the \emph{model hyperbolic example}, the germ $$f_a\in \mathrm{Diff}(\mathbb R,0),\ f_a(x)=ax,\ 0<a<1,$$ and its (attracting) orbit $\mathcal{O}_{f_a}(x_0)$ with initial condition $0<x_0<1$. 
\medskip

(1)\ We obtain here the distributional expansion, in fact, the closed form of $\varepsilon\mapsto V_{f_a}(\varepsilon)$,  by computing the residues of the meromorphic extension of its distance zeta-function $\zeta_{f_a}(s)$ and applying \cite[Theorem 5.4.32]{fzf}, as in Proposition~\ref{k=1_razvoj_f} in Section~\ref{sec:cont} for model parabolic orbits. 

Note that the orbit of $f_a$ is given by $$\mathcal O_{f_a}(x_0)=\{x_0 a^n:\,n\in\mathbb N_0\}.$$ Such orbit generates a fractal string of consecutive distances: $$\mathcal L_{f_a}:=\{\ell_j:=f_a^{\circ j}(x_0)-f_a^{\circ(j+1)}(x_0)=x_0(1-a)a^j:j\in\mathbb N_0\},$$ named in \cite{fzf} the \emph{geometric progression fractal string}. By \cite[Example 4.2.36]{fzf}, its distance zeta function is given by the meromorphic extension of
$$
\zeta_{f_a}(s):=\frac{2^{1-s}}{s}\sum_{j=0}^{\infty}\ell_j^s=\frac{2^{1-s}x_0^s\cdot (1-a)^s}{s}\frac{1}{1-a^s}
$$
from $\{s\in\mathbb C:\text{Re}(s)>0\}$ to all of $\mathbb C$. The poles are: the double pole $s_0=0$ and the simple poles $s_k:=\frac{2k\pi }{\log a}i$, $k\in\mathbb Z\setminus\{0\}$, situated equidistantly along the vertical line $\{\re s=0\}$. By slightly adapting \cite[Example 5.5.25]{fzf}, where \cite[Theorem 5.3.16]{fzf} is used, one can use the poles and their residues to directly obtain the \emph{exact pointwise} formula, that is, the exact (multiplicative Fourier) pointwise expansion of $V_{f_a}(\varepsilon)$:
\begin{equation}\label{eq:hipfir}
V_{f_a}(\varepsilon)=-\frac{2}{\log a}\varepsilon(-\log\varepsilon)+H\Big(\log_a\frac{2\varepsilon}{x_0(1-a)}\Big)\cdot \varepsilon,\ \varepsilon\in(0,x_0(1-a)/2).
\end{equation}
Here, $H:[0,+\infty)\to\mathbb R$ is a 1-periodic bounded function, given by the following \emph{absolutely convergent} Fourier series:
\begin{equation}\label{eq:hipsec}
H(t):=1+\frac{\log4-2}{\log a}-2\log_a(x_0(1-a))-\frac{2}{\log a}\sum_{k\in\mathbb Z\setminus\{0\}}\frac{e^{2\pi i k t}}{s_k(1-s_k)},\ t\geq 0.
\end{equation}
In particular, this exact pointwise formula coincides with the pointwise and the distributional expansion of $V_{f_a}(\varepsilon)$, as $\varepsilon\to 0^+$.
In order to obtain formula \eqref{eq:hipfir}, note that the distance zeta function $\zeta_{f_a}(s)$ is exactly equal to $x_0^s(1-a)^s\left[\zeta_{A_{\mathscr{L}}}(s)-2s^{-1}\right]$ in the notation of \cite[Example 5.5.25]{fzf}.
One then follows the exact steps of \cite[Example 5.5.25]{fzf}, but taking into account the above functional equation when calculating residues of the function $s\mapsto\frac{\varepsilon^{1-s}}{1-s}\zeta_{f_a}(s)$ as required by \cite[Theorem 5.3.16]{fzf}.
Note that the residues of $s\mapsto\frac{\varepsilon^{1-s}}{1-s}\zeta_{f_a}(s)$ at the nonreal poles are equal to the residues of $s\mapsto\zeta_{A_{\mathscr{L}}}(s)$ given by \cite[Equation (5.5.148)]{fzf} multiplied by the factor $[x_0(1-a)]^{s_k}\frac{\varepsilon^{1-s_k}}{1-s_k}$ for $k\in\Ze\setminus\{0\}$, while the residue of the double pole at zero is a little bit more complicated to obtain but can be easily checked by computer algebra software.
\medskip

(2) We now compute $V_{f_a}(\varepsilon)$ and $V_{f_a}^{\mathrm{c}}(\varepsilon)$ directly, in the closed form. By Proposition~\ref{prop:ntau}, we put $n_\varepsilon=\tau_\varepsilon+G(\tau_\varepsilon)$, where $G$ is the $1$-periodic bounded function explicitly given by \eqref{eq:ge}. 

Obviously, $f_a$ is a time-$1$ map of the vector field $x'=(\log a) x$, and its Fatou coordinate is given by $\Psi(x)=\log_a x$, $\Psi^{-1}(y)=a^y$. The definition of the Fatou coordinate is given in \eqref{eq:fatou}.
Furthermore, since $g_a^{-1}(y)=(1-a)^{-1}y$,
$$
\tau_\varepsilon=\Psi(g^{-1}(2\varepsilon))-\Psi(x_0)=\log_a\frac{2\varepsilon}{x_0(1-a)}.
$$
We now compute, putting $G:=G(\tau_\varepsilon)$ for simplicity,
\begin{align}
V_{f_a}(\varepsilon)&=(\tau_\varepsilon+G)\cdot 2\varepsilon+\Psi^{-1}\big(\tau_\varepsilon+\Psi(x_0)+G\big)=\nonumber\\
&=\big(\log_a\frac{2\varepsilon}{(1-a)x_0}+G\big)\cdot 2\varepsilon+\frac{2\varepsilon}{1-a}a^{G}=\nonumber\\
&=-\frac{2}{\log a}\varepsilon(-\log\varepsilon)+\varepsilon\Big(2\log_a\frac{2}{x_0(1-a)}+2G+\frac{2a^G}{1-a}\Big)=\nonumber\\
&=-\frac{2}{\log a}\varepsilon(-\log\varepsilon)+\varepsilon F(G),\label{eq:ah}\\
V_{f_a}^{\mathrm{c}}(\varepsilon)&=\ \tau_\varepsilon\cdot 2\varepsilon+\Psi^{-1}(\tau_\varepsilon+\Psi(x_0))=\nonumber\\
&=-\frac{2}{\log a}\varepsilon(-\log\varepsilon)+\varepsilon F(0)=\nonumber\\
&=-\frac{2}{\log a}\varepsilon(-\log\varepsilon)+\varepsilon\Big(2\log_a\frac{2}{x_0(1-a)}+\frac 2 {1-a}\Big),\ \varepsilon\in(0,\delta).\label{eq:hipthree}
\end{align}
Here, $F(s):=2\log_a\frac{2}{x_0(1-a)}+2s+\frac{2a^s}{1-a}$. 
Comparing \eqref{eq:hipfir} and \eqref{eq:hipsec} with \eqref{eq:ah}, we conclude that
$$
H=F\circ G=2\log_a\frac{2}{x_0(1-a)}+2G+\frac{2a^G}{1-a}.
$$ 
\bigskip

We see from \eqref{eq:ah} and \eqref{eq:hipthree} that 'analogues' of Theorem~A $(1)$ and $(2)$ hold also in the hyperbolic case. However, unlike the parabolic case, the distributional asymptotic expansion of $V_{f_a}(\varepsilon)$, which is here equal to the exact pointwise (multiplicative) Fourier series expansion of $V_{f_a}(\varepsilon)$, contains the oscillatory term in front of $\varepsilon$. In the expansion of $V_{f_a}^{\mathrm{c}}(\varepsilon)$, on the other hand, the oscillatory terms are necessarily eliminated. Indeed, as explained in Remark~\ref{rem:sve} (b), the analogue of Theorem~A $(3)$ in the hyperbolic case does not hold. In Remark~\ref{rem:sve}, we explain a technical reason why Theorem~A $(3)$ fails in the hyperbolic case. In the hyperbolic case, we conjecture that the geometric reason is in an underlying \emph{linear self-similarity propery} of hyperbolic orbits which is not present in parabolic orbits, thus making them \emph{fractal} in the self-similar sense. Accordingly, hyperbolic orbits are fractal also in the sense of the new definition proposed by the authors of \cite{fzf,lapidusfrank12}, which states that the presence of nonreal complex dimensions is the indication of fractality. Indeed, nonreal complex dimensions (the simple poles $s_k:=\frac{2k\pi }{\log a}i$, $k\in\mathbb Z\setminus\{0\}$) appear in the case of the hyperbolic orbit, reflecting the presence of the oscillatory term $\varepsilon F(G(\tau_\varepsilon))$ in the distributional asymptotic expansion \eqref{eq:ah}, see \cite[Theorem 5.3.16 or Theorem 5.3.21]{fzf}.
Further discussion on this topic will be given in Section \ref{sec_conc}.

\section{Proof of Theorems~B and C}\label{sec:general}
%Moreover, polynomial languidity bounds will always be satisfied for the corresponding fractal zeta function.
%This will be a consequence of the following general theorem about the tube zeta function which generalizes \cite[Theorem ??]{fzf}.
%We point out that this result has the advantage that once we extend meromorphically the zeta function (by any possible way) we know immediately that the corresponding complex dimensions are indeed the co-exponents in the (pointwise or distributional) asymptotics of the tube function $\varepsilon\mapsto|\mathcal{O}(x_0)_{\varepsilon}|$ (without the need to first prove the languidity bounds on the fractal zeta function).
This section concerns the proof of Theorem~C and of Theorem~B from Section~\ref{sec:results}. Aside from being used in this paper for the proof of Theorem~B, Theorem~C has independent value in the theory of fractal zeta functions. Recall from Section~\ref{sec:results} that Theorem~C can be understood as a generalization of \cite[Theorem 2.3.18]{fzf} and partial converse of \cite[Theorem 5.4.30]{fzf}.
It gives, under appropriate hypotheses, a one-to-one correspondence of the distributional expansion of the tube function of a set and the collection of its complex dimensions and their principal parts.

In the proof, we need another type of a fractal zeta function of a bounded set, the \emph{tube zeta function}, introduced in \cite[\S2.2.2]{fzf}. The \emph{tube zeta function} $\tilde{\zeta}_A$ of a bounded subset $A\subseteq\eR^N$ is defined as the shifted one-sided Mellin transform of the tube function $\varepsilon\mapsto V_A(\varepsilon):=|A_{\varepsilon}|$, i.e., the Lebesgue measure of the $\varepsilon$-neighboorhood of the set $A$:
\begin{equation}\label{tube_zeta_def}
	\tilde{\zeta}_A(s):=\int_0^{\delta}t^{s-N-1}|A_t|\di t,\ \re s>\overline{\dim}_B A,\ \delta>0.
\end{equation}
It has analogous properties as the distance zeta function $\zeta_A$ of the set $A$ and is connected to it by the folowing functional equation, see \cite[Theorem 2.2.1]{fzf}:
\begin{equation}\label{tilde_dist}
	\zeta_A(s)=\delta^{s-N}|A_{\delta}|+(N-s)\tilde{\zeta}_A(s), \ \re s>\overline{\dim}_B A,\ \delta>0.
\end{equation}
For properties of the tube and distance zeta functions see \cite{fzf,dtzf,mefzf}.
Note that the dependence on $\delta>0$ in Equation \eqref{tube_zeta_def} is, as usual, nonessential from the point of view of the theory of complex dimensions.
This stems from the fact that changing $\delta>0$ amounts to adding a term which is entire in the variable $s$. Hence, it does not affect any complex dimensions, nor corresponding residues or principal parts.

\subsection{Proof of Theorem~C}\label{subsec:proofC}\
\smallskip

Theorem~C in Section~\ref{sec:results} is a direct generalization of Corollary~\ref{distr_asy_dist_zeta} of Proposition~\ref{distr_asy_tube_zeta} below to more asymptotic terms, and can be proven similarly. Therefore, for simplicity and clarity, we will prove here Proposition~\ref{distr_asy_tube_zeta} and its Corollary~\ref{distr_asy_dist_zeta}, and omit the lengthy technical proof of Theorem~C.

\begin{proposition}
	\label{distr_asy_tube_zeta}
Let $A\subseteq\eR^N$ be a bounded set such that the $m$-th primitive\footnote{The $m$-th primitive function $V_{A}^{[m]}$ of $\varepsilon\mapsto V_A(\varepsilon)$ is precisely defined in \eqref{k-th-prim}. In the sequel, we will call $V_{A}^{[m]}$ the \emph{$m$-th primitive tube function of the set $A$}, as in \cite{fzf}.}of the function $\varepsilon\mapsto V_A(\varepsilon)$, denoted by $V_{A}^{[m]}(\varepsilon)$, for some $m\in\mathbb N_0$, has the following asymptotics:
	\begin{equation}\label{log_asym}
		V_{A}^{[m]}(\varepsilon)=M\varepsilon^{N-\alpha+m}P_n(-\log\varepsilon)+O(\varepsilon^{N-\beta+m}),\ \varepsilon\to 0^+,
	\end{equation}
for some $\alpha,\beta\in\eR$ such that $\alpha>\beta$ and some constant $M>0$. Here, $P_n$ denotes a monic polynomial with real coefficients, of degree $n\in\eN_0$.

	Then the tube zeta function $\tilde{\zeta}_A(s)$ is meromorphic in the right open half-plane $\{\re s>\beta\}$, and given by:
	\begin{equation}\label{tilde_zeta_polovi}
		\tilde{\zeta}_A(s)=M(N-s+1)_m\,\delta^{s-\alpha}\sum_{j=0}^n\frac{P_n^{(j)}(-\log\delta)}{(s-\alpha)^{j+1}}+R(s),
	\end{equation}
	where $R$ is holomorphic in $\{\re s>\beta\}$, and
	\begin{equation}
		\label{g(s)}
		R(s)=M\sum_{n=1}^m(N-s+1)_{n-1}\delta^{s-N-n}V_{A}^{[n]}(\delta)+\int_{0}^{\delta}t^{s-1}O(t^{-\beta})\di t.
	\end{equation}
\end{proposition}
	
We conclude from \eqref{tilde_zeta_polovi} and \cite[Theorem 2.2.11]{fzf} that the $(n+1)$-th order pole $\alpha$ of $\tilde{\zeta}_A$ equals to the upper box dimension of $A$, i.e., $\overline{\dim}_B A=\alpha$, since it is the first real complex dimension of $A$.

As was already mentioned above, by the theory of \cite{fzf,dtzf}, the principal part of $\tilde{\zeta}_A(s)$ given by \eqref{tilde_zeta_polovi} at the pole $\alpha$ does not depend on $\delta>0$, a fact that is not immediately clear from \eqref{tilde_zeta_polovi}. 
Therefore, we can obtain the principal part at $s=\alpha$ by letting $\delta=1$ in \eqref{tilde_zeta_polovi}, and examining:
	\begin{equation*}\label{tilde_zeta_prinicip}
		M(N-s+1)_m\sum_{j=0}^n\frac{P_n^{(j)}(0)}{(s-\alpha)^{j+1}}.
	\end{equation*}
One now needs to expand $(N-s+1)_m$ into its Taylor polynomial at $s=\alpha$, multiply it with the sum above and extract the coefficients. 
The leading coefficient, that is, the coefficient in front of $(s-\alpha)^{-n-1}$, is easily obtained and equal to
\begin{equation*}
	M(N-\alpha+1)_m\cdot n!.
\end{equation*}

\begin{remark}\label{Tauber}
It can be proven from \eqref{tilde_zeta_polovi}, using an appropriate Tauberian theorem (see \cite[Theorem 5.4.2]{fzf} for the case when $n=0$, and
 \cite{log_gauge} for the case when $n>0$),
that $D:=\alpha$ is, in fact, the box (Minkowski) dimension of the set $A$, and that the constant $M$ is  equal to the Minkowski content of $A$ (in case $n=0$) or to its \emph{gauge Minkowski content} (in case $n>0$), modulo a multiplicative constant:
	\begin{equation}\label{eq:taubi}
		\mathcal{M}^{D}(A,h)=\lim_{\varepsilon\to0^+}\frac{|A_{\varepsilon}|}{\varepsilon^{N-D}h(\varepsilon)}=M(N-D+1)_m,
	\end{equation}
	where the gauge function $h(\varepsilon)$ is given as $(-\log\varepsilon)^n$. In case $n=0$, $\mathcal{M}^{D}(A,1)$ is just the standard Minkowski content, and this statement is proven in \cite[Theorem 5.4.2]{fzf}.
	
Note also that, when $m>0$, \eqref{eq:taubi} does not follow directly from the expansion \eqref{log_asym}, since pointwise asymptotic expansions cannot in general be differentiated termwise, and hence, we cannot directly deduce that:
\begin{equation*}\label{eq:taub}
|A_{\varepsilon}|\sim M(N-\alpha+1)_m\cdot\varepsilon^{N-\alpha}(-\log\varepsilon)^n,\ \varepsilon\to 0^+.
\end{equation*}
Therefore, to conclude that \eqref{eq:taubi} is true in case $m>0$, the use of a Tauberian theorem as in \cite{fzf,log_gauge} is essential.
\end{remark}

The following corollary is the analogue of Proposition~\ref{distr_asy_tube_zeta}, but for the distance zeta function. It is a direct consequence of Proposition~\ref{distr_asy_tube_zeta} and the functional equation \eqref{tilde_dist} relating the distance and the tube zeta function.
\begin{corollary}[Analogue of Proposition~\ref{distr_asy_tube_zeta} for distance zeta functions]\label{distr_asy_dist_zeta}
Let $A\subseteq\eR^N$ be a bounded set such that its $m$-th primitive tube function $V_{A}^{[m]}$, for some $m\in\mathbb N_0,$ has the following asymptotics:
	\begin{equation*}\label{log_asym1}
		V_{A}^{[m]}(\varepsilon)=M\varepsilon^{N-\alpha+m}P_n(-\log\varepsilon)+O(\varepsilon^{N-\beta+m}),\ \varepsilon\to0^+,
	\end{equation*}
for some $\alpha,\beta\in\eR$ such that $\alpha>\beta$ and some constant $M>0$. Here, $P_n$ denotes a monic polynomial with real coefficients of degree $n\in\eN_0$.

Then the distance zeta function ${\zeta}_A(s)$ of the set $A$ is meromorphic in the right open half-plane $\{\re s>\beta\}$, and given by:
	\begin{equation}\label{distance_zeta_polovi}
		{\zeta}_A(s)=M(N-s)_{m+1}\,\delta^{s-\alpha}\sum_{j=0}^n\frac{P_n^{(j)}(-\log\delta)}{(s-\alpha)^{j+1}}+R(s),
	\end{equation}
	where $R$ is holomorphic in $\{\re s>\beta\}$, and:
	\begin{equation}
		\label{g(s)}
		R(s)=M\sum_{n=0}^m(N-s)_{n}\delta^{s-N-n}V_{A}^{[n]}(\delta)+(N-s)\int_{0}^{\delta}t^{s-1}O(t^{-\beta})\di t.
	\end{equation}
Here, $\delta>0$ is non-essential, as in definition of the distance zeta-function \eqref{eq:dz}, and may be taken as e.g. $\delta=1$.

Finally, the distance zeta function is at most polynomially super languid in any closed vertical strip of finite width contained in $\{\re s>\beta\}$. 
\end{corollary}

\begin{proof}
	It only remains to prove the statement about super languidity.
	This follows from the fact that each term in \eqref{distance_zeta_polovi} is polynomially super languid.
	Namely, the variable $s$ appears in rational or polynomial terms (which are by default super languid) or in the term $\delta^s$ which is bounded by a constant in any vertical strip of finite width, and this is also true for the integral appearing in Equation \eqref{g(s)}.  
	Hence, by Remark \ref{sup-lang-ex}, $\zeta_A$ is also polynomially super-languid in any vertical strip of finite width.
\end{proof}

%OVO ISPOD BIH OBRISAO

%---------------------------------------

%As above, the $(n+1)$-th order pole $\alpha$ of ${\zeta}_A$ equals to the box (Minkowski) dimension of $A$, i.e., $\dim _BA=\alpha$.
%Moreover, taking $\delta=1$ (which does not change the principal part), the principal part of ${\zeta}_A(s)$ at the pole $\alpha$ can be obtained by examining:\edz{Tu je bila greška - ovo nije principal part}
%	\begin{equation}\label{zeta_prinicip}
	%	M(N-s)_{m+1}\sum_{j=0}^n\frac{P_n^{(j)}(0)}{(s-\alpha)^{j+1}},
	%\end{equation}
%that is, one needs to expand $(N-s)_{m+1}$ into the Taylor series at $\alpha$ and then, after multiplication, extract the coefficients
%-------------------------------------
	
In the proof of Proposition~\ref{distr_asy_tube_zeta}, we need the following auxiliary lemma, whose proof is in the Appendix.

\begin{lemma}\label{lemma_tilde}
	Let $A\subseteq\eR^N$ be a bounded set. Then, for every $m\in\eN$, we have the following functional equation:
	\begin{equation}\label{eq_tilde}
		\tilde{\zeta}_A(s)=\sum_{n=1}^m(N-s+1)_{n-1}\delta^{s-N-n}V_{A}^{[n]}(\delta)+(N-s+1)_m\int_{0}^{\delta}t^{s-N-1-m}V_{A}^{[m]}(t)\di t,
	\end{equation}
	valid for all $s\in\Ce$  such that $\re s>\overline{\dim}_B A$.
	%Here, $V_{A(U_\varepsilon)}^{[m]}$ is the $m$-th primitive tube function of $U$, as defined by \eqref{k-th-prim}.
\end{lemma}

\begin{proof}[Proof of Proposition~\ref{distr_asy_tube_zeta}]
	Let $A\subseteq \mathbb R^N$ be a bounded set and $m\in\mathbb N$ such that \eqref{log_asym} holds. Let $\overline D:=\overline{\dim}_B A$. From Lemma \ref{lemma_tilde}, the functional equation \eqref{eq_tilde} for $\tilde{\zeta}_A(s)$ is valid on $\mathrm{Re}\,s>\overline D.$ Putting \eqref{log_asym} in the integral on the right-hand side of \eqref{eq_tilde}, we get:
	\begin{equation}\label{int_eq_n}
		\int_{0}^{\delta}t^{s-N-1-m}V_{A}^{[m]}(t)\di t=M\int_{0}^{\delta}t^{s-\alpha-1}P_n(-\log t)\di t+\int_{0}^{\delta}t^{s-1}O(t^{-\beta})\di t.
	\end{equation}
	The last integral above is holomorphic in the open right half-plane $\{\re s>\beta\}$, since the integrand is of order $O(t^{\re s-\beta-1})$.
	It is now left to show that the first integral on the right-hand side of \eqref{int_eq_n} can be meromorphically extended to all of $\Ce$, with $\alpha$ being the only pole of order $n+1$.
	
	\noindent Indeed, for any $k\in\eN_0$ and for $\re s>\alpha$, we have: 
	$$
	\int_{0}^{\delta}t^{s-\alpha-1}(-\log t)^k\di t=\frac{\delta^{s-\alpha}(-\log\delta)^k}{s-\alpha}+\frac{k}{s-\alpha}\int_0^{\delta}t^{s-\alpha-1}(-\log t)^{k-1}\di t.
	$$
	This can be iterated $k$ times on right-hand side integrals, and we get:
	\begin{equation}\label{eq:hi}
	\int_{0}^{\delta}t^{s-\alpha-1}(-\log t)^k\di t=\delta^{s-\alpha}\sum_{j=0}^k\frac{k^{\underline{j}}\,(-\log\delta)^{k-j}}{(s-\alpha)^{j+1}}.
	\end{equation}
	Here, $k^{\underline{j}}:=k(k-1)\cdots(k-j+1)$ denotes the \emph{falling factorial}, with convention $k^{\underline{0}}:=1$. 
	We note that the right-hand side in \eqref{eq:hi} is meromorphic in all of $\Ce$, with only pole of order $k+1$ at $s=\alpha$.
	Using \eqref{eq:hi} in $\int_{0}^{\delta}t^{s-\alpha-1}P_n(-\log t)\di t$ in \eqref{int_eq_n} and interchanging the order of (finite) summation, we get:
	\begin{equation}\label{eq_poln+1}
		\int_{0}^{\delta}t^{s-\alpha-1}P_n(-\log t)\di t=\delta^{s-\alpha}\sum_{j=0}^n\frac{P_n^{(j)}(-\log\delta)}{(s-\alpha)^{j+1}},
	\end{equation}
	where, again, the right hand side is meromorphic in $\Ce$ with a single pole of order $n+1$ at $s=\alpha$.
	Now the statement \eqref{tilde_zeta_polovi} of Proposition~\ref{distr_asy_tube_zeta} follows combining \eqref{eq_tilde}, \eqref{int_eq_n} and \eqref{eq_poln+1}. By $R(s)$, we denote all terms that are holomorphic in $\{\re s>\beta\}$.
\end{proof}
\medskip

\begin{proof}[Proof of Theorem~C] Direct generalization of Corollary~\ref{distr_asy_dist_zeta} to more terms of power-logarithmic type in the asymptotic expansion.
Also, we have changed notation $\alpha_i\leftrightarrow N-\alpha_i$ in Theorem~C in the sense that $N$ is not shown in the expansion but instead it is shown in the expression for $\zeta_A$. 
\end{proof}
\begin{comment}
\begin{remark}
	It is straightforward to generalize Proposition~\ref{distr_asy_tube_zeta} in case if in Equation \eqref{log_asym} we have more than one (leading) term of the power-log type.
	More precisely, let $I\in\eN$, $M_i\in\eR$, $\alpha_i>\beta$ for every $i=0,1,\ldots,I$.
	Then, each term of the form $M_i\varepsilon^{N-\alpha_i+m_i}P_{i,n_i}(-\log\varepsilon)$, where $P_{i,n_i}$ are monic polynomials with real coefficients and of degree $n_i$, will generate a pole of order $n_i+1$ of the tube zeta function $\tilde{\zeta}_A(s)$ at $s=\alpha_i$.
	Furthermore, the tube zeta function will satisfy the following functional equation:
	\begin{equation}
		\tilde{\zeta}_A(s)=(N-s+1)_m\sum_{i=1}^IM_i\delta^{s-\alpha_i}\sum_{j=0}^{n_i}\frac{P_{i,n_i}^{(j)}(-\log\delta)}{(s-\alpha_i)^{j+1}}+g(s),
	\end{equation}
	in the half-plane $\{\re s>\beta\}$ where $g(s)$ is holomorphic in that half-plane and given by Equation \eqref{g(s)} where we replace $M$ by $\sum_{i=1}^IM_i$.
	Of course, the analogue remark can be made for the distance zeta function.
	We leave it to the interested reader to work out the details.
\end{remark}
\end{comment}
\smallskip

\begin{remark}
	It is an interesting open question to ask under which additional conditions we can replace polynomials $P_n$ in \eqref{log_asym} of Proposition~\ref{distr_asy_tube_zeta} by a convergent power series $P$ and get a similar conclusion, but, in this case, with $s=\alpha$ being an essential singularity of the tube zeta function $\tilde{\zeta}_A(s)$.
% 	One would need additional assumptions on the corresponding power series, for instance, that the sum of the power series $P(-\log\varepsilon)=O(\varepsilon^{\gamma})$, for every $\gamma>0$.
% 	The proof of proposition also needs to be modified in this case, since a finite number of partial integrations will not give us a meromorphic extension beyond $s=\alpha$.\edz{Razumijem samo okvirno ovaj remark. Mozda da ne ulazimo u tehnicke detalje u njemu, nego samo pru i zadnju recenicu stavimo?}
	We point out that there exist fractal sets such that their fractal zeta functions possess essential singularities but, at least for the time being, there are no results about the asymptotics of their tube functions; see \cite{essential}.
\end{remark}

\medskip

\subsection{Proof of Theorem B}\label{sec:proof_B}
\smallskip

\begin{proof}[Proof of Theorem~B]

Since the reconstruction of the \emph{distance} zeta function $\zeta_f$ from the coefficients of the distributional expansion of $V_f$ in Theorem~A is rather complicated directly, we will use the functional equation between the distance and the \emph{tube} zeta function in order to describe $\zeta_f$.
Namely, by \cite[Theorem 5.2.6]{fzf}, the relation to the tube zeta function is direct and from it the distance zeta function can be reconstructed.
Set $\delta=1$ in the relative version of equation \eqref{tilde_dist} to obtain
\begin{equation}\label{tube-dist-f}
	\zeta_f(s)=\zeta_{\mathcal O_f(x_0),[0,x_0]}(s)=x_0+(1-s)\tilde{\zeta}_{\mathcal O_f(x_0),[0,x_0]}(s),
\end{equation}
where
\begin{equation}
	\tilde{\zeta}_{\mathcal O_f(x_0),[0,x_0]}(s)=\int_0^1t^{s-2}|\mathcal O_f(x_0)_t\cap[0,x_0]|dt,
\end{equation}
is the \emph{relative} tube zeta function with $A:=\mathcal O_f(x_0),\ \Omega:=[0,x_0]$; for details see \cite[\S4.5.1.]{fzf}.

The proof follows indirectly by combining Theorem~A and Theorem~C, along with \cite[Theorem 5.2.6]{fzf} which reconstructs distributional aysmptotics of the tube function $V_f$ from the relative tube zeta function $\tilde{\zeta}_{\mathcal O_f(x_0),[0,x_0]}$.
The final ingredient is the functional equation \eqref{tube-dist-f}.

Namely, from Theorem~A we know that a complete distributional asymptotic expansion of the tube function $V_f(\varepsilon)$ exists and we also have its explicit formula (3). 
From the proof of Theorem~A we know, moreover, that, for every $N\in\mathbb N$, there exists $k_N$-th primitive $V_f^{[k_N]}(\varepsilon)$ with an asymptotic expansion given by \eqref{eq:ha}. To this we apply Theorem~C. In Theorem~C, we let $A:=\mathcal O_f(x_0)$. Note that $V_{\mathcal O_f(x_0)_\varepsilon}=V_f(\varepsilon)+2\varepsilon,\ \varepsilon>0$.
By letting $N\to+\infty$ we obtain, in steps, the existence of meromorphic continuation of $\zeta_{\mathcal{O}_f(x_0)}(s)=\zeta_f(s)+\frac{2\delta^{s}}{s}$, $\delta\geq(x_0-f(x_0))/2$, to all of $\Ce$.
Also, from Theorem~C we know that this extension is polynomially languid and hence so is $\zeta_f(s)$, which equals to the relative distance zeta function $\zeta_{\mathcal O_f(x_0),[0,x_0]}$ in the notation of \cite[Theorem 5.3.21]{fzf}, where $A:=\mathcal O_f(x_0)$ and $\Omega:=[0,x_0]$.
Finally, the functional equation \eqref{tube-dist-f} gives us the same extension and languidity conclusion for the corresponding relative tube zeta function $\tilde{\zeta}_{\mathcal O_f(x_0),[0,x_0]}$.
Thus, the hypotheses of \cite[Theorem 5.2.6]{fzf} are satisfied for any screen given as a vertical line $\{\re s=\beta\}$, $\beta\in\mathbb R$.
Namely, by Theorem~C, one then knows that $\zeta_{\mathcal O_f(x_0)}$, and hence $\tilde{\zeta}_{\mathcal O_f(x_0),[0,x_0]}$, can be meromorphically extended to the window $\{\re s>\beta\}$ and since $\beta\in\eR$ is arbitrary, we conclude that both functions are meromorphic in $\Ce$. 

Note that we use Theorem~C here merely to prove the existence of the meromorphic extension of the zeta function and the polynomial languidity on windows of the type $\{\re s>\beta\}$, and not to recover the principal parts, since formula \eqref{distance_zeta_polovi1} in Theorem~C is not easily applicable. Instead, once we have proven the existence and the polynomial languidity of zeta functions, we use \cite[Theorem 5.2.6]{fzf} to recover the principal parts at the poles from the distributional expansion of $V_{f}$ from Theorem A $(3)$, as follows.
Indeed, by polynomial languidity proven above, the conclusion of \cite[Theorem 5.2.6]{fzf} holds and one obtains the following distributional asymptotics for the tube function $V_f$:
\begin{equation}\label{eq:dasyi}
	V_f(\varepsilon)=|\mathcal O_f(x_0)_{\varepsilon}\cap[0,x_0]|=_{\mathcal D}\sum_{\omega\in \{\re s>\beta\}}\res\left(\varepsilon^{1-s}\tilde{\zeta}_{\mathcal O_f(x_0),[0,x_0]}(s),\omega\right)+\mathscr{R}(\varepsilon),
\end{equation}
where $\mathscr{R}(\varepsilon)=O(\varepsilon^{1-\beta})$, as $\varepsilon\to 0^+$ in the distributional sense.
From the above equation one sees how the poles of $\tilde{\zeta}_{\mathcal O_f(x_0),[0,x_0]}$ generate terms in the distributional asymptotics.
Namely, if the pole $\omega$ is of order one, then the term it generates is $\varepsilon^{1-\omega}\res(\tilde{\zeta}_{\mathcal O_f(x_0),[0,x_0]},\omega)$. If the pole $\omega$ is of order strictly higher than $1$, then it generates additional logarithmic terms in $\varepsilon$, and the whole principal part of $\tilde{\zeta}_{\mathcal O_f(x_0),[0,x_0]}(s)$ at pole $\omega$ contributes.
More precisely, let $\omega$ be a pole of $\tilde{\zeta}_{\mathcal O_f(x_0),[0,x_0]}(s)$ of order $r+1$, $r\in\eN_0$. Then:
\begin{equation}\label{rezpol_t}
	\res\left({\varepsilon^{1-s}}\tilde{\zeta}_{\mathcal O_f(x_0),[0,x_0]}(s),\omega\right)\!\!=\!\!\varepsilon^{1-\omega}\sum_{n=0}^{r}\frac{(-1)^n\cdot [(s-\omega)^{-n-1}]\tilde{\zeta}_{\mathcal O_f(x_0),[0,x_0]}(s)\cdot \log^n \varepsilon}{n!},
\end{equation}
where $[(s-\omega)^{-n-1}]\tilde{\zeta}_{\mathcal O_f(x_0),[0,x_0]}(s)$ denotes the $(-n-1)$-st coefficient in the Laurent expansion of $\tilde{\zeta}_{\mathcal O_f(x_0),[0,x_0]}(s)$ at $\omega$.
One obtains this by multiplying the Taylor series of $\varepsilon^{1-s}$ at $\omega$ by the Laurent series of $\tilde{\zeta}_{\mathcal O_f(x_0),[0,x_0]}(s)$ at $\omega$, and then extracting the residue.
We omit the details and refer to the proof of \cite[Theorem 5.4.27]{fzf} where exactly the same type of calculation was done.
Comparing the coefficients of the distributional asymptotics given by \eqref{eq:dasyi} and \eqref{rezpol_t} with the distributional asymptotics of $V_f$ from Theorem~A $(3)$, due to the uniqueness of the distributional asymptotics, we recover the principal part of $\tilde{\zeta}_{\mathcal O_f(x_0),[0,x_0]}(s)$ at each pole $\omega$. Denote by $\omega_m:=1-\frac{m}{k+1}$, $m\in\mathbb N,$ the co-exponents of powers of $\varepsilon$ in Theorem A $(3)$. 
 Denoting the terms of the distributional asymptotics from Theorem~A $(3)$ by:
\begin{equation*}
	\varepsilon^{1-\omega_m}\sum_{k=0}^{n_m} l_k^m\log^k \varepsilon,\ l_k^m\in\mathbb R,\ k=0,\ldots,n_m,\ m\in\mathbb N,
\end{equation*}
we deduce that $\omega_m$, $m\in\mathbb N$, are the poles of $\tilde{\zeta}_{\mathcal O_f(x_0),[0,x_0]}(s)$, and that:
\begin{equation*}\label{lk}
	[(s-\omega_m)^{-k-1}]\tilde{\zeta}_{\mathcal O_f(x_0),[0,x_0]}(s)=(-1)^k \cdot k!\cdot l_k^m,\, k=0,\ldots,n_m,\ m\in\mathbb N.
\end{equation*}
For each $\omega_m$, the maximal power of the logarithm that multiplies $\varepsilon^{1-\omega_m}$ in the distributional expansion of $V_f$ determines the order of the pole $\omega_m$.
In fact, for $m\in\{1,2,\ldots,k\}$, the poles are of the first order. 

Finally, by formula \eqref{tube-dist-f}, we get \eqref{dis.zet.para} in Theorem~B.
\end{proof}

The following remark explains how Theorems A and B can be used in practice.
\begin{remark}\label{rem:zadnji}
Suppose that we are able to find (in a relatively easy way) a meromorphic extension to all of $\mathbb C$ of the distance (or geometric) zeta function of a parabolic orbit, that we know exists by Theorem B.
Then, by \cite[Theorem~5.3.21]{fzf}, we can directly determine, from complex dimensions and their residues, the distributional asymptotics of $V_f(\varepsilon)$, as $\varepsilon\to 0^+$.
There is \emph{no need to check languidity conditions}, which is normally necessary for the application of \cite[Theorem~5.3.21]{fzf}, since the existence of the distributional expansion is already known theoretically by Theorem~A, and the polynomial languidity of the distance zeta function follows theoretically from Theorem~B.

Moreover, comparing Theorem~A $(1)$ and $(3)$, the distributional asymptotics of $V_f(\varepsilon)$ for a parabolic diffeomorphism $f$ of multiplicity $k$, $k\in\mathbb N$, is actually valid pointwise up until the order of $O(\varepsilon^{2-\frac{1}{k+1}})$.
\end{remark}
%\begin{remark}\edz{mozda formulirati u teorem?NE BIH}
	%We observe that Theorem \ref{para_dist_zeta} gives the following procedure for obtaining distributional and pointwise asymptotics of the length $|\mathcal{O}(x_0)_{\varepsilon}|$ for parabolic orbits via their geometric or distance zeta functions.
	%If one is able to find a meromorphic extension of the distance (or geometric) zeta function of a parabolic orbit to some open half-plane of the type $\{\re s>\alpha\}$ with $\alpha<\dim_B\mathcal{O}(x_0)$, then it immediately follows that one can determine the distributional asymptotics of the length $|\mathcal{O}(x_0)_{\varepsilon}|$ up to the order of $O(\varepsilon^{1-\beta})$ since the languidity of $\zeta_{\mathcal{O}(x_0)}$ is already a theoretical consequence of Theorem \ref{para_dist_zeta}.
	%In general, the distributional asymptotics will be actually valid pointwise up until the order of $O(\varepsilon^{1-\frac{2k+1}{k+1}})$ as a consequence of Theorem~A.
	%In other words, if one is able to obtain the complex dimensions of $\mathcal{O}(x_0)$ and the corresponding residues, one immediately can recover the (pointwise or distributional) asymptotics of $|\mathcal{O}(x_0)_{\varepsilon}|$ as $\varepsilon\to0^+$.
%\end{remark}

\section{Concluding Remarks and Perspectives}\label{sec_conc}

Let us further explain here the difference between parabolic orbits and hyperbolic orbits from the viewpoint of fractality and existence of nonreal complex dimensions, that we have mentioned in the Introduction and in Subsection~\ref{sec:hyp}.

If we look upon the classical middle-third Cantor set $C$, it has complex dimensions of the type $\log_32+\I pk$, $k\in\mathbb{Z}$, where $p=\frac{2\pi}{\log 3}$ is the so-called {\em oscillatory period}; see \cite[Pages 4--5]{fzf}.
This is, in turn, then reflected in the fact that the tube function of $C$ does not admit a monotonic asymptotic first term, i.e., $C$ is not Minkowski measurable. Instead, the first term is of the type $G(\varepsilon)\varepsilon^{1-\log_32}$, where $G$ is multiplicatively periodic with period $p$.
This geometric significance of complex dimensions led the authors of \cite{fzf} to propose a new definition of fractality, which states that a set is fractal if it possesses a nonreal complex dimension.
  
For instance, the somewhat simplified version of devil's staircase is indeed fractal according to the new definition, since it has nonreal complex dimensions which are, as expected, exactly the same complex dimensions as those of the Cantor set $C$. In addition, it also has the complex dimension $\omega=1$. It coincides with its topological, Minkowski and Hausdorff dimension.
Hence, this set is not fractal according to Mandelbrot's definition; see, e.g., \cite[\S5.5.4]{fzf}.

In this paper we provide examples which strengthen the conjecture that this new definition of fractality, as a form of self-similarity, is a step forward in resolving the conundrum. Namely, the orbits of 1-dimensional dynamical systems generated by parabolic germs that we are investigating here do not posses nonreal complex dimensions, although they have nontrivial Minkowski dimension and are hence fractal according to the classical definition.
Indeed, these orbits are just sequences on the real line that have an accumulation point at the origin. They have Hausdorff dimension equal to zero because of the countable stability property.
Asymptotically, they behave as $1/k$-strings, $k\in\mathbb N$, as defined in \cite{lapidusfrank12} and, as is well known, their Minkowski dimension is nontrivial.
It is our point of view that in this case the Minkowski dimension should be understood as a kind of a measuring tool for the convergence speed of the string, i.e., the sequence generated by the string, rather than an indication of its self-similarity.
The orbits in the parabolic case do not posses any kind of (linear) self-similarity or self-affinity, and this is reflected in the fact that they do not possess any nonreal complex dimensions.
On the other hand, in the hyperbolic case from Subsection~\ref{sec:hyp}, we do get nonreal purely imaginary complex dimensions. We conjecture that this apparent paradox may be justified as follows.
Namely, in the hyperbolic case, the orbit is self-similar in a more general sense, i.e., it can be understood as an inhomogneous self-similar set. This fact is indeed reflected in its complex dimensions; see  \cite[Remark 2.1.87 and Example 5.5.25]{fzf}.
Inhomogeneous self-similar sets were introduced in \cite{BarDemk}; see also \cite{Bar,BakFraMa,Fra1,Fra2}.

Even more interesting and novel in this paper is the insight into the asymptotics of the tube function of orbits in the parabolic case. The tube function does not posses classical asymptotics beyond a certain term, again, due to oscillations, which are precisely described in Theorem A (1). These oscillations, however, are not multiplicatively periodic, or, one might also say, are not resonant to the intrinsic geometry of the orbit.
In this case, the fractal zeta function does not \emph{encode} these oscillations, at least not in its poles, i.e., they do not generate nonreal complex dimensions.
Accordingly, these \emph{nonresonant} oscillations are sufficiently weak in the sense that the complete distributional asymptotics of the tube function of the parabolic orbit in Theorem~A $(3)$ exists, or, in other words, repetitive integration of the tube function will completely dampen the nonresonant oscillations, as is seen from the proof.  
On the other hand, in the hyperbolic case treated in Subsection~\ref{sec:hyp}, similarly as in the case of the Cantor set, the oscillations in the tube functions of orbits are multiplicatively periodic, i.e., \emph{resonant} with the geometry of orbits, and no amount of integration will dampen them out. In other words, the fractal zeta function will encode them as nonreal complex dimensions and the complete distributional asymptotics of the tube function will not exist in the classical sense. For computational details, see Remark~\ref{rem:sve} in Subsection~\ref{sec:hyp}.

To conclude, the notion of fractality in the sense of complex dimensions as in \cite{fzf,lapidusfrank12}, in general, seems to detect \emph{linear} self-similarity of the set. On the other hand, the orbits of parabolic germs may be considered \emph{self-similar} in some broader, \emph{dynamical} or \emph{non-linear}, sense, which the complex dimensions cannot detect. A corresponding generalized notion of \emph{non-linear fractality} could be investigated.

Finally, the continuous time tube function of orbits will have a complete asymptotic expansion in both cases (i.e. parabolic and hyperbolic), see Theorem A $(2)$ and Remark~\ref{rem:sve}\,$(b)$. Thus, this regularization technique dampens both, the \emph{resonant} and the \emph{nonresonant} oscillations in a given orbit. Therefore, it can be thought of as a sort of a dynamical regularization technique.
For further research, it would be interesting to rigorously define the \emph{resonant} and \emph{nonresonant} intrinsic oscillations of general sets and to relate these definitions to the theory of complex dimensions.
\smallskip

Let us mention at the end yet another perspective for further research. We have seen in the paper that there exist two ways of regularization of the tube function $V_f(\varepsilon)$ of parabolic orbits: the \emph{continuous time tube function} and the \emph{distributional expansion} of the tube function, as $\varepsilon\to 0^+$. Both are continuations, in a power-logarithmic scale, of the standard tube function of parabolic orbits, after it becomes oscillatory. It was shown in \cite{nonlin} that the non-existence of the full asymptotic expansion in power-logarithmic scale of the tube function for parabolic orbits prevents seeing the analytic class of a germ from the tube function, while the formal class can be seen from finitely many terms of the tube function \cite{formal}. In \cite{MRRZ3} it was suggested using the \emph{continuous time tube function} instead of the standard tube function, as its dynamical regularization, by embedding in a flow. Here, we have shown that the distributional expansion of $V_f(\varepsilon)$, as $\varepsilon\to 0^+,$ is, on one hand, indirectly related to the continuous time tube function by Theorem A. On the other hand, it is in one-to-one correspondence with complex dimensions of the orbit (the poles of the distance zeta function of the orbit) and their principal parts, see Corollary~\ref{cor:oneone}. Moreover, the tube zeta function of the orbit, which is directly related to the distance zeta function by a functional equation, is, by its definition, the Mellin transform of the tube function of the orbit \cite[\S2.2.2]{fzf}. This passage, by an integral transform, from the original $\varepsilon$-plane, where we have an asymptotic expansion of the tube function, to the complex $\zeta$-plane, where we analyze singularities of the tube zeta function, strongly resembles the Borel-Laplace resummation method, see, e.g., \cite{sauzin}. It is a central method in the so-called \emph{resurgent approach} to analytic classification of parabolic germs by Ecalle \cite{ecalle}, later explained also in \cite{candel} or \cite{sauzin}. By Borel transform, one passes from the original $z$-plane, that is, from the algebra of formal series to which formal embedding of a germ in a vector field belongs, to the algebra of so-called \emph{resurgent functions} in the Borel $\zeta$-plane, where an appropriate analysis of singularities (using the so-called \emph{alien calculus}) recovers the analytic class of the germ. All mentioned leaves us wondering if, by an appropriate treatment and analysis of complex dimensions of one orbit of a parabolic germ, we can reconstruct the moduli of analytic classification of the germ, in parallel with the method of Ecalle.

\section{Appendix}\label{sec:last}

\subsection{Geometric zeta function of the shifted $a$-string}\label{subsec:sas}\  

In this subsection we prove Theorem~\ref{shifted_a_prop} which generalizes \cite[Theorem 6.21]{lapidusfrank12} about the meromorphic extension and complex dimensions of the geometric zeta function of the standard $a$-string to the case of the \emph{shifted $a$-string} defined by \eqref{shift-a}. It also extends \cite[Theorem 6.21]{lapidusfrank12} in the sense that we give an explicit formula for the meromorphic extension of the geometric zeta function of the (shifted) $a$-string. Theorem~\ref{shifted_a_prop} is crucial in the proof of Proposition~\ref{prop:geo_k} about meromorphic extension of the distance zeta function of model orbits in Subection~\ref{subsec:model}. 
The first part of the proof is completely analogous to the proof of \cite[Theorem 6.21]{lapidusfrank12}, but, in order to keep the second part clear, we give here a brief sketch.

\begin{theorem}\label{shifted_a_prop} Let $a,\,b>0$ and let 
\begin{equation}\label{shift-a}
\mathcal L_{a,b}:=\{\ell_j(b):j\in\mathbb N_0\},\ \ell_j(b):=(j+b)^{-a}-(j+1+b)^{-a},\ j\in\mathbb N_0,
\end{equation} be a shifted $a$-string. 
\smallskip

$(1)$ The geometric zeta function $\zeta_{\mathcal{L}_{a,b}}(s):=\sum_{j=0}^{\infty}\ell_j(b)^s$ can be meromorphically extended from $\{s\in\mathbb C:\re s>\frac{1}{a+1}\}$ to all of $\Ce$. The poles of $\zeta_{\mathcal{L}_{a,b}}$ are located at $\frac{1}{a+1}$ and at $($a subset of$)$ the set of the points $\frac{-m}{a+1}$, $m\in\eN$, and they are all simple. In particular, the box dimension of $\mathcal{L}_{a,b}$ is $D=\frac{1}{a+1}$, and this is the only pole of $\zeta_{\mathcal{L}_{a,b}}$ with a
positive real part.

More precisely, for a fixed $M\in\eN_0$, $\zeta_{\mathcal{L}_{a,b}}(s)$ can be expressed as a sum of Hurwitz zeta functions in the following sense:
\begin{equation}\label{zetaab_W}
	\zeta_{\mathcal{L}_{a,b}}(s)=\sum_{n=0}^{M}W(s,a,n)\zeta\big((a+1)s+n,b\big)+a^sR(s),
\end{equation}
where $R(s)$ defined and holomorphic for $\re s>-\frac{M}{a+1}$. Moreover, for every $\gamma>0$, $R(s)=O\big(|s|^{M+1}\big)$, as $|s|\to+\infty$ on $\{s\in\mathbb C:\,\mathrm{Re}\,s\geq-\frac{M}{a+1}+\gamma\}$.
% The functions $W(s,n)$ are entire and given as
% \begin{equation}\label{W(s,a,n)}
% 	W(s,a,n):=a^s\sum_{m=1}^{n}\left(-\frac{1}{a}\right)^m{s\choose m}\sum_{|{\mathbf{q}}_m|=n-m}\,\prod_{j=1}^m{-a\choose q_j+2},
% \end{equation}
% where the inner sum goes over all multiindices ${\mathbf{q}}_m:=(q_1,\ldots,q_m)\in\eN_0^m$ such that their length $|{\mathbf{q}}_m|:=q_1+\cdots+q_m$ equals to $n-m$ with the convention that $W(s,0):=a^s$.
The functions $W(s,a,n)$ are entire, more precisely, polynomials in $s$ of degree $n$ modulo the factor $a^s$, and polynomials in $a$ of degree $n$. They are given as the following finite triple sums:
\begin{align}\begin{split}\label{Wsan}
	W(s,a,n):=a^s\sum_{k=0}^{n}\left(-\frac{1}{a}\right)^k{s\choose k}\sum_{i=0}^k(-1)^i{k\choose i}\sum_{j=0}^i(-a)^j&{i\choose j}{(i-k)a\choose n+k-j},\\
	&\qquad n\in\mathbb N_0,\ s\in\mathbb C.
\end{split}\end{align}
In particular, we have that $W(s,a,0)=a^s$.
%Furthermore, $W(s,a,n)$ are polynomials, modulo the factor $a^s$, of degree $n$ in both variables, $s$ and $a$.
% Furthermore, the functions $W(s,a,n)$ can be expressed equivalently, as the following double sum
% \begin{equation}
% \begin{aligned}
% 	W(s,a,n)&:=a^s\sum_{m=0}^{n}\left(-\frac{1}{a}\right)^m{s\choose m}\sum_{i=0}^m(-1)^i{m\choose i}{(i-m)a\choose n+m}\cdot\\
% 	&\cdot {}_2F_1\left(-i,-n-m;(i-m)a-n-m+1;-a\right)
% \end{aligned}
% \end{equation}
\smallskip
  
$(2)$ If we choose for the screen the vertical line $\{\re s=\sigma\}$, where $\sigma\in(-M/(a+1),-(M-1)/(a+1)]$, then $\zeta_{\mathcal{L}_{a,b}}$ is languid with exponent $\kappa_f=M+1$.
\end{theorem}

Before the proof, we state a conjecture which we have obtained by computer algebra software, but which we are so far unable to prove: that the only poles of $\zeta_{\mathcal{L}_{a,b}}$ are $$\omega_{a,2m}:=\frac{1-2m}{a+1},\ m\in\eN_0.$$
To get the poles and the corresponding residues of geometric zeta function $\zeta_{\mathcal L_{a,b}}$, we evaluate the functions $W(s,a,n)$ at $\omega_{a,n}:=\frac{1-n}{a+1}$, $n\in\mathbb N_0$, since these are the poles of the corresponding Hurwitz zeta function $\zeta((a+1)s+n,b)$.
In the case $a=1$ it is straightforward to show that $W(s,1,n)={-s\choose n}$, as expected from Proposition \ref{geo_k1}; therefore, $W(\omega_{1,n},1,n)$ is zero when $n$ is odd and nonzero when $n$ is even.
Computer algebra software suggests that this could be a general fact for any $a>0$, but we are unable to prove it rigorously.
\medskip

In the proof of Theorem~\ref{shifted_a_prop}, we need the following technical lemmas:
\begin{lemma}[Generalized inclusion-exclusion principle]\label{lem:iep}
	Let $U$ be a finite universal set and let $B_i\subseteq U$, $i=1,\ldots,k$, $k\in\mathbb N$. Let $f\colon U\to \Ce$ be a complex function defined on $U$.
	Then
	\begin{equation}\label{ine}
		\sum_{x\in U\setminus\left(\bigcup_{i=1}^{k}B_i\right)}f(x)=\sum_{I\subseteq\{1,\ldots,k\}}(-1)^{|I|}\sum_{x\in B_I}f(x),
	\end{equation}
	where $B_I:=\bigcap_{i\in I}B_i$ and $B_{\emptyset}:=U$.
\end{lemma}
\begin{proof}
	Write the right-hand side of \eqref{ine} as
	\begin{equation*}
		\sum_{x\in U}f(x)\sum_{I\subseteq\{1,\ldots,k\}}(-1)^{|I|}\chi_{B_I}(x),
	\end{equation*}
	where $\chi_{B}$ denotes the characteristic function of the set $B$.
	Next, one immediately sees that the inner sum above equals to 1 if $x$ does not belong to any $B_i$, $i=1,\ldots,k$, since then the only nonzero term is $\chi_{B_{\emptyset}}(x)=1$, when $I=\emptyset$.
	On the other hand, if $J=\{i\in\{1,\ldots,k\}:x\in B_i\}\neq\emptyset$, then we let $j:=|J|$. In that case, $x\in B_I$ if and only if $I\subseteq J$, and the inner sum is then
	$
	\sum_{i=0}^{j}(-1)^i{j\choose i}=0.
	$
\end{proof}
\begin{lemma}\label{lem:q} Let $a,\,b>0$ and let
\begin{equation}\label{Qn}
	Q_k(s):=\sum_{j=0}^{\infty}h_j^k(j+b)^{-(a+1)s},\ \re s>\frac{1}{a+1},\ k\in\mathbb N_0, 
\end{equation}
where
\begin{equation}\label{hj}
h_j:=(j+b)\int_0^{1/(j+b)}\left((1+t)^{-a-1}-1\right)\di t,\ j\in\mathbb N_0.
\end{equation}
Then $h_j=O(\frac{1}{j+b})$, as $j\to\infty$, and $Q_k$, $k\in\mathbb N_0$, define holomorphic functions on the open half plane $\re s>\frac{1-k}{a+1}$. Moreover, $Q_k$, $k\in\mathbb N_0$, can be meromorphically extended to all $\mathbb C$, with simple poles at $($a subset of\,$)$ the set:
$$
\mathcal P_k:=\left\{\frac{1-m}{a+1}:\ m\in\eN_0,\ m\geq k\right\}.
$$
\end{lemma}

\begin{proof}
First, write the integrand in \eqref{hj} as $g(t)-g(0)$, where $g(t)=(1+t)^{-a-1}$, and apply the mean value theorem to the integrand. We get: $$|h_j|\leq(j+b)\int_0^{1/(j+b)}|g(t)-g(0)|\di t\leq(j+b)|g'(t_j)|\cdot \int_0^{1/(j+b)}t\di t\leq\frac{1}{2}\frac{a+1}{j+b},$$ where $t_j\in(0,\frac{1}{j+b})$ and $|g'(t_j)|=|-a-1|(1+t_j)^{-a-2}\leq a+1$, $j\in\mathbb N$. Therefore, $h_j=O(\frac{1}{j+b})$, $j\to\infty$.

For $k=0$, $Q_0(s)$ is equal to the Hurwitz zeta function $\zeta\big((a+1)s,b\big)$, so it has a simple pole at $s=\frac{1}{a+1}$. Note that $Q_k(s)$, $k\in\mathbb N$, given by \eqref{Qn}, is absolutely convergent in the open half-plane ${\re s>\frac{1-k}{a+1}}$, which follows directly since $h_j=O(\frac{1}{j+b})$, $j\to\infty$. We obtain its meromorphic extension to $\mathbb C$ and analyse its poles step by step, by extending meromorphically to $\{\re s>-\frac{M_k+k}{a+1}\}$, for every $M_k\in\mathbb N_0$.

In \eqref{hj}, we use the power series expansion of $(1+t)^{-a-1}=\sum_{m=0}^{\infty}{-a-1\choose m}t^m$ in order to obtain a power series expansion for $h_j$:
\begin{equation*}
\begin{aligned}
	h_j&=(j+b)\int_0^{1/(j+b)}\sum_{m=1}^{\infty}{-a-1\choose m}t^m\di t=-\frac{1}{a}\sum_{m=1}^{\infty}{-a\choose m+1}(j+b)^{-m}.
\end{aligned}
\end{equation*}
Taking the $k$-th power of the above series (in the Cauchy sense), we obtain a power series for $h_j^k$:
\begin{equation}\label{eq:hn}
\begin{aligned}
	h_j^k&=\left(-\frac{1}{a}\right)^k(j+b)^{-k}\left[\sum_{m=0}^{\infty}{-a\choose m+2}(j+b)^{-m}\right]^k=\\
	&=\left(-\frac{1}{a}\right)^k(j+b)^{-k}\sum_{m=0}^{\infty}c_{m,k}(j+b)^{-m}.
\end{aligned}
\end{equation}
Here,
\begin{equation}\label{cmn}
	c_{m,k}:=\sum_{|\mathbf{r}_k|=m}\prod_{i=1}^{k}{-a\choose r_i+2},
\end{equation}
where the sum goes over all multiindices $\mathbf{r}_k=(r_1,\ldots,r_k)\in\eN_0^k$ of length $m$.
For any integer $M_k\in\mathbb N_0$, from \eqref{eq:hn} we have:
\begin{equation*}
	h_j^k=\left(-\frac{1}{a}\right)^k(j+b)^{-k}\sum_{m=0}^{M_k}c_{m,k}(j+b)^{-m}+O((j+b)^{-k-M_k-1}).
\end{equation*}
Therefore, by \eqref{Qn},
\begin{equation}\label{eq:fin}
	Q_k(s)=\left(-\frac{1}{a}\right)^k\sum_{m=0}^{M_k}c_{m,k}\cdot\zeta\big((a+1)s+k+m,b\big)+r_{k}(s),
\end{equation}
where $\zeta(\cdot,\cdot)$ above denotes the Hurwitz zeta function, holomorphic except when the first variable equals to $1$ having a simple pole at that point.
The function $r_k$ is holomorphic in the open half-plane $\{\re s>-\frac{M_k+k}{a+1}\}$ and, for every $\gamma>0$, $r_k(s)=O\left((j+b)^{-k-M_k-1}\right)$, as $|s|\to\infty$ in $\{\re s\geq-\frac{M_k+k}{a+1}+\gamma\}$. We justify this claim by a standard argument analogous to the one in \eqref{func_eq_bin} in the proof of Proposition~\ref{geo_k1}.
Since $M_k\in\mathbb N_0$ is arbitrarily large, we conclude that, by formula \eqref{eq:fin}, $Q_k$ can be meromorphically extended to all of $\Ce$, with possible simple poles at $(1-m)/(a+1)$ for $m\geq k$, $m\in\eN_0$,\ $k\in\mathbb N$.
\end{proof}

\begin{lemma}\label{lem:chu}
Let $n\in\mathbb N$. Then, for $k\in\{1,\ldots,n\}$, it holds that: 
\begin{equation*}
	\sum_{|\mathbf{r}_k|=n-k}\prod_{j=1}^{k}{-a\choose r_j+2}=\sum_{i=0}^k(-1)^i{k\choose i}\sum_{l=0}^i{i\choose l}(-a)^l{(i-k)a\choose k+n-l},\end{equation*}
where $\mathbf{r}_k=(r_1,\ldots,r_k)\in\eN_0^k$ denote the multiindices of length $n-k$, $k=1,\ldots,n$.
\end{lemma}

\begin{proof} The combinatorial proof uses the Chu-Vandermonde identity and the inclusion-exclusion principle.
We let

$$
H(n,k):=\sum_{|\mathbf{r}_k|=n-k}\prod_{j=1}^{k}{-a\choose r_j+2},
\ k=1,\ldots,n.$$
Changing the multiindices of summation to $q_j:=r_j+2$, $j=1,\ldots,k,$ so that the sum now goes over all multiindices $\mathbf{q}_k\in X$, where
\begin{equation}\label{iks}
	X:=\big\{\mathbf{q}_k\in\eN_{0}^k:|\mathbf{q}_k|=n+k,\ q_i\geq 2,\ i\in\{1,\ldots,k\} \big\},
\end{equation}
we get
$$H(n,k)=\sum_{\mathbf{q}_k\in X}\prod_{j=1}^{k}{-a\choose q_j}.$$
Without the limitations $q_i\geq 2$, $i=1,\ldots,k,$ we could directly use the Chu-Vandermonde identity. However, due to limitations, we first split the above sum by the inclusion-exclusion principle.
Let
\begin{equation*}
	A_i:=\big\{\mathbf{q}_k\in\eN_{0}^k:|\mathbf{q}_k|=n+k,\ q_i\in\{0,1\}\big\},\ i=1,\ldots,k.
\end{equation*}
For any subset $I\subseteq\{1,\ldots,k\}$, we let
\begin{equation*}
	A_I:=\bigcap_{i\in I}A_i,
\end{equation*}
with the convention 
\begin{equation*}
	A_{\emptyset}:=\{\mathbf{q}_k\in\eN_{0}^k:|\mathbf{q}_k|=n+k\},\ \text{ and }A^c:=A_{\emptyset}\setminus A,\ A\subseteq A_\emptyset.
\end{equation*}
By \eqref{iks}, we have 
\begin{equation*}
	X=\left(\bigcup_{i=1}^{k}A_i\right)^c.
\end{equation*}
By the generalization of the inclusion-exclusion formula from Lemma~\ref{lem:iep}, putting $U:=A_\emptyset$ as the finite universal set and $B_i:=A_i$, $i=1,\ldots,k,$ we get:
\begin{equation}\label{hprod}
	H(n,k)=\sum_{I\subseteq\{1,\ldots,k\}}(-1)^{|I|}\sum_{\mathbf{q}_k\in A_I}\prod_{j=1}^{k}{-a\choose q_j}.
\end{equation}
Note that the inner sum above depends only on $|I|$.
Indeed, if $|I|=i$, this means that  exactly $i$ of the components of the multiindex $\mathbf{q}_k$ are fixed to be equal to $0$ or to $1$. We rearrange the product in \eqref{hprod} so that these are exactly the first $i$ components $q_1,\ldots,q_i$:
\begin{equation}\label{hprod1}
	H(n,k)=\sum_{i=0}^k(-1)^i{k\choose i}\sum_{\mathbf{q}_k\in A_{\{1,\ldots,i\}}}\prod_{j=1}^{k}{-a\choose q_j}.
\end{equation}
Consider $\mathbf{q}_k\in A_{\{1,\ldots,i\}}$. Let us now assume that $l$ out of $i$ {\em fixed} components of $\mathbf{q}_k$ are equal to $1$, and the other $i-l$ fixed components are equal to $0$.
This means that in the product in \eqref{hprod1} we have $l$ factors of the type ${-a\choose 1}=-a$ and $i-l$ factors of the type ${-a\choose 0}=1$, so we can factor out $(-a)^l$.
Furthermore,  the \emph{free}, i.e., \emph{nonfixed} components $q_{i+1},\ldots,q_k$ of $\mathbf{q}_k$ then must satisfy $q_{i+1}+\cdots+q_k=k+n-l$.
Since there are ${i\choose l}$ ways to choose $l$ components out of $i$ (to be equal to $1$) and $l$ can be any integer between $0$ and $i$ (inclusive), we rearrange further the inner sum in \eqref{hprod1} by introducing another sum over $l$ (the number of components equal to 1) in order to obtain: 
\begin{equation}\label{eq:ab}
	H(n,k)=\sum_{i=0}^k(-1)^i{k\choose i}\sum_{l=0}^i{i\choose l}(-a)^l\sum_{q_{i+1}+\cdots+q_k=k+n-l}\prod_{j=i+1}^{k}{-a\choose q_j}.
\end{equation}
Finally, we apply the (generalized) Chu-Vandremonde identity (see, e.g., \cite[pp.\ 59-60]{askey}) to the innermost sum of \eqref{eq:ab}. We get:
\begin{equation*}
	H(n,k)=\sum_{i=0}^k(-1)^i{k\choose i}\sum_{l=0}^i{i\choose l}(-a)^l{(i-k)a\choose k+n-l}.
\end{equation*}
\end{proof}

\medskip

\noindent \emph{Proof of Theorem~\ref{shifted_a_prop}.}

Let $\mathcal L_{a,b}:=\{\ell_j(b):j\in\mathbb N\}$ be a shifted $a$-string, as in \eqref{shift-a}. We first determine the first term of the asymptotic expansion of $\ell_j(b),\ j\in\mathbb N$:
\begin{equation*}
	\ell_j(b)=(j+b)^{-a}-(j+b+1)^{-a}=a\int_{j+b}^{j+b+1}x^{-a-1}\di x=a(j+b)^{-a-1}(1+h_j),
\end{equation*}
where we have let
%\begin{equation}
	%H(j):=a\int_{j+b}^{j+b+1}(x^{-a-1}-(j+b)^{-a-1})\di x.
%\end{equation}
%ext we write further
%\begin{equation}
	%\ell_j(b)=a(j+b)^{-a-1}(1+h_j),
%\end{equation}
\begin{equation*}
	h_j:=\int_{j+b}^{j+b+1}((j+b)^{a+1}x^{-a-1}-1)\di x.
\end{equation*}
After changing variable of integration to $t:=(j+b)^{-1}x-1$, we have:
\begin{equation*}
	h_j=(j+b)\int_0^{1/(j+b)}\left((1+t)^{-a-1}-1\right)\di t.
\end{equation*}
By Lemma~\ref{lem:q}, $h_j=O(\frac{1}{j+b})$, as $j\to\infty$. Therefore, for $M\in\mathbb N_0$, we have:
\begin{equation}\label{sum_lb}
	\begin{aligned}
	\ell_j(b)^s&=a^s(j+b)^{-(a+1)s}(1+h_j)^s=\\
	&=a^s(j+b)^{-(a+1)s}\left(\sum_{n=0}^M{s\choose n}h_j^n+{s\choose M+1}O\left((j+b)^{-M-1}\right)\right),\ j\to\infty.
	\end{aligned}
\end{equation}
Moreover, the implicit constant in $O\left((j+b)^{-M-1}\right)$ can be taken as $1$, uniformly for every $M\in\mathbb N_0$, which can be easily seen by bounding the remainder term of the $M$-th Taylor polynomial, similarly as was done in the proof of Proposition \ref{geo_k1}.

Now, summing \eqref{sum_lb} over all $j\geq 0$, we get:
\begin{align}\label{zeta_lab}
	\zeta_{\mathcal{L}_{a,b}}(s)=\sum_{j=0}^{\infty}\ell_j(b)^s=&a^s\sum_{n=0}^M{s\choose n}\sum_{j=0}^{\infty}h_j^n(j+b)^{-(a+1)s}+a^sR(s)=\nonumber\\
	=&a^s\sum_{n=0}^M{s\choose n}Q_n(s)+a^sR(s),\ \re s>\frac{1}{a+1},
\end{align}
where  we have let:
\begin{align}
  &Q_n(s):=\sum_{j=0}^{\infty}h_j^n(j+b)^{-(a+1)s},\ n\in\mathbb N_0,\nonumber\\
	&R(s):={s\choose M+1}\sum_{j=0}^{\infty}(j+b)^{-(a+1)s}O\left((j+b)^{-M-1}\right).\label{eq:r}
\end{align}
The first sum in \eqref{zeta_lab} is absolutely and uniformly convergent on compact subsets of the open half-plane $\{\re s>1/(a+1)\}$ and the interchange of order of summation is justified. The sum \eqref{eq:r} converges absolutely and uniformly on compact subsets of the open half plane $\{\re s>-M/(a+1)\}$ and thus, is holomorphic in that half plane. The growth estimate for $R$, as $|s|\to\infty$, follows immediately.  Thus, \eqref{zeta_lab} is a functional equation between holomorphic functions in the half plane $\{\re s>1/(a+1)\}$. 
Next, we show that the right-hand side can be meromorphically extended to the open half plane $\{\re s>-M/(a+1)\}$. Indeed, by Lemma~\ref{lem:q}, $Q_0$ has a simple pole at $\frac{1}{a+1}$ and $Q_n$, $n\geq 1$, are meromorphically extendable to all of $\Ce$, with simple poles belonging to a subset of $\{-\frac{l}{a+1}: \ l\in\mathbb N_0\}$. Letting $M\to\infty$, this proves the statement of Theorem~\ref{shifted_a_prop} about the meromorphic extension and the set of poles of $\zeta_{\mathcal L_{a,b}}$.
\smallskip

The remainder of the proof of $(1)$ concerns proving the \emph{formula for residues} \eqref{Wsan}.
Fix $M\in\mathbb N_0$ and let $M_k:=M-k$, $k=1,\ldots,M$. Then all the remainder functions $r_k(s)$ from the expressions \eqref{eq:fin} for $Q_k$, $k=1,\ldots,M,$ in the proof of Lemma~\ref{lem:q} are holomorphic in $\{\re s>-M/(a+1)\}$ and, for every $\gamma>0$, $r_k(s)=O\left((j+b)^{-(M+1)}\right)$, as $|s|\to+\infty$ in the half-plane $\{\re s\geq-M/(a+1)+\gamma\}$.
Substituting \eqref{eq:fin} in \eqref{zeta_lab}, we have:
	\begin{equation}\label{zeta_lab_2}
	\begin{aligned}
	&\zeta_{\mathcal{L}_{a,b}}(s)=a^s\zeta\big(-(a+1)s,b\big)+\\
	&\ \ +a^s\sum_{k=1}^M{s\choose k}\left(\left(-\frac{1}{a}\right)^k\sum_{m=0}^{M-k}c_{m,k}\zeta\big((a+1)s+k+m,b\big)+r_{k}(s)\right)+a^sR(s)=\\
	&\!\!\!=a^s\zeta\big(\!\!-\!(a+1)s,b\big)\!+\!a^s\sum_{k=1}^M{s\choose k}\!\!\left(-\frac{1}{a}\right)^k\sum_{m=0}^{M-k}c_{m,k}\zeta\big((a+1)s+k+m,b\big)\!+\!a^s\tilde{R}(s)\!\!=\\
	&\!\!\!=a^s\zeta\big(\!\!-\!(a+1)s,b\big)+a^s\sum_{k=1}^M{s\choose k}\left(-\frac{1}{a}\right)^k\sum_{n=k}^{M}c_{n-k,k}\zeta\big((a+1)s+n,b\big)+a^s\tilde{R}(s),\\
	&\qquad\qquad\qquad\qquad\qquad\qquad\qquad\qquad\qquad\qquad\qquad\qquad\qquad\qquad\re s>-\frac{M}{a+1}.
	\end{aligned}
\end{equation}
Here, we have incorporated the sum involving all the remainder functions $r_k$,\newline $k=1,\ldots,M$, along with the original $R$, in $\tilde R$. The constants $c_{m,k}$ are given in \eqref{cmn}. It follows directly that, for every $\gamma>0$, $$\tilde{R}(s)=O(R(s))=O\big(|s|^{M+1}\big),$$ as $|s|\to+\infty$ in the half-plane $\{\re s\geq-\frac{M}{a+1}+\gamma\}$.

\noindent Changing the order of summation in finite sums,
	\begin{equation}\label{zeta_lab_3}
	\begin{aligned}
	&a^s\sum_{k=1}^M{s\choose k}\Big(-\frac{1}{a}\Big)^k\sum_{n=k}^{M}c_{n-k,k}\zeta\big((a+1)s+n,b\big)=\\
	&=a^s\sum_{n=1}^{M}\zeta\big((a+1)s+n,b\big)\sum_{k=1}^nc_{n-k,k}{s\choose k}\left(-\frac{1}{a}\right)^k,\ \re s>-\frac{M}{a+1}.
	\end{aligned}
\end{equation}
Substituting \eqref{zeta_lab_3} into \eqref{zeta_lab_2}, we get \eqref{zetaab_W}, with: 
\begin{equation}\label{prevander}
	W(a,s,n):=a^s\sum_{k=1}^n\left(-\frac{1}{a}\right)^k{s\choose k}\sum_{|\mathbf{r}_k|=n-k}\prod_{j=1}^{k}{-a\choose r_j+2},\ n=1,\ldots,M.
\end{equation}
Recall that $\mathbf{r}_k=(r_1,\ldots,r_k)\in\eN_0^k$ are multiindices of length $n-k$. Obviously, $W(a,s,n)$ in \eqref{prevander} are polynomials, modulo the factor $a^s$, in $s$ and in $a$ of degree $n$ in both variables.
Indeed, the factor $(-a)^{-k}$ can be canceled out by the innermost product (it has $k$ factors, and from each we can factor out $-a$). Using combinatorial Lemma~\ref{lem:chu} to the inner sum in \eqref{prevander}, we finally obtain the expression of $W(a,s,n)$ as a finite triple sum, as stated in \eqref{Wsan}. This is particularly useful for practical computation of residues of $\zeta_{\mathcal{L}_{a,b}}$. Indeed, in \eqref{Wsan}, we may actually sum from $k=0$ or from $k=1$ in the case $n>0$, since the term under the sum for $k=0$ then equals to $0$. In the case $n=0$ and $k=0$ the term under the sum equals to $1$. This yields exactly $W(s,a,0)=a^s$. 

\smallskip
To conclude the proof of $(1)$, it is left to show that $0$ is not a pole of $\zeta_{\mathcal L_{a,b}}$, and that, moreover, some of the poles $\frac{-m}{a+1},\ m\in\mathbb N,$ may cancel by zeros. Note that $W(a,s,n)$ is a polynomial of degree $n$ in $s$, multiplied by the factor $a^s$. Consequently, it may happen that the pole $\omega_n=\frac{1-n}{a+1}$ of the shifted Hurwitz zeta function $\zeta((a+1)s+n,b)$ is cancelled by a zero of $W(a,s,n)$, $n\in \mathbb N$. In particular, this happens when $n=1$, i.e., for $\omega_1=0$, since we can factor out $s$ in $W(a,s,1)$. 
\smallskip

$(2)$ We derive the languidity estimates of $\zeta_{\mathcal{L}_{a,b}}$ from the decomposition into the sum of shifted Hurwitz zeta functions given in \eqref{zetaab_W}. For a fixed $M\in\mathbb N_0$, the decomposition  \eqref{zetaab_W} is valid in the open half-plane $\{\re s>-\frac{M}{a+1}\}$. Choose for the screen the vertical line $\{\re s=\sigma\}$, where $\sigma\in\big(-\frac{M}{a+1},-\frac{M-1}{a+1}\big]$.
Due to the estimate on $R(s)$ proven in $(1)$, the holomorphic remainder $a^sR(s)$ is super languid for this screen with exponent $\kappa_f=M+1$.
Furthermore, the functions $W(a,s,n)$ are super languid with exponent $\kappa_n=n$, since, by \eqref{Wsan}, all of the binomials involving $s$ that appear in $W(a,s,n)$ are in fact polynomials in $s$ of degree at most $n$. These finite sums of polynomials in $s$ of degree at most $n$ are of order $O(|s|^n)$, as $|s|\to\infty$, in any finite vertical strip, while the factor $a^s$ is unformly bounded in finite vertical strips.
Recall the bounds \eqref{hurwitz_bound} for the exponent of growth of the shifted Hurwitz zeta functions $s\mapsto \zeta\big((a+1)s+n,b\big)$ from \eqref{zetaab_W} along vertical lines $\{\mathrm{Re} \,s=\sigma\}$:
\begin{equation*}\label{hurwitz_shift_bound}
		\mu_a(\sigma;b;n)\leq\begin{cases}
	                      	\frac12-(a+1)\sigma-n,& \textnormal{ if }\sigma\leq -\frac{n}{a+1},\\
	                      	\frac{1}{2}(1-(a+1)\sigma-n),&\textnormal{ if } -\frac{n}{a+1}\leq\sigma\leq -\frac{n-1}{a+1},\\
	                      	0,&\textnormal{ if }\sigma\geq -\frac{n-1}{a+1}.
	                      \end{cases}
\end{equation*}
From the above discussion we deduce that each of the products $W(a,s,n)\zeta((a+1)s+n,b)$, $n=0,\ldots,M-1$, is languid with exponent $\frac12-(a+1)\sigma$, except for $n=M$ in which case the exponent is $\frac{1}{2}(1-(a+1)\sigma)+\frac{M}{2}$. Thus, $M$ is a uniform lower bound for the languidity exponent of all these products, and $M+1/2$ is a uniform upper bound\footnote{Recall that we have fixed $\sigma\in(-\frac{M}{a+1},-\frac{M-1}{a+1})$.}.
From this we can see that the holomorphic remainder has the dominating languidity exponent $M+1$.
We also note here that, in the original statement of \cite[Theorem 6.21]{lapidusfrank12}, the claim about the languidity exponent is not true.
The reason is that the authors forgot to take into account the growth rate of the binomials ${s\choose k}$, as well as the growth rate of the holomorphic remainder.
%Finally, if we choose for the screen the vertical line $\{\re s=\sigma\}$ where $\sigma\in(-M/(a+1),-(M-1)/(a+1))$, then $\zeta_{\mathcal{L}_{a,b}}$ is languid with exponent $\kappa_f=M+1$, since the holomorphic remainder $a^sR(s)$ dominates the rest of the terms in \eqref{zetaab_W}.
\qed

\subsection{Proofs from Section~\ref{sec:proofA}}\
\smallskip

\noindent Let $\mathcal{O}_f(x_0)$, $0<x_0<1$, be an orbit generated by $f\in\mathrm{Diff}(\mathbb R,0)$, where $f$ is either parabolic, $f(x)=x+o(x)$, or hyperbolic, $f(x)=\lambda x+o(x)$, $0<\lambda <1$.
\smallskip

\noindent From \eqref{eq:epsi} and \eqref{eq:ac} (for more details, see e.g. \cite{MRRZ3}):
\begin{align}\label{eq:formula}
V_f(\varepsilon)&=2\varepsilon \cdot n_{\varepsilon}+f^{\circ n_\varepsilon}(x_0),\nonumber \\
V_f^\mathrm{c}(\varepsilon)&=2\varepsilon \cdot \tau_{\varepsilon}+f^{\circ \tau_\varepsilon}(x_0),\ \varepsilon>0.
\end{align}
The difference from \cite{MRRZ3} is that $\varepsilon$-neighborhoods, for simplicity, are considered only inside the segment $[0,x_0]$, so we do not add $+2\varepsilon$ in \eqref{eq:formula}. Here, $n_\varepsilon\in\mathbb N_0$ is the critical index of separation between the tail and the nucleus: $n_\varepsilon$ is the number of points in the tail and $f^{\circ n_\varepsilon}(x_0)$ is the first point of the orbit belonging to the nucleus. That is, $n_\varepsilon\in\mathbb N$ is characterized by the inequality:
\begin{equation}\label{eq:ine}
g(f^{\circ (n_\varepsilon)}(x_0))<2\varepsilon, \ g(f^{\circ (n_\varepsilon-1)}(x_0))\geq 2\varepsilon.
\end{equation}
Here, $g=\mathrm{id}-f$. 

For the continuous time counterparts $\tau_\varepsilon$ and $f^{\circ\tau_\varepsilon}(x_0)$, we embed $f$ in a flow $\{f^{\circ t}:t\in\mathbb R\}\subseteq\mathrm{Diff}(\mathbb R,0)$, as its time-1 map ($f^{\circ 1}\equiv f$), see \cite{MRRZ3} or \cite{MRRZ2}. A canonical choice of the flow is the standard flow in which the complex extension of $f$ from $\mathrm{Diff}(\mathbb C,0)$ can be embedded as its time-one map. This can be done sectorially on an attracting petal of opening $\frac{2\pi}{k}$ around $\mathbb R_+$ in the parabolic case of multiplicity $k$, or globally on $(\mathbb C,0)$ in the hyperbolic case, see e.g.\cite{boetcher} for description of dynamics of complex  analytic germs from $\mathrm{Diff}(\mathbb C,0)$. Embedding in a flow is equivalent to finding a \emph{Fatou coordinate} $\Psi$ for $f$. The Fatou coordinate is a  trivializing (time) coordinate for the flow, strictly monotonic on $(\mathbb R_+,0)$, see \cite{MRRZ2}:
\begin{equation}\label{eq:fatou}
\Psi(f^{\circ t}(x))-\Psi(x)=t,\ t\in\mathbb R,\ x\in(\mathbb R_+,0).
\end{equation}  
The \emph{continuous critical $\varepsilon$-time} $\tau_\varepsilon$ is defined as $\tau_\varepsilon>0$ such that: $$g(f^{\circ \tau_\varepsilon}(x_0))=2\varepsilon,$$ as compared with inequalities \eqref{eq:ine} for $f^{n_\varepsilon}(x_0)$. Therefore, \begin{equation}\label{eq:ana}f^{\circ \tau_\varepsilon}(x_0)=g^{-1}(2\varepsilon),\ \tau_\varepsilon=\Psi(g^{-1}(2\varepsilon))-\Psi(x_0).\end{equation}

\noindent In the sequel, we put, omitting $x_0$ in the notation for simplicity, $$x_t:=f^{\circ t}(x_0), \ t\in\mathbb R.$$

Let \begin{equation}\label{eq:epsn}\varepsilon_n:=\frac{x_{n}-x_{n+1}}{2}=\frac{g(x_n)}{2},\ n\in\mathbb N_0,\end{equation} denote the (strictly decreasing) sequence of halves of the distances between the consecutive points of the orbit $\mathcal{O}_f(x_0)$. Evidently, $\varepsilon_n\to 0$, as $n\to \infty$.

\bigskip
In the proof of Lemma~\ref{lem:prva} from Section~\ref{sec:proofA}, we need the following proposition:
\begin{proposition}\label{prop:ntau} Under the notation above,
\begin{equation}\label{eq:conn}
n_\varepsilon-\tau_\varepsilon=G(\tau_\varepsilon),\ \varepsilon\in(0,\varepsilon_0].
\end{equation}
Here, $s\mapsto G(s)$ is a 1-periodic discontinuous function on $[0,+\infty)$, explicitely given on its period by \begin{equation}\label{ge}G(s)=\begin{cases}0,&s=0,\\1-s,& s\in(0,1).\end{cases}\end{equation}
In other words,
$$
n_\varepsilon-\tau_\varepsilon=\begin{cases}1-\{\tau_\varepsilon\},&\ \varepsilon\in(0,\varepsilon_0]\setminus\{\varepsilon_n:n\in\mathbb N_0\},\\
0,&\ \varepsilon\in \{\varepsilon_n:n\in\mathbb N_0\}. \end{cases}
$$
where $\{\tau_{\varepsilon}\}=\tau_\varepsilon-\lfloor\tau_\varepsilon\rfloor.$
\end{proposition}

\begin{proof}
Let $\{\varepsilon_n:\ n\in\mathbb N_0\}$ be as in \eqref{eq:epsn}. Note that $\varepsilon\mapsto n_\varepsilon$ on $\varepsilon\in (0,\varepsilon_0]$ is, by its definition \eqref{eq:ine}, a step function, given explicitely by:
$$
n_\varepsilon=\begin{cases}0,&\ \varepsilon= \varepsilon_0,\\
k+1,&\ \varepsilon\in (\varepsilon_k,\varepsilon_{k+1}],\ k\in\mathbb N_0.
\end{cases}
$$
In the \emph{straightening} chart $s:=\tau_\varepsilon$, which depends analytically on $\varepsilon>0$ by \eqref{eq:ana}, we put
$$
\tilde n(\tau_\varepsilon):=n_\varepsilon,\ \varepsilon\in (0,\varepsilon_0].
$$
By \eqref{eq:ana}, $\tau_{\varepsilon_k}=\Psi(g^{-1}(2\varepsilon_k))-\Psi(x_0)=k,\ k\in\mathbb N_0$, since, by definition of the sequence $\{\varepsilon_k\}_k$, it follows that $g(x_k)=2\varepsilon_k$, $k\in\mathbb N_0$. Therefore, $s\mapsto \tilde n(s),\ s\in [0,+\infty)$, becomes a $1$-jump function, continuous on intervals of length exactly $1$:
$$
\tilde n(0)=0,\ \tilde n(s)=k+1,\ s\in(k,k+1].
$$
Now, let us define function $G:=\tilde n-\mathrm{id}$. It is now $1$-periodic function on $[0,+\infty)$ given by \eqref{ge}.
Thus,
$$
G(\tau_\varepsilon)=\tilde n(\tau_\varepsilon)-\tau_\varepsilon=n_\varepsilon-\tau_\varepsilon,\ \varepsilon\in(0,\varepsilon_0],
$$
and \eqref{eq:conn} follows. The other statement follows by the simple fact that, by \eqref{ge}, $G(s)=1-\{s\},\ s\in(0,+\infty),\ s\notin \mathbb N_0$, and $G(s)=0$, $s\in\mathbb N_0$.
\end{proof}
\medskip

\noindent \emph{Proof of Lemma~\ref{lem:prva}}. \
Let
$$
f(x)=x-ax^{k+1}+o(x^{k+1}),\ a>0,\ x\to 0^+.
$$
By Proposition~\ref{prop:ntau},
\begin{equation}\label{eq:put}
n_\varepsilon-\tau_\varepsilon=G(\tau_\varepsilon),
\end{equation}
where $s\mapsto G(s)$ is a bounded 1-periodic function on $[0,+\infty)$, given by the explicit formula \eqref{ge}. 

Putting \eqref{eq:put} in \eqref{eq:formula}, we get (we omit variable $\tau_\varepsilon$ of $G$ for simplicity in further computations):
\begin{align}\label{eq:formulas}
V_f(\varepsilon)&=2\varepsilon (\tau_{\varepsilon}+G)+\Psi^{-1}(\tau_\varepsilon+G+\Psi(x_0)),\nonumber \\
V_f^{\mathrm c}(\varepsilon)&=2\varepsilon \tau_{\varepsilon}+\Psi^{-1}(\tau_\varepsilon+\Psi(x_0)), \ \varepsilon>0.
\end{align}
Now we expand both expressions, as $\tau_\varepsilon\to\infty$, i.e., as $\varepsilon\to 0$. Recall from standard literature (e.g. \cite{ilyayak},\,\cite{loray}) that, for a parabolic germ of order $k+1$, the asymptotic expansion of its Fatou coordinate and then of its inverse (deduced term by term by formal binomial expansions) is given by:
\begin{align}\label{eq:psiinv}
&\Psi(x)\sim \big(\frac{1}{kx^k}+\rho\log x\big)\circ \widehat\varphi(x),\ x\to 0,\nonumber\\
&\Psi^{-1}(y)\sim \widehat\varphi^{-1}\circ \Big(k^{-\frac{1}{k}}y^{-\frac{1}{k}}\big(1+\rho k^{-2}y^{-1}\log y+\widehat R(y)\big)\Big)=\nonumber\\
&=a^{-\frac{1}{k}}k^{-\frac{1}{k}}y^{-\frac{1}{k}}+\sum_{r=2}^{k}b_r y^{-\frac{r}{k}}+\rho a^{-\frac{1}{k}}k^{-\frac{1}{k}-2}y^{-1-\frac{1}{k}}\log y+\widehat H(y),\ y\to\infty,
\end{align}
for some $b_r\in\mathbb R$, $r=2,\ldots,k$,  $\widehat\varphi(x)-a^{\frac{1}{k}}x\in x^2\mathbb R[[x]]$\footnote{Here, $\mathbb R[[x]]$ standardly denotes the set of formal positive integer power series with real coefficients, and $x^k\mathbb R[[x]]$ the subset of those with valuation at least $k$. Similarly, $\mathbb R[[x,y]]$ denotes a formal power series in two variables.}, $\widehat\varphi^{-1}(x)\in x\mathbb R[[x]]$, $\widehat R(y)\in \mathbb R[[y^{-1},y^{-1}\log y]]$, and $\widehat H(y)\in\mathbb R[[y^{-\frac{1}{k}},y^{-1}\log y]]$. Note that $\rho a^{\frac{1}{k}}k^{-\frac{1}{k}-2}y^{-1-\frac{1}{k}}\log y$ is the first term in the expansion of $\Psi^{-1}(y)$, as $y\to 0$, containing a logarithm. 

On the other hand, the asymptotic expansion of $g(x)$, as $x\to 0$, belongs to $x^{k+1}\mathbb R[[x]],$ so the expansion of $g^{-1}(x)$ belongs to $a^{-\frac{1}{k+1}}x^{\frac{1}{k+1}}+x^{\frac{2}{k+1}}\mathbb R[[x^{\frac{1}{k+1}}]].$ The coefficients $b_r,\ r=2,\ldots,k,$ and the coefficients in all $\mathbb R[[.]]$ in \eqref{eq:psiinv} depend only on the coefficients of (finite jets of) $f$, and not on the initial condition $x_0$. 

Due to asymptotic expansion of $\Psi^{-1}$ from \eqref{eq:psiinv}, to analyse the expansion of the second term in both formulas \eqref{eq:formulas} as $\tau_\varepsilon\to+\infty$, it suffices to consider expansions of terms:
\begin{align*}
&\Big(\tau_\varepsilon\big(1+(G+\Psi(x_0))\tau_\varepsilon^{-1}\big)\Big)^{-\frac{m}{k}}=\tau_\varepsilon^{-\frac{m}{k}}\Big(1+(G+\Psi(x_0))\tau_\varepsilon^{-1}\Big)^{-\frac{m}{k}},\ m\in\mathbb N,\\
&\log^n\Big(\tau_\varepsilon\big(1+(G+\Psi(x_0))\tau_\varepsilon^{-1}\big)\Big)=\Big(\log\tau_\varepsilon+\log\big(1+(G+\Psi(x_0))\tau_\varepsilon^{-1}\big)\Big)^n,\\
&\qquad\qquad\qquad\qquad=(\log\tau_\varepsilon)^n\Big(1+(\log\tau_\varepsilon)^{-1}\log\big(1+(G+\Psi(x_0))\tau_\varepsilon^{-1}\big)\Big)^n,\ n\in\mathbb N.
\end{align*}
Since $G+\Psi(x_0)$ is bounded, and $\tau_\varepsilon\to+\infty$, we proceed simply by the binomial expansion. The first term and the first logarithmic term can moreover be explicitely computed. We get that the expansion of $\Psi^{-1}(\tau_\varepsilon+G+\Psi(x_0))$ in $\tau_\varepsilon$ is as follows:
\begin{align}\label{eq:oh}
\Psi^{-1}&(\tau_\varepsilon+G+\Psi(x_0))\sim a^{-\frac{1}{k}}k^{-\frac{1}{k}}\tau_{\varepsilon}^{-\frac{1}{k}}+\sum_{m=2}^{k}H_{m}(G+\Psi(x_0))\tau_\varepsilon^{-\frac{m}{k}}+\\
&+\rho a^{-\frac{1}{k}}k^{-\frac{1}{k}-2}\tau_\varepsilon^{-1-\frac{1}{k}}\log\tau_\varepsilon+ H_{k+1}(G+\Psi(x_0))\tau_\varepsilon^{-1-\frac{1}{k}}+\nonumber\\
&+\sum_{m=k+2}^{\infty} \Big(H_m(G+\Psi(x_0))\tau_\varepsilon^{-\frac{m}{k}}+\sum_ {p=1}^{\lfloor \frac{m}{k}\rfloor} K_{m,p}(G+\Psi(x_0))\tau_\varepsilon^{-\frac{m}{k}}\log^p\tau_\varepsilon\Big),\ \tau_\varepsilon\to \infty.\nonumber
\end{align}
Here, $H_{m}(s)$ are \emph{polynomials} of degree at most $\lfloor \frac{m-1}{k}\rfloor$, and $K_{m,p}(s)$, $p=1,\ldots,\lfloor\frac{m}{k}\rfloor$, polynomials of degree at most $\lfloor\frac{m}{k}\rfloor-1$.
 
Now we expand $2\varepsilon(\tau_\varepsilon+G)$ from \eqref{eq:formulas} in power-logarithmic scale in $\tau_\varepsilon$. Indeed,
$2\varepsilon=g\big(\Psi^{-1}(\tau_\varepsilon+\Psi(x_0))\big)$, where $g=\mathrm{id}-f=ax^{k+1}+o(x^{k+1})\in\mathbb R[[x]]$. Therefore, from \eqref{eq:oh}, we get:
\begin{align}\label{eq:new}
2\varepsilon(\tau_\varepsilon&+G)=g\big(\Psi^{-1}(\tau_\varepsilon+\Psi(x_0))\big)\cdot(\tau_\varepsilon+G)=\\
&=a^{-\frac{1}{k}}k^{-1-\frac{1}{k}}\tau_\varepsilon^{-\frac{1}{k}}+\sum_{m=2}^{k}c_m(x_0)\tau_\varepsilon^{-\frac{m}{k}}+\rho (k+1)k^{-3-\frac{1}{k}}a^{-\frac{1}{k}}\tau_\varepsilon^{-1-\frac{1}{k}}\log\tau_\varepsilon+\nonumber\\
&+(G a^{-\frac{1}{k}}k^{-1-\frac{1}{k}}+c_{k+1}(x_0))\tau_\varepsilon^{-1-\frac{1}{k}}+\nonumber\\
&+\sum_{m=k+2}^{\infty}\,\sum_ {p=0}^{\lfloor \frac{m}{k}\rfloor+1} \big(q_{m,p}(x_0)+G\cdot r_{m,p}(x_0)\big)\tau_\varepsilon^{-\frac{m}{k}}\log^p\tau_\varepsilon,\ \tau_\varepsilon\to +\infty.\nonumber
\end{align}
 Here, $c_m(x_0)\in\mathbb R,\ m\in\{2,\ldots,k+1\},$ and $q_{m,p}(x_0),\,r_{m,p}(x_0)\in\mathbb R,\ m\geq k+2,\ 1\leq p\leq \lfloor \frac{m}{k} \rfloor+1,$ denote the coefficients depending on $x_0$.
Putting \eqref{eq:oh} and \eqref{eq:new} in both expansions in \eqref{eq:formulas}, we get:
\begin{align}
V_f(\varepsilon)\sim &a^{-\frac{1}{k}}k^{-\frac{1}{k}}\frac{k+1}{k}\tau_{\varepsilon}^{-\frac{1}{k}}+\sum_{m=2}^{k}\tilde H_{m}(G)\tau_\varepsilon^{-\frac{m}{k}}+\nonumber\\
&+\rho a^{-\frac{1}{k}}k^{-\frac{1}{k}-2}\frac{k+1}{k}\tau_\varepsilon^{-1-\frac{1}{k}}\log\tau_\varepsilon+ \tilde H_{k+1}(G)\tau_\varepsilon^{-1-\frac{1}{k}}+\nonumber\\
&+\sum_{m=k+2}^{\infty} \Big(\tilde H_m(G)\tau_\varepsilon^{-\frac{m}{k}}+\sum_ {p=1}^{\lfloor \frac{m}{k}\rfloor+1} \tilde K_{m,p}(G)\tau_\varepsilon^{-\frac{m}{k}}\log^p\tau_\varepsilon\Big),\label{eq:epsok} \\
V_f^{\mathrm c}(\varepsilon)\sim &a^{\frac{1}{k}}k^{-\frac{1}{k}}\frac{k+1}{k}\tau_{\varepsilon}^{-\frac{1}{k}}+\sum_{m=2}^{k}\tilde H_{m}(0)\tau_\varepsilon^{-\frac{m}{k}}+\nonumber\\
&+\rho a^{\frac{1}{k}}k^{-\frac{1}{k}-2}\frac{k+1}{k}\tau_\varepsilon^{-1-\frac{1}{k}}\log\tau_\varepsilon+ \tilde H_{k+1}(0)\tau_\varepsilon^{-1-\frac{1}{k}}+\nonumber\\
&+\sum_{m=k+2}^{\infty} \Big(\tilde H_m(0)\tau_\varepsilon^{-\frac{m}{k}}+\sum_ {p=1}^{\lfloor \frac{m}{k}\rfloor+1} \tilde K_{m,p}(0)\tau_\varepsilon^{-\frac{m}{k}}\log^p\tau_\varepsilon\Big), \ \tau_\varepsilon\to+\infty.\label{eq:cepsok}
\end{align}

Here, $\tilde H_m(s)$ resp. $\tilde K_{m,p}(s)$ denote polynomials of degree at most $\lfloor \frac{m-1}{k}\rfloor$ resp. at most $\lfloor \frac{m}{k}\rfloor-1$, with coefficients \emph{depending on $x_0$}.
Therefore, negative powers of $\tau_\varepsilon$ are multiplied in $V_f(\varepsilon)$ by oscillatory (bounded, $1$-periodic) functions in $\tau_\varepsilon$, realized as polynomials in $1$-periodic function $G(\tau_\varepsilon)$. On the other hand, by formulas \eqref{eq:formulas}, we get the expansion \eqref{eq:cepsok} of $V_f^{\mathrm c}(\varepsilon)$ by putting $0$ instead of $G(\tau_\varepsilon)$ in \eqref{eq:epsok}. 

All polynomials $\tilde H_m$ up to $m=k$ are just constants, that is, of degree $0$, and are equal to the corresponding terms in $V_f^{\mathrm c}(\varepsilon)$. Moreover, using \eqref{taueps} and the fact that the terms of the asymptotic expansion of $\varepsilon\mapsto V_f(\varepsilon)$ up to the residual term do not depend on the initial point $x_0$, see \cite{formal}, it can be seen that these constants %and the coefficients of polynomials $\tilde H_m$, $\tilde K_{m,p}$, for $m\geq k+2$ 
do not depend on the initial condition $x_0$, but only on coefficients of $f$. We denote them simply by $c_m\in\mathbb R$, $m=2,\ldots,k$. By \cite[Theorem B]{MRRZ3}, $V_f(\varepsilon)-V_f^{\mathrm c}(\varepsilon)=O(\varepsilon^{\frac{2k+1}{k+1}})$, as $\varepsilon\to 0$. Thus, since $V_f^{\mathrm c}(\varepsilon)$ does not contain oscillatory terms, $V_f(\varepsilon)$ also does not have oscillatory coefficients up to the order $O(\varepsilon^{\frac{2k+1}{k+1}})$. Since $\tau_\varepsilon^{-\frac{2k+1}{k}}\sim \varepsilon^{\frac{2k+1}{k+1}}$, we conclude that $\tilde H_{m}(s)$, $m=k+1,\ldots,2k,$ and $\tilde K_{m,p}(s)$, $m=k+1,\ldots,2k+1,$ $p=1,\ldots,\lfloor\frac{m}{k}\rfloor+1,$ are also constant polynomials, that we denote by $c_{m}(x_0)$ resp. $d_{m,p}(x_0)$. Here, $x_0$ denotes possible dependence on the initial point. Thus, from \eqref{eq:epsok} and \eqref{eq:cepsok}, we get \eqref{eq:ok}.
\hfill $\Box$
\bigskip

\begin{proof}[Proof of Proposition~\ref{lem:integrali}]
The proof follows directly from  the fact that distributional asymptotics behaves well under differentiation. 
Namely, let $T\in\mathcal{S}'(0,\delta)$ be a Schwartz distribution such that $T=O(x^{\beta})$ for some $\beta$, as $x\to 0$, then its distributional derivative $T'$ is $O(x^{\beta-1})$, as $x\to 0$.
Indeed, take $0<a<1$ and any test function $\varphi\in\mathcal{S}(0,\delta)$ (extended by $0$ to $(0,+\infty)$), then
\begin{equation*}
	\langle T',\varphi_a\rangle=-\langle T,(\varphi_a)'\rangle=-\frac{1}{a}\langle T,(\varphi')_a\rangle=\frac{1}{a}O(a^{\beta})=O(a^{\beta-1}),\ a\to 0.
\end{equation*}
The rest of the proof follows by induction, i.e., we apply the distributional derivative to the function
\begin{equation*}           
F^{[k]}(t)-\left(\sum_{p_1=0}^{n_1} c_{1,p_1}t^{\alpha_1}\log^{p_1}t+\ldots+\sum_{p_m=0}^{n_m}c_{m,p_m} t^{\alpha_m}\log^{p_m}t\right)
\end{equation*}
$k$ times and obtain that it is of distributional order $O(t^{\alpha_{m+1}-k})$, as $t\to 0$.
\end{proof} 

\begin{lemma}[Integration and growth]\label{lem:growth} 
Let $F:[0,+\infty)\to \mathbb R$ be integrally bounded $1$-periodic function. Let $f,\, g:(0,+\infty)\to \mathbb R$ be analytic, such that:
\begin{enumerate} 
\item $f'(s)\sim s^{\alpha}$,  $s\to 0^+,$
\item $g(s)\sim s^{\beta}\log^p s$ and $g'(s)\sim (s^{\beta}\log^p s)'$, $s\to 0^+$, $p\in\mathbb N_0$,
\item $\alpha,\,\beta\in\mathbb R$ such that $\beta-\alpha>0$.
\end{enumerate}
Then:
$$
\int_0^t F(f(s))\cdot g(s)\,ds=h(t)\cdot H(f(t))-\int_0^t H(f(s))\cdot h'(s)\,ds ,\ t>0.
$$
where $H:[0,+\infty)\to \mathbb R$ is again integrally bounded $1$-periodic, and $h:(0,+\infty)\to \mathbb R$ analytic such that
$$h(t)\sim t^{\beta-\alpha}\log^p t,\ t\to 0^+.$$
Moreover, due to boundedness of $H$,
$$
\int_0^t F(f(s))\cdot g(s)\,ds=O(t^{\beta-\alpha-\delta}),\ t\to 0^+,
$$
for every $\delta>0$.
\end{lemma}
\begin{remark} Note that, due to boundedness of $F$, we can immediately bound by integration:
$$
\int_0^t F(f(s))\cdot g(s)\,ds=O(t^{\beta+1-\delta}),\ t\to 0,\ \delta>0.
$$
The importance in the statement of Lemma~\ref{lem:growth} is that, if $\alpha<-1$, we get even \emph{higher order} $O(t^{\beta+|\alpha|-\delta})$, $\beta+|\alpha|>\beta+1$.
\end{remark}

\begin{proof} Directly by partial integration. We write:
$$
\int_0^t F(f(s))\cdot g(s)\,ds=\int_0^t F(f(s))f'(s)\cdot \frac{g(s)}{f'(s)}\,ds,
$$
and take $dv=F(f(s))f'(s)\,ds$ and $u=\frac{g(s)}{f'(s)}$ in partial integration. Let $F_1:[0,\infty)\to \mathbb R $ denote the primitive of $F$, $dF_1=F$, which is bounded and $1$-periodic by Definition~\ref{def:IBP}. We have :
\begin{equation}\label{eq:pom}
\int_0^t F(f(s))\cdot g(s)\,ds=u(t)\cdot F_1(f(t))-\int_0^t F_1(f(s))\cdot u'(s)\,ds ,\ t>0,
\end{equation}
where $u(t)\sim t^{\beta-\alpha}\log^p t$, $t\to 0^+$, and $u(0)=0$. We put $h:=u$. 
Finally, by integral normalization of $F_1$, putting $H:=F_1-\int_0^1 F_1(s)\,ds$ (see Proposition~\ref{prop:intnorm}), we get $H$ integrally bounded $1$-periodic. Due to cancellations in \eqref{eq:pom}, \eqref{eq:pom} holds also with $H$ instead of $F_1$. 
\end{proof}

As a direct consequence of Lemma~\ref{lem:growth}, we get the following Corollary~\ref{cor:h} that is directly applicable to proving \eqref{eq:jedan} in the proof of Lemma~\ref{lema:druga}. To prove \eqref{eq:jedan}, we re-iterate Corollary~\ref{cor:h} as many times as needed.
\begin{corollary}\label{cor:h} Let $f(s):=\tau_s\sim s^{-\frac{k}{k+1}}$, $g(s):=\tau_s^{-\frac{m}{k}}\log^p\tau_s\sim s^{\frac{m}{k+1}}\log^p s$, $s\to 0^+$, $p\in\mathbb N_0$, as in the case of critical continuous time for parabolic orbits of order $k\geq 1$. Let $F:=\tilde P_m\circ G$ or $F:=\tilde R_{m,p}\circ G$, as in \eqref{eq:s}. Then:
 \begin{align*}
\int_0^\varepsilon \tilde R_{m,p}(G(\tau_s)) \tau_s^{-\frac{m}{k}}\log^p\tau_s\ ds= h(\varepsilon) \cdot \tilde Q_m(G(\tau_\varepsilon))-&\int_0^\varepsilon \tilde Q_m\big(G(\tau_s)\big) h'(s)\,ds=\\
&=O(\varepsilon^{\frac{m}{k+1}+1+\frac{k}{k+1}-\delta}),\ \varepsilon\to 0^+,\end{align*}
for every $\delta>0$, where $\tilde Q_m\circ G$ is integrally bounded $1$-periodic and $h:(0,\infty)\to\mathbb R$ analytic, $$h(s)\sim s^{\frac{m}{k+1}+1+\frac{k}{k+1}}\log^p s, \ s\to 0^+.$$ Moreover, each of the 'new' functions $h(s) \cdot \tilde Q_m(G(\tau_s))$ and $h'(s) \cdot \tilde Q_m(G(\tau_s))$ on the right-hand side again satisfy the assumptions of Lemma~\ref{lem:growth}, with $F:=\tilde Q_m\circ G$, $f(s):=\tau_s$ and $g(s):=h(s)$ i.e. $h'(s)$.
\end{corollary}

\medskip

\subsection{Proofs from Section~\ref{sec:general}.}\
\smallskip

\noindent \emph{Proof of Lemma~\ref{lemma_tilde}}.
	To get \eqref{eq_tilde}, we consecutively integrate $m$ times (by parts) the functional equation \eqref{tube_zeta_def} and show that \eqref{eq_tilde} is valid for all $s\in\Ce$ such that $\re s>\overline{\dim}_B A$.
	The key part is the observation that, for $s$ such that $\re s>\overline{D}:=\overline{\dim}_B A$, the integral $\int_{0}^{\delta}t^{s-N-1-m}V_{A}^{[m]}(t)\di t$  is absolutely convergent and therefore defines a holomorphic function in the open half-plane $\{\re s >\overline{D}\}$.
	Indeed, let $r>0$ be small enough such that $\re s>\overline{D}+r$.
	Then the upper Minkowski content $\mathcal{M}^{*(\overline{D}+r)}(A)$ equals $0$ and hence, there exists a positive constant $C_{\delta}$ such that $V_A(t)\leq C_{\delta}t^{N-(\overline{D}+r)}$ for all $t\in(0,\delta]$.
	We then have\footnote{Observe that $N-(\overline{D}+r)+1>0$ since we always have $\overline{D}\leq N$.}
	$$
	V_{A}^{[1]}(t)\leq\int_0^t V_A(\tau)\di\tau\leq\frac{C_{\delta}}{N-(\overline{D}+r)+1}t^{N-(\overline{D}+r)+1},
	$$
	for all $t\in(0,\delta]$. By induction, we get that, for all $m\in\eN$, there exists a constant $C_{\delta,m}>0$, such that: $V_{A}^{[m]}(t)\leq C_{\delta,m}t^{N-(\overline{D}+r)+m}$, $t\in(0,\delta]$.
	
	This implies that the integral $\int_{0}^{\delta}t^{s-N-1-m}V_{A}^{[m]}(t)\di t$ is bounded in absolute value by:
	$$
	\int_{0}^{\delta}t^{\re s-N-1-m}\cdot V_{A}^{[m]}(t)\di t\leq C_{\delta,m}\int_0^{\delta}t^{\re s-(\overline{D}+r)-1}\di t.
	$$
	The last integral above is convergent, since $\re s>\overline{D}+r$. Since $r>0$ is arbitrary, by the principle of analytic continuation, we conclude that \eqref{eq_tilde} is valid on $\{\re s>\overline D\}$. 
\qed

\section*{Declarations}

\begin{itemize}
\item Funding: The research of Goran Radunovi\'c and Maja Resman was supported by the Croatian Science Foundation under the grant UIP-2017-05-1020. The research of all authors was partially supported by Croatian Science Foundation (HRZZ) grant PZS-2019-02-3055 from
Research Cooperability funded by the European Social Fund. The research of Pavao Marde\v si\' c was partially supported by  
EIPHI Graduate School (contract ANR-17-EURE-0002).

\item Conflict of interest/Competing interests: The authors have no competing interests to declare that are relevant to the content of this article.
\item Ethics approval: Not applicable
\item Consent to participate: Not applicable
\item Consent for publication: All authors give consent to publish the article.
\item Availability of data and materials: Not applicable
\item Code availability: Not applicable
\item Authors' contributions: All authors contributed equally to this work.
\item Data availability statement: Data sharing not applicable to this article as no datasets were generated or analysed during the current study.
\end{itemize}

\end{document}